\documentclass[a4paper,11pt,twoside]{article}

\usepackage[nodayofweek]{datetime}

\usepackage[utf8]{inputenc}
\usepackage[T1]{fontenc}
\usepackage[english]{babel}

\usepackage{charter}

\usepackage{xcolor}
\usepackage{verbatim}

\usepackage{hyperref}

\usepackage[top=3.cm, bottom=4.0cm, left=2.2cm, right=2.2cm]{geometry}
\usepackage{amsmath}
\usepackage{amsthm}	%package pour les styles des ?nonc?s
\usepackage{amsfonts}	%pourles polices mathematiques vraisemblablement
\usepackage{amssymb,bbm}
\usepackage{shuffle}
\usepackage{mathrsfs}
\usepackage{pifont}% http://ctan.org/pkg/pifont

\usepackage{enumerate}
\usepackage{enumitem}
\usepackage[nobysame, alphabetic,initials]{amsrefs}

\usepackage{stmaryrd}
\usepackage{appendix}

\usepackage{graphicx}

\usepackage{matlab-prettifier}

\usepackage{%soul
						ulem		% use \sout{...}
} 

\numberwithin{equation}{section}
\theoremstyle{plain}
\newtheorem{theorem}{Theorem}[section]
\newtheorem{lemma}[theorem]{Lemma}
\newtheorem{proposition}[theorem]{Proposition}

\newtheorem{definition}[theorem]{Definition}

\newtheorem{remark}[theorem]{Remark}
\newtheorem{example}[theorem]{Example}
\newtheorem{assumption}[theorem]{Assumption}

\allowdisplaybreaks

\newcommand{\bD}{\mathbb{D}}
\newcommand{\bE}{\mathbb{E}}
\newcommand{\bF}{\mathbb{F}}

\newcommand{\bH}{\mathbb{H}}

\newcommand{\bN}{\mathbb{N}}
\newcommand{\bP}{\mathbb{P}}

\newcommand{\bR}{\mathbb{R}}

\newcommand{\bW}{\mathbb{W}}

\newcommand{\bZ}{\mathbb{Z}}

\newcommand{\cB}{\mathcal{B}}
\newcommand{\cC}{\mathcal{C}}

\newcommand{\cE}{\mathcal{E}}
\newcommand{\cF}{\mathcal{F}}
\newcommand{\cG}{\mathcal{G}}
\newcommand{\cH}{\mathcal{H}}

\newcommand{\cK}{\mathcal{K}}

\newcommand{\cP}{\mathcal{P}}

\newcommand{\cS}{\mathcal{S}}

\newcommand{\cV}{\mathcal{V}}

\newcommand{\cX}{\mathcal{X}}

\newcommand{\cZ}{\mathcal{Z}}

\newcommand{\fP}{\mathfrak{P}}

\newcommand{\scI}{\mathbf{I}}
\newcommand{\scJ}{\mathbf{J}}
\newcommand{\scK}{\mathbf{K}}

\newcommand{\rA}{\mathbf{A}}

\newcommand{\rQ}{\mathbf{Q}}
\newcommand{\rR}{\mathbf{R}}

\newcommand{\rS}{\mathbf{S}}

\newcommand{\rZ}{\mathbf{Z}}

\newcommand{\RKHS}{\cH}
\newcommand{\fWIC}{\widehat{\cH}}
\newcommand{\clq}[2]{\mathrm{cl}_{#2}(#1)}

\newcommand{\BSi}{\mathbf{i}}

\DeclareMathOperator{\spn}{span}
\DeclareMathOperator{\lin}{Lin}

\DeclareMathOperator{\Toeplitz}{Toeplitz}
\newcommand{\oSlash}{{\mbox{\o}}}

\newcommand{\1}{\mathbbm{1}}

\usepackage{pgf,tikz}
\usetikzlibrary{arrows}
\usetikzlibrary{graphs,graphs.standard,quotes}

\tikzstyle{vertex} = [fill, shape=circle,inner sep=2pt,]
\tikzstyle{edge} = [fill, line width = 0.5pt]

\newcommand{\gmu}{\nu} 

\newcommand{\glambda}{\theta}

\usepackage{enumerate}

\providetoggle{ArXiv}
\settoggle{ArXiv}{false}

\title{The fundamental martingale with applications to \\ Markov Random Fields}
\author{
	\normalsize Kevin Hu \\[8pt]
	\small kevin\_hu@brown.edu
	\and
	\normalsize Kavita Ramanan \\[8pt]
	\small kavita\_ramanan@brown.edu
	\and
	\normalsize William Salkeld \\[8pt]
	\small william\_salkeld@brown.edu
}

\begin{document}

	\maketitle

	\setcounter{tocdepth}{1}
	\tableofcontents
    
    \newpage
    \section{Introduction}

    The theory of random fields finds its origin in the combination of two deep schools of thought. Firstly, from the kinetic theory of matter it is well understood that large scale, long range dependencies arise from the interactions of individuals with only their neighbours. On the other hand, probabilistic reasoning would conclude that each of these local interactions should be probabilistic in nature. The upshot of these ideas is that many problems in the modern world can be modelled by the joint distribution for some large collection of random variables that exhibit correlation with one another derived from some underlying graph structure. The study of Markov random fields starts with the Ising model of ferromagnetism from statistical mechanics \cite{Baxter1989Exactly}, but there there is also interest in neurophysiology due to \emph{Hopfield networks} for modelling addressable memory \cite{Hopfield1982Neural}. 

    Our focus of interest is on large collections of stochastic differential equations
    \begin{equation*}
        dX_t^u = b_u\Big( t, X^u[t], X^{N_u}[t] \Big) dt + dZ_t^u, \quad u \in V
    \end{equation*}
    where $V$ represents an index set, $N_u \subseteq V$ and $(Z^u)$ is a collection of independent additive noises. Such stochastic processes arise in statistical physics \cites{Dereudre2003Interacting, Redig2010Short}, mathematical finance \cite{Nadtochiy2020Mean} and oscillator synchronisation \cite{Medveded2019Continuum}. 

    Previous works studying the Markov random field properties for such locally interacting diffusions start with \cite{Deuschel1987Infinite} and \cite{Cattiaux1996Une} (which uses a variation of the Malliavin integration by parts formula), continuing with \cite{Dereudre2005Propagation} (which uses a truncation argument and convergence of the associated Hamiltonian) and later \cite{Roelly2014Propagation} which considers stochastic differential equations with a confining potential and a separate locally interacting drift term. 

    More recently, \cite{lacker2020Locally} establishes a 2-Markov random field property for a much broader class of stochastic differential equations by allowing for many different drift terms $b_u$ using a Girsanov theorem. This work was motivated by \cite{lacker2020marginal} which showed how a Global 2-Markov Random Field property is necessary for the derivation of an autonomous description of the dynamics of some small neighbourhood of locally interacting SDEs as the marginal of some much larger collection. This autonomous description takes the form of a singular McKean-Vlasov equation and this result has powerful implications for the study of these infinite systems of locally interacting SDEs.  
    
    Our interest arose from studying the local equation in the context where the Brownian motion was replaced by a Gaussian noise. We found that a necessary property for the Gaussian noises is that the \emph{reproducing kernel Hilbert space} needs to be appropriately large so that one can use a Girsanov Theorem applied to the drift term to describe the changes of measure that are central to proving the Markov Random Field property. 

    A Gaussian noise that proved to be a very effective choice was \emph{fractional Brownian motion} since the reproducing kernel Hilbert space is already well documented in \cite{Decreusfond1999Stochastic} and its failure of the semi-martingale property made it a compelling choice. While reviewing the literature, we came across the concept of the \emph{fundamental martingale} and realised this was central to our methods: a Gaussian process is said to have a fundamental martingale if there is a Volterra kernel that transforms the Gaussian process to a martingale (under the filtration generated by the Gaussian process). In the case of fractional Brownian motion, the Volterra kernel is already known (see Equation \eqref{eq:fbm_Volterra} below) and we consider some Gaussian processes whole Volterra kernel has a similar \emph{Sonine pair}. 

    More specifically, we found that the property that fractional Brownian motion satisfies that we require in this work is that the collection of reproducing kernel Hilbert spaces $(\RKHS_t)_{t\in [0,T]}$ forms a \emph{Hilbert space filtration} that is isomorphic to the Hilbert space filtration of reproducing kernels for Brownian motion. We refer to this property as \emph{securely locally non-deterministic} (see Definition \ref{definition:SLND} below). 

    We also emphasise that secure local non-determinism and the existence of the fundamental martingale also leads to many other vital properties, including entropy estimates and a Mimicking Theorem with important implications for distribution dependent dynamics that we explore in future works.

    \subsection*{Basic definitions}

    Let $V$ be a finite or countably infinite set and let $E \subseteq \big\{ \{u, v\}: u, v \in V \big\}$. Then we say that $(V, E)$ is graph. For every $u \in V$, we denote the neighbourhood $N_u = \big\{v \in V: \{u, v\} \in E \big\}$. We say that a graph $(V, E)$ is locally finite if for every $u \in V$ the set $|N_u|< \infty$. We say that two graphs $G_1=(V_1, E_1)$ and $G_2=(V_2, E_2)$ are isomorphic if there exists a bijection $\phi: V_1 \to V_2$ such that $E_2 = \big\{ \{\phi[u], \phi[v]\}: \{u, v\} \in E_1 \big\}$ and we define the set of locally finite graph isomorphism classes by $\cG$. 

    For a subset $A \subseteq V$, we define the first and second order boundary sets
    \begin{align*}
        \partial A:=& \big\{ u \in V: \exists v \in A \mbox{ such that } \{u, v\} \in E \big\},
        \\
        \partial^2 A:=& \partial A \cup \partial \big( A \cup \partial A \big). 
    \end{align*}
    \begin{definition}
		A \emph{clique} in a graph $G = (V, E)$ is a complete subgraph of $G$, a subset $A \subseteq V$ such that $(u, v) \in E$ for every $u, v \in A$. Equivalently, a clique is a set $A \subset V$ of diameter at most 1. We define $\clq{G}{1}$ to be the set of all cliques of the graph $G$. 
		
		Similarly, we say that any subset $A \subset V$ with diameter at most 2 is a \emph{2-clique} of the graph $G$ and let $\clq{G}{2}$ denote the set of 2-cliques of $G$. 
	\end{definition}
    
    For a normed vector space $\big( \cX, \| \cdot\| \big)$ we denote by $\cB(\cX)$ the Borel sets of $\cX$. For a finite index set $V$ and a collection of normed vector spaces $\big( \cX^u, \| \cdot \|_u \big)_{u\in V}$, we denote $\cX^V = \oplus_{u\in V} \cX^u$ with norm $\| \cdot \|_V:\cX^V \to \bR$ by 
    \begin{equation*}
        \big\| (x^u)_{u\in V} \big\|_V = \sum_{u\in V} \| x_u \|_u. 
    \end{equation*}
    On the other hand, when $V$ is countably infinite index set and $(\cX^u, \| \cdot \|_u )_{u\in V}$ is a collection of normed vector spaces, we denote the locally convex topological vector space 
    \begin{equation*}
        \cX^V = \bigoplus_{u\in V} \cX^u 
        \quad \mbox{with collection of seminorms} \quad
        f_u\big( x^V \big) = \| x^u \|_u. 
    \end{equation*}
    For a locally convex topological vector space $\cX$, we denote by $\cX^*$ the collection of continuous linear functionals and $\cB'(\cX)$ to be cylindrical $\sigma$-algebra, the minimal $\sigma$-algebra with respect to which all continuous linear functionals are measurable. 

    For a measure $\mu \in \cP\big( \cX^V \big)$ and a subset $U \subseteq V$, we denote $\mu^U$ to be the pushforward of the measure $\mu$ under the canonical projection from $\cX^V$ onto $\cX^U$.     
    
    \begin{definition}
        \label{def-MRFs}
        Let $(\cX, d)$ be a metric space and let $G=(V, E)$ be a locally finite graph. Let $Y^V:=(Y^v)_{v \in V}$ be a collection of random variables indexed by $V$ with probability distribution $\mu \in \cP(\cX^V)$.
        
        We say that $(Y^v)_{v \in V}$ (or equivalently $\mu$) is a \emph{first-order local Markov random field} (abbreviated as 1-MRF) on $\cX^V$ if for every finite set $A \subset V$, the collection of random variables $Y^A=(Y^v)_{v\in A}$ is conditionally independent of $Y^{(A \cup \partial A)^c}=(Y^v)_{v\in (A \cup \partial A)^c}$ given $Y^{\partial A}$. 
        
        We say that $(Y^v)_{v \in V}$ (or equivalently $\mu$) is a \emph{second-order local Markov random field} (abbreviated as 2-MRF) on $\cX^{V}$ if for every finite set $A \subset V$ the collection of random variables $Y^A$ is conditionally independent of $Y^{(A \cup \partial^2 A)^c}$ given $Y^{\partial^2 A}$.

        A collection of $\cX$-valued random elements $(Y_v)_{v \in V}$ is said to form a \emph{second-order global Markov Random Field} if for any sets $A \subset V$, $B \subset V \setminus (A \cup \partial^2 A)$, the collection of random variables $Y^A$ is conditionally independent of $Y^{B}$ given $Y^{\partial^2 A}$.
    \end{definition}
    When the space $\cX^V$ is clear from the context, we will simply say that $(\cX^v)_{v \in V}$, or equivalently its law $\mu$, is a 1MRF \ or 2MRF.

    \subsection*{Summary}

    In Section \ref{section:MainresultsLoc}, we provide an accessible summary of the results we prove in this paper. These results are specific to \emph{fractional Brownian motion}, which is the canonical example of the Gaussian processes that we use as additive noise. 

    In Section \ref{section:Gaussian}, we study the existence and properties of the \emph{fundamental martingale}. In Section \ref{subsection:HeuristicFundMart}, we provide a discrete time justification for the existence of a Volterra kernel that transforms Gaussian processes into a martingale and show that the existence of a fundamental martingale is equivalent to the Gaussian process satisfying a \emph{local non-determinism} condition. In Section \ref{subsection:VolterraProcess}, we consider continuous time Gaussian Volterra  processes and demonstrate under an appropriate assumption (see Assumption \ref{assumption:VolterraK}) that allows us to bijectively transform this Volterra process into a Brownian motion, and further that we can also bijectively transform stochastic processes that take their value on the reproducing kernel Hilbert space of the Volterra process to stochastic processes that take their values on the reproducing kernel Hilbert space of Brownian motion. In Section \ref{subsection:SLND+fm}, we introduce the concept of ``\emph{secure local non-determinism}'' and show that this is equivalent to Assumption \ref{assumption:VolterraK}. Finally, in Section \ref{subsection:Girsanov} we prove a Girsanov Theorem (see Theorem \ref{theorem:ap:girsanov}) for Gaussian Volterra processes that are securely locally non-deterministic. 

    In Section \ref{section:2MRF}, we consider collections of stochastic differential equations that interact locally with one another through the drift term based on the underlying structure of a graph. In Section \ref{subsection:WeakExist+Uniq}, we prove weak existence and uniqueness using Theorem \ref{theorem:ap:girsanov}. This result also includes the case where the graph $(V, E)$ has countably infinite vertex set which we believe is a new result even when the collection of Gaussian Volterra processes are taken to be Brownian motions. In Section \ref{subsection:MRF}, we prove that such collections of locally interacting stochastic differential equations form a \emph{Markov random field}. The proof when the graph is finite follows from the well known $2^{nd}$-\emph{Hammersey-Clifford} result, but when the vertex set is taken to be countably infinite this becomes more challenging and we have to adapt a truncation argument (the full details of which can be found in Appendix \ref{section:appendix}. 
    
    \subsection*{Notation}

    Let $\cC_{T}^d = C([0,T]; \bR^d)$ be the vector space of continuous functions on the interval $[0,T]$ taking their values in $\bR^d$ paired with the supremum norm $\| x\|_{\infty, t} = |x_0| + \sup_{s \in [0,t]} |x_s|$. Similarly, let $\cC_{0, T}^d$ be the subspace of continuous functions with $x_0 = 0$. Using the canonical decomposition that for any $x\in C_T^d$, we can equivalently write $x \equiv (x_0, x-x_0) \in \bR^d \times \cC_{0, T}^d$. 

    \subsubsection*{Measure theory}
    For $p \geq 1$, $(\cX, d)$ a metric space with Borel $\sigma$-algebra $\cB(\cX)$ and two probability measures $\mu, \nu \in \cP(\cX)$ we denote the $p$-\emph{Wasserstein distance} by
    \begin{equation*}
        \bW_d^{(p)} \big[ \mu, \nu \big] = \left( \inf_{\pi \in \Pi(\mu, \nu)} \int_{\cX \times \cX} d(x, y)^p d\pi(x, y) \right)^{\tfrac{1}{p}}
    \end{equation*}
    where $\Pi(\mu, \nu)$ is the set of all measures on $\Big( \cX \times \cX, \sigma \big( \cB(\cX) \times \cB(\cX) \big) \Big)$ with marginals $\mu$ and $\nu$. For a measure $\mu \in \cP(\cX)$ and $f:\cX \to \bR$ we denote
    \begin{equation*}
        \big\langle \mu, f \big\rangle = \int_{\cX} f(x) d\mu(x). 
    \end{equation*}

    For two measures $\mu, \nu \in \cP(\cX)$ we use $\nu << \mu$ to denote that $\nu$ is absolutely continuous with respect to $\mu$. We denote the \emph{relative entropy functional} by
    \begin{equation}
        \label{eq:RelativeEntropy}
        \bH\big[ \mu\big| \nu \big] = \left\{
        \begin{aligned}
            &\int_\cX \log\bigg( \frac{ d\mu}{d\nu}(x) \bigg) d\mu(x) 
            &\quad& \mbox{if}\quad
            \log\Big( \tfrac{ d\mu}{d\nu} \Big) \in L^1\big( \cX, \mu; \bR \big), 
            \\
            &\infty &\quad& \mbox{otherwise. }
        \end{aligned}
        \right.
    \end{equation}
    Sometimes, $\bH\big[ \mu \big| \nu \big]$ is referred to  as the \emph{Kullback–Leibler divergence}. 

%    \newpage
    \section{Main results}
    \label{section:MainresultsLoc}

    The main contributions of this paper are the extension of Girsanov's Theorem for a collection of Gaussian processes and an application of this to proving that collections of locally interacting equations form a 2-Markov Random Field. We also develop many of the concepts established in this paper to prove a \emph{Mimicking Theorem} for the study of McKean-Vlasov equations in the sequel \cite{Hu2023Fundamental}

    \subsection{Fractional Brownian motion}

    The focus of this work is for stochastic differential equations driven by an additive Gaussian noise, the most intuitive of these being \emph{fractional Brownian motion}:
    \begin{definition}
        \label{definition:fBm}
        A one dimensional \emph{fractional Brownian motion} $(Z_t)_{t\in [0,T]}$ with Hurst parameter $H\in (0, 1)$ is a centered Gaussian process with covariance defined for $t, s\in [0,T]$ by
        \begin{equation}
            \label{eq:covariance-fbm}
            \bE\Big[ Z_t Z_s \Big] = R(t, s) = \tfrac{1}{2} \Big( |t|^{2H} + |s|^{2H} - |t-s|^{2H} \Big). 
        \end{equation}
        A $d$-dimensional fractional Brownian motion is a centered Gaussian process with covariance
        \begin{equation*}
            \bE\Big[ \big\langle Z_t , Z_s \big\rangle_{\bR^d} \Big] = d \cdot R(t, s). 
        \end{equation*}
        Where notation can be reduced, we will not write the identity matrix.  
    \end{definition}
    A fractional Brownian motion is a continuous time stochastic process that satisfies $Z_0 = 0$ and is the unique centered Gaussian process that is self-similar and has stationary increments:
    \begin{equation*}
        Z_{\lambda t} \sim |\lambda|^H \cdot Z_t 
        \quad \mbox{and}\quad
        Z_t - Z_s \sim Z_{t-s}. 
    \end{equation*}

    The case $H=1/2$ corresponds to that of Brownian motion, but otherwise $Z$ is neither a Markov process nor a semimartingale. However, much like Brownian motion the process is $\bP$-almost surely $\alpha$-H\"older continuous for $\alpha < H$.

    A fractional Brownian motion can be written as a Gaussian Volterra process 
    \begin{equation*}
        Z_t = \int_0^t K(t, s) dW_s 
    \end{equation*}
    where $K(t, s)=0$ for $t<s$ and for $s<t$
    \begin{equation}
        \label{eq:fBmKernel}
        K(t, s) = \left\{
        \begin{aligned}
            &c_H \cdot s^{\tfrac{1}{2}-H} \int_s^t (u-s)^{H - \tfrac{3}{2}} u^{H-\tfrac{1}{2}} du
            &\quad
            H>\tfrac{1}{2},
            \\
            &c_H\bigg( \Big( \frac{t(t-s)}{s} \Big)^{H-\tfrac{1}{2}} - \big( H- \tfrac{1}{2} \big) \cdot s^{\tfrac{1}{2}-H} \int_s^t u^{H-\tfrac{3}{2}} (u-s)^{H-\tfrac{1}{2}} du \bigg)
            &\quad
            H<\tfrac{1}{2}
        \end{aligned}
        \right.
    \end{equation}
    where
    \begin{equation*}
        c_H = \left\{
        \begin{aligned}
            &\bigg( \frac{H(2H-1)}{\beta(2 - 2H, H- \frac{1}{2})} \bigg)^{\frac{1}{2}} 
            &\quad
            H>\tfrac{1}{2},
            \\
            &\bigg( \frac{2H}{(1 - 2H) \beta(1 - 2H, H + \frac{1}{2})}\bigg)^{\frac{1}{2}}
            &\quad
            H<\tfrac{1}{2} 
        \end{aligned}
        \right.
    \end{equation*}
    see for example \cite{nualart2006malliavin}*{Chapter 5}. Through direct computation, we can verify via the It\^o isometry that 
    \begin{align*}
        \bE\Big[ Z_t Z_s \Big] = \int_0^T K(t, u) K(s, u) du = R(t, s). 
    \end{align*}
    
    \subsubsection*{The fundamental Martingale}

    The fundamental martingale is concept first introduced in \cite{Norros1999Elementary} to understand the change of measure for stochastic differential equations driven by a fractional Brownian motion and was subsequently developed in \cites{Kleptsyna2000General, Kleptsyna2000Parameter} to study filtering problems for stochastic differential equations driven by fractional Brownian motion. We use a slightly different Volterra kernel (Equation \eqref{eq:fbm_Volterra} below) to the one stated in \cite{Norros1999Elementary} which allows us to map directly onto a Brownian motion. The Volterra kernel is somewhat more complicated, but the advantage is that we do not need to account for additional quadratic variation terms that arise in \cites{Kleptsyna2000General, Kleptsyna2000Parameter}. 

    It is commonly understood that a fractional Brownian motion is not a semimartingale; indeed fractional Brownian motion often chosen as the go-to stochastic process that is not a semimartingale. However, consider the following Volterra process driven by a fractional Brownian motion $Z$ (with Hurst parameter $H\in (0,1)$)
    \begin{equation*}
        W_t^* = \int_0^T L(t, s) dZ_s 
    \end{equation*}
    where $L(t, s)=0$ for $t<s$ and for $s<t$
    \begin{equation}
        \label{eq:fbm_Volterra}
        L(t, s) = \left\{
        \begin{aligned}
            &\Big( \frac{s(t-s)}{t} \Big)^{\tfrac{1}{2}-H} - (H-\tfrac{1}{2}) s^{\tfrac{1}{2} - H} \int_s^t (r-s)^{\tfrac{1}{2}-H} r^{H-\tfrac{3}{2}} dv
            &\quad
            H>\tfrac{1}{2},
            \\
            &s^{\tfrac{1}{2}-H} \int_s^t (r-s)^{-H-\tfrac{1}{2}} r^{H-\tfrac{1}{2}} dr
            &\quad
            H<\tfrac{1}{2}.
        \end{aligned}
        \right.
    \end{equation}
    The Volterra kernel $L:[0,T] \to \fWIC$ has the unique property that
    \begin{align*}
        \bE\Big[ W_t^* W_s^* \Big] = t\wedge s
    \end{align*}
    so that $W^*$ is a Brownian motion. We prove that the stochastic process $(W_t^*)_{t\in [0,T]}$ is a local-martingale with respect to the filtration $\bF^Z = (\cF_t^Z)_{t\in [0,T]}$ and further the two filtrations $\bF^Z$ and $\bF^{W^*}$ are equal. Therefore, while a fractional Brownian motion is not a semimartingale, it is a Volterra convolution away from a Brownian motion. We prove these results not just for fractional Brownian motion but also for a class of Gaussian processes in Section \ref{section:Gaussian}. 
    
    \subsection{Girsanov's Theorem}
	
    The dynamics of a stochastic process of the form
    \begin{equation*}
        X_t = X_0 + \int_0^t b\big( s, X[s] \big) ds + Z_t
    \end{equation*}
    can be transformed via a convolution with the Volterra kernel \eqref{eq:fbm_Volterra} into the stochastic process
    \begin{equation*}
        X_t^{\dagger} = X_0 + \int_0^t L(t, s) dX_s
        = 
        X_0 + \int_0^t Q^{b}\big (s, X[s] \big)ds + W_t^*
    \end{equation*}
    where for any progressively measurable function $b: [0,T] \times \cC_T^d \to \bR^d$ we define $Q^b:[0,T] \times \cC_T^d \to \bR^d$ by
    \begin{align}
        \label{eq:Q-transform_fbm}
        Q^b\big( t, X[t] \big) :=& \frac{d}{dt} \int_0^t L(t, s) b\big(s, X[s] \big) ds
        \\
        \nonumber
        =& (\tfrac{1}{2} - H) t^{H-\tfrac{1}{2}} \int_0^t (t-s)^{-(\tfrac{1}{2}+H)} b\big( s, X[s] \big)  s^{\tfrac{1}{2}-H} ds. 
    \end{align}
    This choice of definition is motivated by the fact that for every $t\in [0,T]$
    \begin{equation*}
        \int_0^t Q^b\big( s, X[s] \big) ds = \int_0^t L(t, s) b\big( s, X[s] \big) ds
        \quad \mbox{and}\quad
        \int_0^t b\big( s, X[s] \big) ds = \int_0^t K(t, s) Q^b\big( s, X[s] \big) ds
    \end{equation*}
    where the Volterra kernel $L(t, \cdot)$ was defined in Equation \eqref{eq:fbm_Volterra}. 

    When the integral
    \begin{equation*}
        \int_0^T \Big| Q^b\big( t, x[t] \big) \Big|^2 dt < \infty
        \quad \mbox{we conclude that}\quad
        \Big\| \int_0^\cdot b\big( t, x[t] \big) dt \Big\|_{\RKHS_T}< \infty
    \end{equation*}
    where $\RKHS_T$ is the reproducing kernel Hilbert space generated by the covariances $R(t, \cdot)$ defined in Equation \eqref{eq:covariance-fbm}. As the Hilbert space of square integrable functions $L^2\big([0,T]; \bR \big)$ has a natural projection onto the Hilbert space $L^2\big( [0,t]; \bR \big)$ by multiplying by the indicator function $\1_{[0, t]}$, we define a projection operation $\Pi_t: \RKHS_T \to \RKHS_t$ defined by
    \begin{equation}
        \label{eq:definition:RKHSprojection*}
        \Pi_t\Big[ \int_0^\cdot b\big(s, X[s] \big) ds \Big] = \int_0^{\cdot \wedge t} K(\cdot, s) Q^b\big(s, X[s] \big) ds. 
    \end{equation}
    
    Further, we establish a Girsanov type result for stochastic processes with an additive fractional Brownian motion:
    \begin{theorem} 
        \label{theorem:ap:girsanov*}
        Let $d \in \bN$ let $P_0 \in \cP( \bR^d )$ and let $\gamma \in \cP(\cC_{0, T}^d)$ be the law of a fractional Brownian motion. We denote $P^* = P_0 \times \gamma \in \cP(\cC_T^d)$. 
        
        Let $b : [0,T] \times \cC_T^d \to \bR^d$ be progressively measurable and suppose that
        \begin{equation*}
            \Big\| \int_0^\cdot b\big( s, X[s] \big) ds \Big\|_{\RKHS_T}< \infty \quad \mbox{$P^*$-almost surely. }
        \end{equation*}
        Further, suppose that on the probability space $\big( \cC_T^d, \cB'(\cC_T^d), P^* \big)$ the process
        \begin{equation}
            t\mapsto \cZ_t\Big[ \int_0^\cdot b\big( s, X[s] \big) ds \Big] := \exp\Bigg( \delta\bigg( \Pi_t\Big[ \int_0^{\cdot} b\big(s, X[s] \big) ds \Big] \bigg) - \tfrac{1}{2} \Big\| \int_0^\cdot b\big(s, X[s] \big) ds \Big\|_{\RKHS_t}^2 \Bigg)
        \end{equation}
        (where $\delta$ is the Malliavin divergence) is a martingale that satisfies that
        \begin{equation}
            \bE^{P^*}\bigg[ \cZ_T\Big[ \int_0^\cdot b\big( s, X[s] \big) ds \Big] \bigg] = 1. 
        \end{equation}
        Let $P$ be the probability measure defined by
        \begin{equation*}
            \frac{dP}{dP^*} \bigg|_{\cF_t} = \cZ_t\Big[ \int_0^\cdot b\big( s, X[s] \big) ds \Big]. 
        \end{equation*}
        Then the law of the process
        \begin{equation*}
            X_t - \int_0^t b\big( s, X[s] \big) ds \quad\mbox{under $P$}
        \end{equation*}
        is the same as the law of the canonical process $X$ under $P^*$. 
    \end{theorem}
    This is a slight adaption of the well known result first proved in \cite{Decreusfond1999Stochastic} and we emphasise that we prove this result for a much richer class of Gaussian processes (see Theorem \ref{theorem:ap:girsanov} and Section \ref{section:Gaussian} for more details). 
    
    \subsection{Locally interacting stochastic differential equations} 
    
    For some $(V, E) \in \cG$, we are interested in studying the correlation between the collection of random variables from the solution of the locally interacting stochastic differential equation 
    \begin{equation}
        \label{eq:theorem:MRF-1}
        dX_t^u = b_u\Big( s, X^u[s], X^{N_u}[s] \Big) ds + dZ_t^u,
        \quad
        u \in V,
        \quad
        (X_0^u)_{u\in V} \sim \mu_0. 
    \end{equation}
    The collection of fractional Brownian motions $(Z^u)_{u\in V}$ will all be independent of one another. 

    \begin{theorem}
        \label{theorem:ExUnfbm}
        Let $(V, E)$ be a countably infinite locally finite graph, let $(b_u)_{u\in V}$ be a collection of functions and let $\mu_0 \in \cP\big( (\bR^d)^V \big)$. 
        
        Suppose that for every $u \in V$, the function
        \begin{equation*}
            b_u:[0,T] \times \cC_T^d \times (\cC_T^d)^{N_u} \to \bR^d
        \end{equation*}
        is progressively measurable and that there exists $M \in L^2\big([0,T]; \bR\big)$ such that for every $(X^u, X^{N_u}) \in \cC_T^d \times (\cC_T^d)^{N_u}$, 
        \begin{equation}
            \label{eq:assumption:MRF-1-linear*}
            \left.
            \begin{aligned}
                &\bigg\| \int_0^{\cdot} b_u \big( s, X^u[s], X^{N_u}[s] \big) ds \bigg\|_{\RKHS_T}^2 < \infty,
                \\
                &\Big| Q^{b_u}\big(t, X^u[t], X^{N_u}[t] \big) \Big|\leq M_t \cdot F_u\big( X^u, X^{N_u} \big) \quad \mbox{and}
                \\
                \exists \epsilon>0 \quad \mbox{such that}\quad& \sup_{u\in V} \bE\bigg[ \exp\Big( \epsilon \cdot F_u\big( Z^u+X_0^u, (Z+X_0)^{N_u} \big) \Big) \bigg]< \infty. 
            \end{aligned}
            \qquad \right\}
        \end{equation}
        Then there exists a unique weak solution to the collection of stochastic differential equations
        \begin{equation*}
            dX_t^u = b_u\Big( t, X^u[t], X^{N_u}[t] \Big) dt + dZ_t^u, \quad u \in V, \quad (X_0^u)_{u\in V} \sim \mu_0. 
        \end{equation*}
    \end{theorem}
    We prove a more general version of Theorem \ref{theorem:ExUnfbm} in Section \ref{section:2MRF}, but very briefly the proof relies on an adaption of Theorem \ref{theorem:ap:girsanov*} to random variables defined on locally convex topological vector spaces. To the best of our knowledge, this adaption is novel and avoids previous techniques that rely on compact embeddings. 

    \begin{remark}
        We emphasise that the dynamics of stochastic differential equations of the form \eqref{eq:theorem:MRF-1} are \emph{non-Markovian} in the sense that for every $u \in V$ the drift term $b_u(t, \cdot)$ is dependent on the path of the stochastic process up to time-$t$. We denote a path by $X[t]$ while the time marginal by $X_t$. 

        There are a couple of reasons for this choice of generality: Firstly, as we shall see below the paths of these locally interacting SDEs form a 2-Markov Random Field as a measure on pathspace in part due to the drift term being expressible as a clique functional. This would also be the case if the drift term was only dependent on time marginals, but it is more interesting to consider the most general setting. 

        Another good reason is that the Volterra transformation of the drift term, which we denote by $Q^b(t, \cdot)$, is by definition a functional on the entire path up to time-$t$ even when $b$ is only dependent on the time marginals of the path. By considering $b(t, \cdot)$ as a path functional, the bijective nature of the transformation $b \iff Q^b$ becomes more natural. 
    \end{remark}

    \begin{example}
        \label{example:Q}
        Let $K:[0,T] \to L^2([0,T]; \bR)$ as in Equation \eqref{eq:fBmKernel} be the Volterra kernel for fractional Brownian motion. Let $b:[0,T] \times \cC_T^d \to \bR^d$ be a progressively measurable and denote 
        \begin{align*}
            &Q_u^b: [0,T] \times \cC_T^d \to \bR^d
            \quad \mbox{defined by} \quad
            Q_u^b \Big( t, x[t] \Big):= \frac{d}{dt} \int_0^t L(t, s) b\Big( s, x[s] \Big) ds 
        \end{align*}
        where $L:[0,T] \to \fWIC$ is defined in \eqref{eq:fbm_Volterra}.

        Further, suppose that there exists $\epsilon>0$ such that
        \begin{equation}
            \label{eq:example:Q-fernique}
            \sup_{u\in V} \bE\bigg[ \exp\Big( \epsilon | X_0 |^2 \Big) \bigg]< \infty. 
        \end{equation}
        
        First consider the case where $H<\tfrac{1}{2}$. Suppose that there exists $M:[0,T] \to \bR$ such that
        \begin{align}
            \label{eq:example:Q-lin}
            &\Big| b\big(t, X[t] \big) \Big| \leq M_t \cdot \Big( 1 + \big\| X \big\|_{\infty, t} \Big)
            \quad \mbox{and}\quad
            \int_0^T \int_0^t (t-s)^{-(\tfrac{1}{2}+H)} |M_s| ds |M_t| dt< \infty. 
        \end{align}
        In particular, if $b$ satisfies a uniform in time linear growth condition then Equation \eqref{eq:example:Q-lin} follows. By adapting the techniques first detailed in \cite{Nualart2002Regularization}, we observe that
        \begin{align*}
            \Big| Q^b\big(s, X \big) \Big| =& \Big| s^{H-\tfrac{1}{2}} \cdot \int_0^s (s-r)^{-(\tfrac{1}{2}+H)} r^{\tfrac{1}{2}-H} b\big(r, X\big) dr \Big|
            \\
            \leq& \bigg( s^{H-\tfrac{1}{2}} \int_0^s (s-r)^{-\tfrac{1}{2}-H} r^{\tfrac{1}{2}-H} |M_r| dr \bigg) \cdot \Big( 1+ \big\| X \big\|_{\infty, s} \Big). 
        \end{align*}
        Further, by direct calculation we obtain that
        \begin{align*}
            \int_0^T \bigg| s^{H-\tfrac{1}{2}}& \int_0^s (s-r)^{-\tfrac{1}{2}-H} r^{\tfrac{1}{2}-H} |M_r| dr \bigg|^2 ds
            =
            \int_0^T s^{2H-1} \bigg( \int_0^s (s-r)^{-\tfrac{1}{2}-H} r^{\tfrac{1}{2}-H} \big| M_r \big| dr \bigg)^2 ds
            \\
            \leq& \int_0^T s^{2H-1} \int_0^s \int_0^s (s-r)^{-(\tfrac{1}{2}+H)} r^{\tfrac{1}{2}-H} |M_r| (s-u)^{-(\tfrac{1}{2}+H)} u^{\tfrac{1}{2}-H} |M_u| du dr ds
            \\
            \leq&\int_0^T \int_0^T \bigg( \int_{u \vee r}^T s^{2H-1} (s-r)^{-(\tfrac{1}{2}+H)} (s-u)^{-(\tfrac{1}{2}+H)} ds \bigg)\cdot  r^{\tfrac{1}{2}-H} |M_r| u^{\tfrac{1}{2}-H} |M_u| du dr 
            \\
            \leq& \int_0^T \int_0^T (u\vee r)^{H-\tfrac{1}{2}} (T- u \vee r)^{\tfrac{1}{2}-H} (u \vee r - u\wedge r)^{-(\tfrac{1}{2}+H)} (u\wedge r)^{\tfrac{1}{2}-H} |M_r| \cdot |M_u| du dr
            \\
            \leq& \int_0^T \int_0^r (r-u)^{-(\tfrac{1}{2}+H)} |M_u| |M_r| du dr < \infty. 
        \end{align*}
        To conclude, we remark by Ferniques Theorem and Equation \eqref{eq:example:Q-fernique} that for some choice of $\epsilon>0$, 
        \begin{equation*}
            \bE\bigg[ \exp\Big( \epsilon \big( 1 + \| Z+X_0\|^2 \big) \Big) \bigg] < \infty. 
        \end{equation*}
        As such, we conclude that Equation \eqref{eq:assumption:MRF-1-linear*} follows. 
    \end{example}

    \begin{example}
        \label{example:Q2}
        Now suppose that $H>\tfrac{1}{2}$ but we still have Equation \eqref{eq:example:Q-fernique}. Consider a function of the form
        \begin{equation}
            \label{eq:example:Q-lin-}
            \left.
            \begin{aligned}
                &b\big(s, x[s] \big):=\hat{b}\big(s, x_s \big)
                \quad\mbox{where} \quad
                \hat{b}:[0,T] \times \bR^d \to \bR^d
                \quad \mbox{and}
                \\
                &\Big| \hat{b}\big(s, x\big) - \hat{b}\big(t, y\big) \Big| \leq C \Big( |t-s|^{\gamma} + |x - y|^{\alpha} \Big)
                \quad \mbox{for}\quad
                \alpha \in (1-\tfrac{1}{2H}, 1), \quad \gamma > H-\tfrac{1}{2}. 
            \end{aligned}
            \quad \right\}
        \end{equation}
        Again, by adapting the techniques first detailed in \cite{Nualart2002Regularization} we have that
        \begin{align*}
            \Big| Q^b\big(s, x[s] \big) \Big| 
            =& 
            \hat{b}(s, x_s) s^{\tfrac{1}{2}-H} 
            + (H-\tfrac{1}{2}) s^{H-\tfrac{1}{2}} \hat{b}(s, x_s) \int_0^s \Big( s^{\tfrac{1}{2}-H} - r^{\tfrac{1}{2}-H} \Big) (s-r)^{-(\tfrac{1}{2}+H)}dr 
            \\
            &+ (H-\tfrac{1}{2}) s^{H-\tfrac{1}{2}} \int_0^s \Big( \hat{b}(s, x_s) - \hat{b}(r, x_s) \Big) r^{\tfrac{1}{2}-H} (s-r)^{-(\tfrac{1}{2}+H)} dr
            \\
            &+ (H-\tfrac{1}{2}) s^{H-\tfrac{1}{2}} \int_0^s \Big( \hat{b}(r, x_s) - \hat{b}(r, x_r) \Big) (s-r)^{-(\tfrac{1}{2}+H)} r^{\tfrac{1}{2}-H} dr. 
        \end{align*}
        By applying Equation \eqref{eq:example:Q-lin-}, we conclude that
        \begin{align*}
            \Big| Q^b\big(s, x[s] \big) \Big| 
            \leq& 
            s^{\tfrac{1}{2}-H} \Big( \big| \hat{b}(0, 0) \big| + C\big( |s|^{\gamma} + |x_s|^{\alpha} \big) \Big)
            \\
            &+ (H-\tfrac{1}{2}) s^{H-\tfrac{1}{2}} \Big( |\hat{b}(0, 0)| + C\big( |s|^{\gamma} + |x_s|^{\alpha} \big) s^{1-2H} \Big)
            \\
            &+ (H-\tfrac{1}{2}) s^{H-\tfrac{1}{2}} \int_0^s (s-r)^{\gamma-H - \tfrac{1}{2}} r^{\tfrac{1}{2}-H} dr
            \\
            &+ (H-\tfrac{1}{2}) s^{H-\tfrac{1}{2}} \int_0^s \big| x_s - x_r \big|^{\alpha} (s-r)^{-(\tfrac{1}{2}+H)} r^{\tfrac{1}{2}-H} dr. 
        \end{align*}
        Next, by evaluation we obtain that for some $\varepsilon>0$ chosen small enough that
        \begin{align*}
            \Big| Q^b\big(s, x[s] \big) \Big| 
            \leq& 
            C_T s^{\tfrac{1}{2}-H} \Big( |x_0| + \big\| x\|_{\infty}^{\alpha} + |s|^{\gamma} + \big| \hat{b}(0, 0)\big| + \big\| x \big\|_{H-\varepsilon}^{\alpha} s^{\alpha(H-\varepsilon)} \Big). 
        \end{align*}
        Notice that this choice of the function $F$ from Equation \eqref{eq:assumption:MRF-1-linear*} is not necessarily finite for every choice of $(X^u, X^{N_u}) \in \cC_T^d \times (\cC_T^d)^{N_u}$ and yet the expectation is finite because paths for which $F$ is not defined occur on a null set of the Gaussian measure. 
        
        Now we remark that since a fractional Brownian motion $Z$ with Hurst parameter $H$ has that for any choice of $\varepsilon \in (0, H)$, there exists some $\epsilon>0$ such that
        \begin{equation*}
            \bE\bigg[ \exp\Big( \epsilon\big\| Z\big\|_{H-\varepsilon}^2 \Big) \bigg]< \infty,
        \end{equation*}
        we conclude for $\epsilon>0$ chosen small enough that
        \begin{equation*}
            \bE\bigg[ \exp\Big( \epsilon \big( |X_0| + \big\| X_0 + Z\|_{\infty}^{\alpha} + \big\| X_0 + Z \big\|_{H-\epsilon}^{\alpha} \big) \Big) \bigg]< \infty. 
        \end{equation*}
    \end{example}

    \subsubsection*{Quenched 2-Markov Random Fields}
    
    As each stochastic differential equation is strongly correlated with its neighbours via the drift term and we restrict ourselves to the case where the number neighbours of any vertex is finite, we should not expect to observe any statistical decoupling that one would hope for in a mean-field setting. Therefore, the correlation between the processes marking any two vertices will be highly challenging to compute, even when those two vertices are far from one another. 

    We prove that the collection of stochastic differential equations considered in this work form a 2-Markov Random Field:
    \begin{theorem}
        \label{theorem:MRF-2}
        Let $(V, E)$ be a countably infinite locally finite graph, let $(b_u)_{u \in V}$ be a collection of functions and let $\mu \in \cP\big( (\bR^d)^{V} \big)$. Let $M \in L^1\big( [0,T]; \bR \big)$. 
        
        Suppose that:
        \begin{enumerate}
            \item for every $u \in V$, let $b_u: [0, T] \times \cC_T^d \times (\cC_T^d)^{N_u} \to \bR^d$ be progressively measurable and satisfies Equation \eqref{eq:assumption:MRF-1-linear*};
            \item the measure $\mu_0$ is a 2-Markov Random Field and there exists a collection of measures $(\lambda_u)_{u\in V}$ such that for any finite set $A\subset V$, the marginal measure $\mu_0^A$ is equivalent to the product measure
            \begin{equation*}
                \mu_0^{*,A} = \prod_{v\in A} \lambda_v
            \end{equation*}
            and the initial law $\mu_0$ satisfies 
            \begin{equation*}
                \sup_{v \in V}\int_{\bR^d} |x^v|^2 d\mu_0(x^V) < \infty; 
            \end{equation*}
        \end{enumerate}
        
        Then the collection of random variables $\big( V, E, (X^v[t])_{v\in V} \big)$ is a quenched 2-MRF. 
    \end{theorem}
    The first thing to note is that we do not require that the graph is finite, only that every vertex has a finite neighbourhood. Secondly, as has already been explained in detail in \cite{lacker2020Locally}, the law of the process forms a 2-MRF as a measure on pathspace and it is not the case that the time marginals of the stochastic processes form a 2-MRF. In a sentence, this is because one needs to condition on the entire path of the collection of processes in the separating set and the $\sigma$-algebra generated by only the time marginals is not large enough. 
    
    To prove Theorem \ref{theorem:MRF-2}, which follows from Theorem \ref{theorem:MRF-1}, we first consider the finite graph setting. In this setting, we can write down an explicit Radon-Nikodym derivative between the law of the law of the solution and a reference measure (which we choose to be the product law the collection of independent fractional Brownian motions). This has a natural 2-clique factorisation, so by Proposition \ref{proposition:2HammersleyClifford} we conclude. The extension to the infinite graph setting is more technical, but the techniques are similar to those used in \cite{lacker2020Locally} and we include an adaption of them here for completeness. 
	
    \section{The fundamental martingale and Girsanov's Theorem}
    \label{section:Gaussian}

    The \emph{fundamental martingale} is a concept that was first introduced in \cite{Norros1999Elementary} and further developed in \cites{Kleptsyna2000General, Kleptsyna2000Parameter} as a tool for studying fractional Brownian motion. While results that rely on the existence of the fundamental martingale of fractional Brownian motion are often well cited in the literature, the concept has not received much study in recent years. It is widely known that a fractional Brownian motion is not a martingale (under the canonical filtration), nor indeed a semimartingale; this is often cited as a reason for the study of such processes. It is also well known that fractional Brownian motion admits a representation as a Volterra process driven by a Brownian motion. What is less often used in the literature is that there is a Volterra kernel (defined in Equation \ref{eq:fbm_Volterra}) that transforms a fractional Brownian motion into a Brownian motion. 
    
    A stochastic process that has a fundamental martingale is one which can be transformed into a martingale via convolution with a Volterra kernel. While the motivating example that we consider is a fractional Brownian motion, we emphasise that we also consider a much larger class of Gaussian processes. 

    \subsection{Discrete time intuition}
	\label{subsection:HeuristicFundMart}
	
	In order to provide some intuition to the purpose of the fundamental martingale, consider the following discrete time Gaussian process:
	
	Let $\rZ = ( Z_n)_{n\in \bN}$ be a sequence of Gaussian random variables on a probability space $(\Omega, \cF, \bP)$ and for each $n\in \bN$ let the cumulative vector $\rZ_n = (Z_i)_{i=1, ..., n}$ be an $n$-dimensional centred Gaussian vector with covariance matrix
	\begin{equation*}
		\rR_n = \bE\Big[ \rZ_n \cdot \rZ_n^T \Big] = \Big( \bE\big[ Z_i \cdot Z_j \big] \Big)_{i, j = 1, ..., n}. 
	\end{equation*}
	Suppose (for the moment) that for each $n\in \bN$ the covariance matrix $\rR_n$ is \emph{invertible}. Then for each $n\in \bN$ the Gaussian vector $\rZ_n$ has density
	\begin{equation*}
		p_n(x) = \frac{1}{\sqrt{(2\pi)^n \cdot \mbox{det}(\rR_n)}} \cdot \exp\Big( - \frac{1}{2} x^T (\rR_n)^{-1} x \Big). 
	\end{equation*}
	By inductively defining the matrices, row vectors and scalars such that
	\begin{equation*}
		\rR_{n+1} = 
		\begin{pmatrix}
			\rR_n ,&  \rR_{n1} \\
			\rR_{1n} ,& r_{n+1}
		\end{pmatrix}
		\quad \mbox{and}\quad
		(\rR_{n+1})^{-1} = 
		\begin{pmatrix}
			\rS_n,& \rS_{n1} \\
			\rS_{1n},& s_{n+1}
		\end{pmatrix}, 
	\end{equation*} 
	we get the identities
	\begin{equation}
		\label{eq:Heuristic-id}
		\begin{aligned}
			\rS_n \cdot \rR_n + \rS_{n1} \cdot \rR_{1n} =& I_n
			&\mbox{and}&\quad
			\rS_{1n} \cdot \rR_n + s_{n+1} \cdot \rR_{1n} = 0
			\\
			\rS_n + \rS_{n1} \cdot \rR_{1n} \cdot ( \rR_n)^{-1} =& ( \rR_n)^{-1} 
			&\mbox{and}&\quad
			\rS_{1n} + s_{n+1} \cdot \rR_{1n} \cdot (\rR_n)^{-1} = 0. 
		\end{aligned}
	\end{equation}
	
	We define the sequence of random variables
	\begin{equation}
        \label{eq:FundMart-DiscreteCase}
		\big( M_n \big)_{n\in \bN} 
		\quad \mbox{where} \quad
		M_n := \textbf{1}_n^T \cdot (\rR_n)^{-1} \cdot \rZ_n,
	\end{equation}
	where $\textbf{1}_n$ is a column vector with all values equal to $1$. 
	
	\begin{lemma}
		Let $\rZ =(Z_n)_{n\in \bN}$ be a sequence of Gaussian random variables on a probability space $(\Omega, \cF, \bP)$ and suppose that for each $n\in \bN$ the covariance matrix $\rR_n$ is invertible. 
		
		Then the sequence of random variables $(M_n)_{n\in \bN}$ is a Martingale. 
	\end{lemma}
    We refer to $\big( M_n \big)_{n\in \bN}$ as \emph{the fundamental martingale}.
	
	\begin{proof}
		First consider the sequence of $\sigma$-algebras $\big( \sigma( M_n ) \big)_{n\in \bN}$. Since for each $n\in \bN$ we have the matrix $\rR_n$ is invertible, we have that $\sigma(M_n) = \sigma(\rZ_n) \subseteq \cF$. 
		
		Courtesy of classical results for conditional Gaussian measures, we have that
		\begin{equation*}
			\bE\Big[ Z_{n+1} \Big| \sigma\big( \rZ_{n} \big) \Big] = \rR_{n1} \cdot (\rR_n)^{-1} \cdot \rZ_n. 
		\end{equation*}
		Combining this with Equation \eqref{eq:Heuristic-id}, we get
		\begin{align*}
			\bE\Big[& M_{n+1} \Big| \sigma\big( M_1, ..., M_n \big) \Big]
			= \bE\Big[ \textbf{1}_n^T \cdot (\rR_{n+1})^{-1} \cdot \rZ_{n+1} \Big| \sigma\big( M_1, ..., M_n \big) \Big]
			\\
			=& \bE\bigg[ 
			\begin{pmatrix}
				\textbf{1}_n^T & 1
			\end{pmatrix}
			\cdot 
			\begin{pmatrix}
				\rS_n, & \rS_{n1}\\
				\rS_{1n}, & s_{n+1}
			\end{pmatrix}
			\cdot 
			\begin{pmatrix}
				\rZ_n \\ Z_{n+1}
			\end{pmatrix}
			\bigg| \sigma\big( \rZ_n \big) \bigg]
			\\
			=&\Big( \textbf{1}_n^T \cdot \rS_n + \rS_{1n}\Big) \cdot \rZ_n + \Big( \textbf{1}_n^T \cdot \rS_{n1} + s_{n+1} \Big) \cdot \bE\Big[  Z_{n+1} \Big| \sigma\big( \rZ_n \big) \Big]
			\\
			=& \textbf{1}_n^T \cdot (\rR_n)^{-1} \rZ_n = M_n
		\end{align*}
	\end{proof}

	In particular, this means that
	\begin{align*}
		&\bE\Big[ M_n \cdot \big( M_{n+1} - M_n \big) \Big] 
		\\
		&= \bE\bigg[ \textbf{1}_n^T\cdot (\rR_n)^{-1} \cdot \rZ_n \cdot \Big( \textbf{1}_{n+1}^T \cdot (\rR_{n+1})^{-1} \cdot \rZ_{n+1} - \textbf{1}_{n+1}^T \cdot (\rR_n)^{-1} \cdot \rZ_n \Big) \bigg]
		\\
		&= \bE\bigg[ \textbf{1}_n^T \cdot \rS_n \cdot \rZ_n \cdot \Big( \textbf{1}_n^T \cdot \rS_{n1} \cdot Z_{n+1} + \rS_{1n} \cdot \rZ_n + s_{n+1} \cdot Z_{n+1} \Big) \bigg]
		\\
		&= \textbf{1}_n^T \cdot \rS_n \cdot \Big( \rR_n \cdot \rS_{n1} +  \rR_{n1} \cdot \rS_{1n} \cdot \textbf{1}_{n} + \rR_{n1} \cdot s_{n+1} \Big) = 0. 
	\end{align*}

	As first described in \cite{Norros1999Elementary}, we can interpret $M_n$ as the (discrete time) stochastic integral of the kernel $L(n, i) = \big( \textbf{1}_n^T \cdot (\rR_n)^{-1} \big)_i$ with respect to the cumulative process $Z_n = \sum_{i=1}^n Z_i$. 
 
	\iftoggle{ArXiv}{
    Matlab code for simulating the vector $L$
    \begin{lstlisting}[style=Matlab-editor]
        n=5;    %size of matrix
        e=0.1;  %time increment
        H=0.3;  %Hurst parameter
        F=2*abs(e.*I).^(2*H) - abs(e.*I-e).^(2*H) - abs(e.*I+e).^(2*H);
        % F is the Toeplitz vector for fBm
        R = toeplitz(F);
        % covariance matrix
        S = inv(R); 
        %inverse of covariance
        O=ones(1,n); 
        %vector of 1s
        L=O*S;
        %L is the vector that represents the kernel that we integrate with respect to
    \end{lstlisting}
    }{}

    \subsubsection*{Local non-determinism and the fundamental martingale (discrete time)}

    The sequence of matrices $(\rR_n)_{n\in \bN}$ need not be invertible, but for the martingale $M_n$ to exist, we require that each inverse exists. With this in mind, we introduce the following:
    \begin{definition}
        Let $\rZ = (Z_n)_{n\in \bN}$ be a sequence of Gaussian random variables and for each $n\in \bN$ let the sumulative vector $\rZ_n = (Z_i)_{i=1,..., n}$ be an $n$-dimensional centred Gaussian vector with covariance matrix
        \begin{equation*}
            \rR_n = \bE\Big[ \rZ_n \cdot \rZ_n^T \Big] = \Big( \bE[ Z_i \cdot Z_j ] \Big)_{i, j=1, ..., n}.
        \end{equation*}
        We say that $\rZ$ is \emph{locally non-deterministic} if for every choice of $n\in \bN$ we have that
        \begin{equation}
            \label{eq:lemma:LND1}
            \bE\bigg[ \Big( Z_{n+1} - \bE\big[ Z_{n+1} \big| \sigma(\rZ_n) \big] \Big)^2 \bigg]>0
        \end{equation}
    \end{definition}
    
    \begin{lemma}
        \label{lemma:LND1}
        Let $\rZ = (Z_n)_{n\in \bN}$ be a sequence of Gaussian random variables and suppose that $\rZ$ is locally non-deterministic. Then for every choice of $n \in \bN$ we have that the matrix $\rR_n$ is invertible. 
    \end{lemma}

    \begin{proof}
        We proceed via induction: suppose that the matrix $\rR_n$ is invertible and consider the matrix
        \begin{align*}
            &\rR_{n+1} = 
            \begin{pmatrix}
                \rR_n & \rR_{+n} \\
                \rR_{n+} & \rR_{++}
            \end{pmatrix}
            \quad \mbox{where}\quad
            \rR_{+n} \in \bR^{d\times 1}, \rR_{+n} \in \bR^{1\times d} \quad \mbox{and}\quad \rR_{++} \in \bR. 
        \end{align*}
        Then our primary hypothesis \eqref{eq:lemma:LND1} can be rewritten as 
        \begin{equation*}
            \lambda_n = \rR_{++} - \rR_{n+} \cdot (\rR_n)^{-1} \cdot \rR_{+n} > 0. 
        \end{equation*}
        Consider the matrix
        \begin{equation*}
            \rQ_{n+1} = \frac{1}{\lambda_n}
            \begin{pmatrix}
                \lambda_n (\rR_n)^{-1} + (\rR_n)^{-1} \cdot \rR_{+n} \cdot \rR_{n+} \cdot (\rR_n)^{-1} 
                &,& 
                -(\rR_n)^{-1} \cdot \rR_{+n}
                \\
                -\rR_{n+} \cdot (\rR_n)^{-1} 
                &,&
                1
            \end{pmatrix}. 
        \end{equation*}
        By direct calculation we obtain that, 
        \begin{equation*}
            \rR_{n+1} \cdot \rQ_{n+1} = 
            \begin{pmatrix}
                I_n &,& 0\\
                0 &,& 1
            \end{pmatrix}
            \quad \mbox{and}\quad
            \rQ_{n+1} \cdot \rR_{n+1} = 
            \begin{pmatrix}
                I_n &,& 0\\
                0 &,& 1
            \end{pmatrix}
        \end{equation*}
        so that the matrix $\rR_{n+1}$ has an explicit inverse. 
    \end{proof}

    This means that the fundamental martingale
    \begin{equation*}
        M_n = \textbf{1}_n^T \cdot (\rR_n)^{-1} \cdot \rZ_n
    \end{equation*}
    satisfies the identity
    \begin{align*}
        M_{n+1} =& \bigg[ \textbf{1}_n^T - \frac{\big( 1 - \textbf{1}_n^T \cdot (\rR_n)^{-1} \cdot \rR_{+n} \big) \cdot \rR_{n+}}{\lambda_n} \bigg] \cdot (\rR_n)^{-1} \cdot \rZ_n
        \\
        &+ \frac{\big( 1 - \textbf{1}_n^T \cdot (\rR_n)^{-1} \cdot \rR_{+n} \big)}{\lambda_n} Z_{n+1}
    \end{align*}

    \begin{example}[Fractional Brownian motion]
        \label{example:discretefBm}
        Let $\big\{ [t_{i,n}, t_{i+1, n}] : i=1, ..., n \big\}$ be a uniform partition of $[0,1]$, so that $t_{i, n} = \tfrac{i}{n}$. Let $Z_t$ be a fractional Brownian motion on $[0, 1]$ with Hurst parameter $H \in (0, 1)$. We denote the increments of the stochastic process $Z=(Z_i^n)_{i=0, ..., n}$ by $Z_0^n = 0$ and
        \begin{equation}
            Z_i^n:= Z_{t_{i, n}} - Z_{t_{i-1, n}},
        \end{equation}
        for $i \in \{1, \ldots, n\}$. Thanks to Equation \eqref{eq:covariance-fbm}, the covariance matrix
        \begin{align*}
    		\rR_n &= \bigg(\bE\Big[ Z_i^n \cdot Z_j^n \Big]\bigg)_{i, j=1, ..., n}
            \quad \mbox{where}
            \\
            \bE\Big[ Z_i^n Z_j^n \Big] &= \frac{1}{2}\Big( |t_{i, n} - t_{j - 1, n}|^{2H} + |t_{i - 1, n} - t_{j, n}|^{2H} - |t_{i, n} - t_{j, n}|^{2H} - |t_{i - 1, n} - t_{j - 1, n}|^{2H} \Big)
            \\
            &= \frac{1}{2n^{2H}} \Big(|i - j - 1|^{2H} + |i - j + 1|^{2H} - 2 |i - j|^{2H} \Big)
            =: A\big( |i-j| \big). 
        \end{align*}    
        where $A:\bN_0 \to \bR$ satisfies
        \begin{equation}
            \label{eq:example:discretefBm-Toep}
            A(k) = \frac{1}{2}\Big( |k - 1|^{2H} + |k + 1|^{2H} - 2 |k|^{2H} \Big).
        \end{equation}

        Hence the covariance matrix takes the form 
        \begin{align*}
            \rR_n = \frac{1}{n^{2H}}\rA_n = \frac{1}{n^{2H}} \Toeplitz \Big( \big\{ A(k) \big\}_{k = 0, \ldots, n-1} \Big).
        \end{align*}
        When $H<\tfrac{1}{2}$, we have that $A(k) < 0$ for all $k \geq 1$ so that
        \begin{align*}
            \sum_{k = 1}^{n - 1} \big| A(k) \big| = 
            \begin{cases}
                - \sum_{k = 1}^{n - 1} A(k)
                \quad& \mbox{when $H< \tfrac{1}{2}$}
                \\
                \sum_{k = 1}^{n - 1} A(k)
                \quad& \mbox{when $H> \tfrac{1}{2}$}
            \end{cases}
        \end{align*}
        and by a telescoping argument
        \begin{equation*}
            \sum_{k = 1}^{n - 1} \big| A(k) \big| = 
            \begin{cases}
                \tfrac{1}{2} \Big( |n|^{2H} - |n-1|^{2H} - |1|^{2H} \Big)
                \quad& \mbox{when $H< \tfrac{1}{2}$, }
                \\
                \\
                \tfrac{1}{2} \Big( |n-1|^{2H} - |n|^{2H} + |1|^{2H} \Big)
                \quad& \mbox{when $H> \tfrac{1}{2}$. }
            \end{cases}
        \end{equation*}
        When $H< \tfrac{1}{2}$, 
        \begin{align*}
            \sum_{k = 1}^{n - 1} \big| A(k) \big| \leq \frac{1}{2} < 1 = A(0)
        \end{align*}
        and the \emph{Gershgorin Circle Lemma} (see \cite{Horn2013MAtrix}) tells us that this is sufficient to ensure the invertibility of a Toeplitz matrix $\rR_n$. 
    \end{example}

    \begin{remark}
        We have shown that the discrete time fractional Brownian motion admits a fundamental martingale, but this is does not necessarily ensure that the fractional Brownian motion admits a fundamental martingale (although at this point we know that it does) . Further, it is not clear how we define the operator $( \rR_n)^{-1} $ in the continuous time limit. 
    \end{remark}

    \subsection{Volterre processes}
    \label{subsection:VolterraProcess}
    
    Let $\RKHS \subseteq \cC_T^d$ be a Hilbert space of continuous functions over $[0,T]$ and let $\BSi: \RKHS \to \cC_T^d$ be a compact linear operator. By the structure theorem for Gaussian measures \cite{bogachev1998gaussian}, we have that there exists a unique centred positive Gaussian measure with support equal to $\overline{\RKHS}$ the closure of $\RKHS$ in the supremum norm. We interchangeably refer to this Gaussian measure and the abstract Wiener space $(\cC_T^d, \cH, \BSi)$. 

    \begin{definition}
        Let $d\in \bN$ and let $(\Omega, \cF, \bF, \bP)$ be a complete filtered probability space carrying a $d$-dimensional Brownian motion. 

        Let $K:[0,T] \to L^2\big( [0,T]; \lin(\bR^d, \bR^d) \big)$ be a Volterra kernel (that is $K(t, s) = 0$ for $s>t$). We say that a stochastic process is a \emph{Gaussian Volterra process} if it can be expressed as
        \begin{equation}
    		\label{eq:GaussVolterra}
    		Z_t = \int_0^t K(t, s) dW_s
	    \end{equation}
        where $W$ is a Brownian motion and $K$ is a Volterra kernel. 
    \end{definition} 
    In particular, the covariance function for a Gaussian Volterra processes satisfies
	\begin{equation}
        \label{eq:R-covariance}
		R(s, t) = \int_0^{t \wedge s} \Big\langle K(t, u), K(s, u) \Big\rangle_{\lin(\bR^d, \bR^d)} du = \int_0^{t\wedge s} \mbox{Tr}\Big( K(t, u) \cdot K^*(s, u) \Big) du. 
	\end{equation}
    In order to illustrate the key properties of the Gaussian processes that we consider for driving signals, we need to consider three distinct Hilbert spaces:
    \begin{enumerate}[label=(\roman*)]
        \item The \emph{reproducing kernel Hilbert space} is the closure of the collection of covariance functions 
        \begin{align*}
            &\RKHS_T:=\spn\bigg\{ \bE\Big[ \langle Z_t, u\rangle_{\bR^d} Z_{\cdot} \Big]: t\in [0,T], u \in \bR^d \bigg\}
            \\
            &\mbox{with inner product} \quad 
            \Big\langle \bE\big[ \langle Z_t, u\rangle Z_{\cdot} \big], \bE\big[ \langle Z_s, v\rangle Z_{\cdot} \big] \Big\rangle_{\RKHS_T} = \bE\Big[ \langle Z_t, u\rangle  \cdot \langle Z_s, v \rangle \Big]. 
        \end{align*}
        \item The \emph{first Wiener-Ito chaos} is the closure of the collection of step functions
        \begin{align*}
            &\fWIC_T:=\spn\bigg\{ \1_{[0,t]}(\cdot) e_{i,i}: t\in [0,T], i\in \{1, ..., d\} \bigg\}
            \\
            &\mbox{with inner product}\quad
            \Big\langle \1_{[0,t]} e_{i,i}, \1_{[0,s]} e_{j,j} \Big\rangle_{\fWIC_T} = \bE\Big[ \langle Z_t, e_i\rangle \cdot \langle Z_s, e_j\rangle \Big]. 
        \end{align*}
        \item The \emph{Volterra space} is the closure of the collection of Volterra kernels
        \begin{equation*}
            \cV_T:=\spn\Big\{ K(t, \cdot): t\in [0,T] \Big\}
            \quad \mbox{with inner product}\quad
            \Big\langle K(t, \cdot), K(s, \cdot) \Big\rangle_{\cV_T} = R(t, s). 
        \end{equation*}
    \end{enumerate}
    We have the following Hilbert space isometric isomorphisms
    \begin{equation*}
        \scI: \fWIC_T \to \RKHS_T, 
        \quad
        \scJ: \cV_T \to \fWIC_T, 
        \quad 
        \scK: \cV_T \to \RKHS_T, 
    \end{equation*}
    defined by
    \begin{equation*}
        \scI \Big[ \1_{[0,t]}(\cdot) e_{i, i} \Big](s) = \bE\Big[ \langle Z_t, e_i\rangle Z_s \Big], 
        \quad
        \scJ\Big[ K(t, \cdot) \Big](s) = \1_{[0,t]}(s) I_d, 
        \quad
        \scK\Big[ K(t, \cdot) \Big](s) = \bE\Big[ \sum_{i=1}^d \langle Z_t, e_i\rangle Z_s \Big], 
    \end{equation*}
    and $\scI \circ \scJ = \scK$. Further, using Equation \eqref{eq:R-covariance} we conclude that
    \begin{equation*}
        \cV_T \subseteq L^2\big( [0,T]; \lin(\bR^d, \bR^d) \big) 
        \quad \mbox{since}\quad
        \Big\langle K(t, \cdot), K(s, \cdot) \Big\rangle_{\cV_T} = \int_0^T \Big\langle K(t, r), K(s, r)\Big\rangle dr. 
    \end{equation*}
    \begin{equation}
        \label{eq:CovarianceCommute}
        \begin{aligned}
            \begin{tikzpicture}
                \node at (0,0) {$\fWIC_T$};
                \node at (2,0) {$\RKHS_T$};
                \node at (1,-2) {$\cV_T$};
                \draw[-to](0.75,-1.75) to (0.25, -0.3);
                \draw[-to](1.25,-1.75) to (1.75, -0.3);
                \draw[-to](0.4, 0) to (1.6,0);
                \node at (0,-1) {$\scJ$};
                \node at (2,-1) {$\scK$};
                \node at (1,0.5) {$\scI$};
            \end{tikzpicture}
        \end{aligned}
    \end{equation}
    
    \subsubsection*{Isonormal Gaussian processes}
 
    Given a probability space $(\Omega, \cF, \bP)$ carrying a Gaussian process $(Z)_{t\in[0,T]}$, we define the mapping $Z: \fWIC \to L^2\big( \Omega, \bP; \bR^d \big)$ by
    \begin{equation*}
        Z\big( \1_{[0,t]} \big) = Z_t 
        \quad \mbox{so that} \quad
        \bE\Big[ \big\langle Z(g), Z(h) \big\rangle_{\bR^d} \Big] = \Big\langle g, h \Big\rangle_{\fWIC_T}
    \end{equation*}
    Then $Z$ is an \emph{isonormal Gaussian process}. Similarly, the isonormal Gaussian process of It\^o integration
    \begin{equation*}
        W: L^2\big( [0,T]; \bR^d \big) \to L^2\big( \Omega, \bP; \bR^d \big)
        \quad \mbox{defined by}\quad
        W(f) = \int_0^T f_s dW_s. 
    \end{equation*}
    Then we conclude that $\bP$-almost surely the random variables
    \begin{equation}
        \label{eq:J-formula}
        Z(h) = \int_0^T \scJ^*\big[h\big]_s dW_s. 
    \end{equation}
    In particular, $I_d \1_{[0,t]} \in \fWIC$, $\scJ^*\big[ \1_{[0,t]} I_d \big] = K(t, \cdot)$ and
    \begin{equation*}
        Z_t = Z\big( \1_{[0,t]} I_d \big) = \int_0^t K(t, s) dW_s. 
    \end{equation*}

    Further, due to Equation \eqref{eq:R-covariance} we observe that the inner product on $\cV_T$ is equivalent to the inner product on $L^2\big( [0, T]; \bR^d \big)$ so that $\cV_T \subseteq L^2\big( [0,T]; \bR^d\big)$ and the dual operator $\scJ^*: \fWIC_T \to \cV_T$ satisfies
	\begin{equation}
		\label{eq:GaussianIso}
		\Big\langle \scJ^*\big[f \big], \scJ^*\big[ g \big] \Big\rangle_{L^2([0,T])} = \Big\langle f, g \Big\rangle_{\fWIC}. 
	\end{equation}

    \subsubsection*{The fundamental martingale}

    Now, we extend our framework the continuous time setting. First of all, we introduce the following condition for the kernel of a Gaussian Volterra process:
	\begin{assumption}
		\label{assumption:VolterraK}
		Let $d\in \bN$ and let $K:[0,T] \to L^2\big( [0,T]; \lin(\bR^d, \bR^d) \big)$ be a Volterra kernel. Let the $d$-dimensional Gaussian Volterra process defined by
        \begin{equation*}
            Z_t = \int_0^t K(t, s) dW_s
            \quad \mbox{with covariance}\quad
            R(t, s) = \int_0^{t \wedge s} \Big\langle K(t, r), K(s, r) \Big\rangle_{\lin(\bR^d, \bR^d)} dr. 
        \end{equation*}
        Suppose that there exists a Volterra kernel $L:[0,T] \to \fWIC$ such that for any $s, t\in [0,T]$, 
        \begin{equation}
            \label{eq:assumption:VolterraK}
            \scJ^*\big[ L(t, \cdot) \big](s) = \1_{[0,t]}(s) I_d
        \end{equation}
        where $I_d \in \lin(\bR^d, \bR^d)$ is the identity matrix. 
	\end{assumption}
    We are interested in choices of $K$ for which the Volterra space 
    \begin{equation*}
        \cV_T \quad\mbox{is isometrically isomorphic to}\quad L^2\big([0,T];\lin(\bR^d, \bR^d) \big).
    \end{equation*}
    
    \begin{lemma}
        \label{lemma:RKHS=I^*}
        Let $(\Omega, \cF, \bF, \bP)$ be a complete filtered probability space carrying a Brownian motion, let $K:[0,T] \to L^2\big([0,T]; \lin(\bR^d, \bR^d) \big)$ be a Volterra kernel that satisfies Assumption \ref{assumption:VolterraK} and let $Z$ be a Gaussian Volterra process of the form \eqref{eq:GaussVolterra}. 

        Then for any $t\in [0,T]$ the reproducing kernel Hilbert space is isometrically isomorphic to
        \begin{equation*}
            \RKHS_t = \bigg\{ \int_0^\cdot K(\cdot, s) h_s ds: h\in L^2\big( [0,t]; \bR^d \big) \bigg\} \subseteq \cC_{0, T}^d
        \end{equation*}
    \end{lemma}

    \begin{proof}
        For any element 
        \begin{align*}
            &h\in \spn\Big\{ \bE\big[ \langle Z_t, u\rangle \cdot Z \big] : t\in [0, T], u \in \bR^d \Big\} \quad \iff
            \\
            &h_t = \bE\Big[ \sum_{i\in I} a_i \langle Z_{s_i}, u_i\rangle Z_t \Big] = \int_0^T K(t, r) \cdot \bigg( \sum_{i\in I} a_i K(s_i, r)\cdot u_i \bigg) dr, 
        \end{align*}
        where $I$ is a finite set and $a_i \in \bR$, $s_i \in [0,T]$ and $u_i \in \bR^d$. Thus we have a bijection between the sets
        \begin{align*}
            &\spn\bigg\{ t \mapsto \bE\Big[ \langle Z_s, u\rangle_{\bR^d} \cdot Z_t \Big]: s \in [0,T], u \in \bR^d \bigg\} 
            \quad \mbox{and}\quad 
            \\
            &\bigg\{ t\mapsto \int_0^T K(t, r) f_r dr: \quad f \in \spn\Big\{ \langle K(s, \cdot), u\rangle: s\in [0,T], u \in \bR^d \Big\} \bigg\}. 
        \end{align*}
        Further
        \begin{align}
            \nonumber
            \bigg\langle& \sum_{i \in I} a_i \bE\Big[ \langle Z_{s_i}, u_i\rangle Z_\cdot\Big], \sum_{j\in J} b_j \bE\Big[ \langle Z_{r_j}, v_j\rangle Z_\cdot \Big] \bigg\rangle_{\cH_T} 
            \\
            &= \int_0^T \bigg\langle \sum_{i\in I} a_i K(s_i, r) \cdot u_i , \sum_{j\in J} b_j K(r_j, r)\cdot v_j \bigg\rangle_{\bR^d} dr
            \\
            \label{eq:RKHS=I^*-pf1}
            &= \bigg\langle \sum_{i \in I} a_i K(s_i, \cdot)\cdot u_i, \sum_{j\in J} b_j K(r_j, \cdot )\cdot v_j \bigg\rangle_{L^2\big( [0,T]; \bR^d \big)}. 
        \end{align}
        so that the topology induced by the inner product $\langle \cdot, \cdot \rangle_{\RKHS}$ is equivalent to the topology induced by the $L^2\big([0,T]; \bR^d \big)$ inner product on $\cV$. Hence
        \begin{align}
            \nonumber
            \RKHS_T&=\overline{\spn\Big\{ R(s, \cdot): s\in [0, T] \Big\} }^{\cH_T}
            \\
            \label{eq:RKHS=I^*-pf2}
            &=\bigg\{ t\mapsto \int_0^T K(t, r) f_r dr: \quad f \in \overline{\spn\Big\{ K(s, \cdot)\cdot u: s\in [0,T], u \in \bR^d \Big\}}^{L^2\big( [0,T]; \bR^d \big)} \bigg\}. 
        \end{align}

        Next, recall that the isometry $\scJ: \cV \to \fWIC$ defined by $\scJ\big[ K(t, \cdot) \big] = \1_{[0,t]} I_d$ is an isomorphism. Assumption \ref{assumption:VolterraK} implies that for every $t\in [0,T]$, there exists some $L(t,\cdot) \in \fWIC$ such that the adjoint operator $\scJ^*: \fWIC \to \cV$ satisfies Equation \eqref{eq:assumption:VolterraK}. In particular, this means that for every $t\in [0,T]$ and $i\in \{1, ..., d\}$, 
        \begin{equation*}
            \1_{[0,t]} e_{i, i} \in \cV_T. 
        \end{equation*}
        As the step functions form a dense subset of $L^2\big( [0,T] ; \bR \big)$, we conclude that 
        \begin{equation*}
            \overline{\spn\Big\{ K(s, \cdot)\cdot u: s\in [0,T], u \in \bR^d \Big\}}^{L^2\big( [0,T]; \bR^d \big)} = L^2\big( [0,T]; \bR^d \big)
        \end{equation*}
        and the final conclusion follows from Equation \eqref{eq:RKHS=I^*-pf2}. 
    \end{proof}
    
    \begin{example}
        \label{example:fbm}
        Recalling the Volterra kernel defined in Equation \eqref{eq:fBmKernel}, we recall the result originally proved in \cite{Decreusfond1999Stochastic} that the reproducing kernel Hilbert space for fractional Brownian motion can be written as
        \begin{equation*}
            \RKHS_T = \left\{
             \begin{aligned}
                \bigg\{ t\mapsto &\int_0^t \Big( \int_0^s u^{\tfrac{1}{2}-H} (s-u)^{H-\tfrac{3}{2}} f_u du \Big) s^{H-\tfrac{1}{2}} ds : \quad f\in L^2\big( [0,T]; \bR \big) \bigg\}, 
                &H>\tfrac{1}{2} 
                \\
                \bigg\{ t\mapsto &\Big( t^{H-\tfrac{1}{2}} \int_0^t s^{\tfrac{1}{2}-H} (t-s)^{H-\tfrac{1}{2}} f_s ds 
                \\
                &- \big( H-\tfrac{1}{2} \big) \int_0^t \Big( \int_0^s u^{\tfrac{1}{2}-H} (s-u)^{H-\tfrac{1}{2}} f_u du \Big) s^{H-\tfrac{3}{2}} ds \Big) : \quad f\in L^2\big( [0,T]; \bR \big) \bigg\},
                &H<\tfrac{1}{2} 
             \end{aligned}
             \right.
        \end{equation*}
    \end{example}

    The existence of such a Volterra kernel provided by Assumption \ref{assumption:VolterraK} introduces a new Gaussian process via the stochastic integral with respect to $Z$:
    \begin{lemma}
		\label{lemma:Fundamental-Wiener}
		Let $\big( \Omega, \cF, \bF, \bP \big)$ be a filtered probability space carrying a Brownian motion. Let $K:[0,T] \to L^2\big( [0,T]; \lin(\bR^d, \bR^d) \big)$ be a Volterra kernel that satisfies Assumption \ref{assumption:VolterraK} and let $Z$ be the Gaussian Volterra process of the form Equation \eqref{eq:GaussVolterra}. 
  
        We define
		\begin{equation}
            \label{eq:proposition:Fundamental-Wiener}
			W_t^* = \int_0^t L(t, s) dZ_s. 
		\end{equation}
		Then $W_t^*$ is a $d$-dimensional $\bF$-Brownian motion. 
	\end{lemma}

    \begin{proof}
        Firstly, for each $t\in [0,T]$ the random variable $W_t^*$ is $\cF_t^Z$-measurable. Further, the random variable $Z_t$ is $\cF_t$-measurable so that $W_t^*$ is also $\cF_t$-measurable. 
        
        By calculating the covariance operator, we get that for any $s, t\in [0,T]$ that
		\begin{align*}
			R(t, s) = \bE\Big[ \big\langle W_t^*, W_s^* \big\rangle_{\bR^d} \Big] =& \bE\bigg[ \Big\langle \int_0^t L(t, u) dZ_u, \int_0^s L(s, u) dZ_u \Big\rangle_{\bR^d} \bigg]
			\\
			&= \int_0^T \Big\langle \scJ^*\big[ L(t, \cdot) \big](r) , \scJ^*\big[ L(s, \cdot) \big](r) \Big\rangle_{\lin(\bR^d, \bR^d)} dr
			\\
			&= d \int_0^T \1_{[0,t]}(r) \cdot \1_{[0,s]}(r) dr = d \cdot (t\wedge s). 
		\end{align*}
		Therefore, $W^*$ has the same covariance relationship as a $d$-dimensional Brownian motion.
    \end{proof}

    \subsubsection*{Drift transformations}

    Given two Hilbert spaces $\cG$ and $\cH$, we define the set of Bilinear forms $\mbox{BiLin}(\cG, \cH)$ to be the collection of bilinear forms from $\cG$ and $\cH$ to $\bR$. 
    \begin{lemma}
        \label{lemma:abscont-Q}
        Let $K: [0,T] \to L^2\big( [0,T]; \lin(\bR^d, \bR^d) \big)$ be a Volterra kernel that satisfies Assumption \ref{assumption:VolterraK} and for every $t\in [0, T]$ let $\RKHS_t$ and $\fWIC_t$ be the associated reproducing kernel Hilbert space and the first Wiener-Ito chaos. 

        We denote the functional $\Lambda:[0,T] \to \mbox{BiLin}(\RKHS_T, \fWIC_T)$ defined by
        \begin{equation*}
            \Lambda_t\big[ \phi, \psi \big]:=\int_0^t \Big\langle \scJ^*\big[ \phi \big](s), \scK^{*}\big[ \psi \big](s) \Big\rangle_{\bR^d} ds. 
        \end{equation*}
        For every $t\in [0,T]$, we define $\Lambda_t\big[ L(t, \cdot), \cdot \big]: \RKHS_T \to \bR$. Then for every $\psi \in \RKHS_T$, 
        \begin{equation}
            t \mapsto \Lambda_t\Big[ L(t, \cdot), \psi \Big]
        \end{equation}
        is absolutely continuous. 
    \end{lemma}
    
    \begin{proof}
        Let $\psi \in \RKHS_T$. Then thanks to Lemma \ref{lemma:RKHS=I^*}, there exists some $h \in L^2\big([0,T]; \bR^d \big)$ such that
        \begin{equation*}
            \psi_t = \int_0^t K(t, s) h_s ds. 
        \end{equation*}
        Then $\scJ^*\big[ L(t, \cdot) \big](s) = \1_{[0,t]}(s)$ and $\scK^*\big[ \psi \big](s) = h_s$ so that
        \begin{equation}
            \label{eq:lemma:abscont-Q-pf1}
            t\mapsto \Lambda_t\big[ L(t, \cdot) , \psi \big] = \int_0^t h_s ds. 
        \end{equation}
        Therefore, by construction this map is absolutely continuous. 
    \end{proof}
    
    Having established that the existence of a Volterra kernel that transforms our Gaussian Volterra process back to a Brownian motion, our next result explores how elements of the Reproducing Kernel Hilbert space can also be transformed bijectively to $L^2\big( [0,T]; \bR^d \big)$:
    \begin{proposition}
        \label{prop:Existence_Q}
        Let $K:[0,T] \to L^2\big( [0,T]; \lin(\bR^d, \bR^d) \big)$ be a Volterra kernel that satisfies Assumption \ref{assumption:VolterraK} and for every $t\in [0,T]$ let $\RKHS_t$ be the associated reproducing kernel Hilbert space. 
    
        Let $b:[0,T] \to \bR^d$ and suppose for any choice of $t\in [0,T]$ that
        \begin{equation}
            \label{eq:lemma:Existence_Q}
            \Big\| \int_0^\cdot b_s ds \Big\|_{\cH_t}< \infty. 
        \end{equation}
        Let $L:[0,T] \to \fWIC$ be the Volterra kernel from Assumption \ref{assumption:VolterraK}. Then the function 
        \begin{equation*}
            t\mapsto \int_0^t L(t, s) b_s ds 
        \end{equation*}
        is contained in $W^{1, 2}\big( [0,T]; \bR^d \big)$ the Sobolev space of functions with square integrable derivatives. We define $Q^{b} \in L^2\big( [0,T]; \bR^d \big)$ by
        \begin{equation}
            \label{eq:def-Q}
            Q_t^{b}:= \frac{d}{dt} \int_0^t L(t, s) b_s ds
            \quad \mbox{and}\quad
            \int_0^t Q_s^{b} ds = \int_0^t L(t, s) b_s ds. 
        \end{equation}
        Then
        \begin{equation*}
            \int_0^t b_s ds = \int_0^t K(t, s) Q_s^b ds
            \quad \mbox{and}\quad 
            \Big\| \int_0^\cdot b_s ds \Big\|_{\RKHS_t} = \bigg( \int_0^t \big| Q_s^b \big|^2 ds \bigg)^{\tfrac{1}{2}}
        \end{equation*}
    \end{proposition}

    \begin{proof}
        Firstly thanks to Lemma \ref{lemma:RKHS=I^*}, for any 
        \begin{equation*}
            h \in L^2\big( [0,T]; \bR^d \big) \quad \iff \quad \int_0^\cdot K(\cdot, s) h_s ds \in \RKHS_T
        \end{equation*}
        and by construction
        \begin{equation*}
            \sum_{i=1}^n a_i \1_{[t_i, t_{i+1}]} I_d \in \fWIC_T
        \end{equation*}
        so that
        \begin{equation}
            \label{eq:prop:Existence_Q:pf1}
            \Lambda_T\Big[ \sum_{i=1}^n a_i \1_{[t_i, t_{i+1}]} I_d, \int_0^{\cdot} K(\cdot, s) h_s ds \Big] = \sum_{i=1}^n a_i \bigg( \int_0^{t_{i+1}} K(t_{i+1}, s) h_s ds - \int_0^{t_i} K(t_i, s) h_s ds \bigg). 
        \end{equation}
        Assumption \ref{assumption:VolterraK} implies that 
        \begin{align*}
            &\Big\| \int_0^\cdot b_s ds \Big\|_{\RKHS_t}< \infty
            \quad \iff \quad
            \int_0^t b_s ds = \int_0^t K(t, s) h_s ds
            \quad \mbox{for some $h \in L^2\big( [0,T]; \bR^d \big)$ }
            \\
            &\mbox{and} \quad
            \Big\| \int_0^\cdot b_s ds \Big\|_{\RKHS_t}^2 = \int_0^t \big| h_s \big|^2 ds. 
        \end{align*}
        Therefore, our goal is to show that for $t$-Lebesgue almost everywhere that $h_t = Q_t^b$ as defined in Equation \eqref{eq:def-Q}. 
        
        More specifically, Equation \eqref{eq:prop:Existence_Q:pf1} means that
        \begin{align*}
            \Lambda_T \Big[ \sum_{i=1}^n a_i \1_{[t_i, t_{i+1}]} I_d, \int_0^{\cdot} b_s ds \Big] =& \sum_{i=1}^n a_i \int_{t_i}^{t_{i+1}} b_s ds
            =\int_0^T \Big( \sum_{i=1}^n a_i \1_{[t_i, t_{i+1}]} I_d \Big) b_s ds 
        \end{align*}
        and more generally that for every $t\in [0, T]$, 
        \begin{equation*}
            \Lambda_t\Big[ \sum_{i=1}^n a_i \1_{[t_i, t_{i+1}]} I_d, \int_0^{\cdot} b_s ds \Big] = \int_0^t \Big( \sum_{i=1}^n a_i \1_{[t_i, t_{i+1}]}(s) \Big) b_s ds. 
        \end{equation*}
        Taking appropriate limits on $\fWIC_t$, we conclude that
        \begin{equation*}
            \Lambda_t\Big[ L(t, \cdot), \int_0^{\cdot} b_s ds \Big] = \int_0^t L(t, s) b_s ds
        \end{equation*}
        and from Lemma \ref{lemma:abscont-Q} we know that
        \begin{equation*}
            t \mapsto \int_0^t L(t, s) b_s ds \quad \mbox{is absolutely continuous. }
        \end{equation*}
        Therefore, we define $Q^b:[0,T] \to \bR^d$ according to Equation \eqref{eq:def-Q} and observe that
        \begin{equation*}
            \int_0^t Q_s^b ds = \Lambda_t \Big[ L(t, \cdot), \int_0^\cdot b_s ds \Big]. 
        \end{equation*}
        Finally, thanks to Equation \eqref{eq:lemma:abscont-Q-pf1}, we observe that
        \begin{equation*}
            \Lambda_t \Big[ L(t, \cdot), \int_0^\cdot b_s ds \Big] = \int_0^t h_s ds
        \end{equation*}
        where $h \in L^2 \big([0,T]; \bR^d \big)$ was defined above. This leads us to our conclusion. 
    \end{proof}
    With Proposition \ref{prop:Existence_Q} established, we now provide a slightly more general definition for the element $Q^b$:
    \begin{definition}
        \label{definition:Q}
		Let $\big( \Omega, \cF, \bF, \bP \big)$ be a complete filtered probability space supporting an $\bF$-Brownian motion $W$. Let $K:[0,T] \to L^2\big( [0,T]; \lin(\bR^d, \bR^d) \big)$ be a Volterra kernel that satisfies Assumption \ref{assumption:VolterraK}. 
		
		Let $b:[0,T] \times \Omega \to \bR^d$ be progressively measurable and suppose that
        \begin{equation*}
            % t \mapsto \int_0^t L(t, s) b_s ds
            \bigg\| \int_0^{\cdot} b_s ds \bigg\|_{\RKHS_T} < \infty \quad \bP\mbox{-almost surely. }
        \end{equation*}

		Then we define $Q^b: [0,T] \times \Omega \to \bR^d$ by
		\begin{equation}
			\label{Q}
			Q^b_t := \frac{d}{dt} \int_0^t L(t, s) b_s ds. 
		\end{equation}
	\end{definition}
    
    \subsubsection*{Hilbert space projections}
    
    Assumption \ref{assumption:VolterraK} also allows us to define a project from the Hilbert space $\RKHS_T$ onto $\RKHS_t$ for any $t\in [0,T]$. 
    \begin{definition}
        \label{definition:RKHSprojection}
        Let $K:[0,T] \to L^2\big( [0,T]; \bR^d \big)$ be a Volterra kernel that satisfies Assumption \ref{assumption:VolterraK}. 

        Then for any $t\in [0,T]$, we define $\Pi_t: \RKHS_T \to \RKHS_t$ for $\dot{h}\in L^2\big([0,T]; \bR^d \big)$ by
        \begin{equation*}
            \Pi_t\Big[ \int_0^\cdot K(\cdot, r) \dot{h}_r dr \Big](s) = \int_0^{s} K(s, r) \1_{[0,t]}(r) \dot{h}_r dr
        \end{equation*}
    \end{definition}
    Such a projection exists because by Assumption we necessarily have that $\1_{[0,t]} \in \cV_T$ for any choice of $t\in [0,T]$. 
    \begin{remark}
    	It is worth emphasising here that for every choice of $t\in [0, T]$, the reproducing kernel Hilbert space $\RKHS_t \subseteq \cC_{0, T}^d$. In words, while $\RKHS_t$ is a smaller Hilbert space than $\RKHS_T$, each element of $\RKHS_t$ remains a function defined on the interval $[0,T]$. When the Gaussian Volterra process is chosen (trivially) to be Brownian motion, this point is lost because
	\begin{equation*}
	    \RKHS_t = \Big\{ \int_0^\cdot h_r \1_{[0,t]}(r) dr : h  \in L^{2}\big( [0,T]; \bR^d \big) \Big\}
	\end{equation*}
	so that any element of the reproducing kernel Hilbert space is constant for any value $s\in [t, T]$. This is emphatically not the case for other Gaussian Volterra processes. 
    \end{remark}
    
    These next results, which build on some of the techniques first used in \cite{Decreusfond1999Stochastic}, allow us to link the projection operator that arises naturally from our Volterra kernel satisfying Assumption \ref{assumption:VolterraK} with conditional expectations which provide a natural sense of projection on filtered probability spaces:
    \begin{lemma}        
        Let $(\Omega, \cF, \bF, \bP)$ be a filtered probability space carrying a Brownian motion. Let $K:[0,T] \to L^2\big( [0,T]; \lin(\bR^d, \bR^d) \big)$ be a Volterra kernel that satisfies Assumption \ref{assumption:VolterraK} and let $(Z_t)_{t\in [0, T]}$ be a Gaussian process of the form \eqref{eq:GaussVolterra}. Then
        \begin{equation*}
            \bE\bigg[ \exp\Big( \delta\big(h) - \tfrac{\|h\|_{\RKHS_T}^2}{2} \Big) \bigg| \cF_t^Z \bigg] = \exp\Big( \delta\big(\Pi_t[h]) - \tfrac{\|h\|_{\RKHS_t}^2}{2} \Big)
        \end{equation*}
        where $\delta$ is the Malliavin divergence and $\bF = (\cF_t^Z)_{t\in [0,T]}$ is the filtration generated by the Gaussian Volterra process \eqref{eq:GaussVolterra}. 
    \end{lemma}
    
    \begin{proof}
        Let $F$ be a smooth random variable that is $\cF_t^Z$-measurable, so that it can be expressed as 
        \begin{equation*}
            F = f\Big( Z_{t_{i_1}}, ..., Z_{t_{i_m}} \Big)
        \end{equation*}
        for some continuously differentiable function $f:\bR^m \to \bR$ that has at most polynomial growth and $t_{i_1}, ..., t_{i_m} \in [0,t]$. As we have that the random variables $Z_{t_{i_j}} = \delta\big( K(t_{i_j}, \cdot) \big)$ (where $\delta$ is the Malliavin divergence), we obtain that
        \begin{equation}
            \label{eq:sigma-alg1}
            \cF_t^Z = \sigma\Big( \delta\big( K(s, \cdot) \big): s\in [0,t] \Big) \otimes \cF_0^Z
            = \sigma\Big( \delta\big( h \big): h\in \RKHS_t \Big) \otimes \cF_0^Z. 
        \end{equation}
        Let $(h_n)_{n\in \bN}$ is an orthogonal basis of $\RKHS_t$. Equation \eqref{eq:sigma-alg1} means that any smooth random variable can be expressed as
        \begin{equation*}
            F = \hat{f}\Big( \delta(h_{i_1}), ..., \delta(h_{i_m}) \Big)
        \end{equation*}
        where $\hat{f}:\bR^m \to \bR$ is continuously differentiable that has at most polynomial growth. 
        
        Let $h\in \RKHS_T$. Applying the Cameron-Martin Theorem and using that
        \begin{equation*}
            \langle h_{i_j}, h \rangle_{\RKHS_T} = \langle h_{i_j}, \Pi_t[h] \rangle_{\RKHS_T} = \langle h_{i_j}, \Pi_t[h] \rangle_{\RKHS_t}, 
        \end{equation*}
        we obtain
        \begin{align*}
            \bE\bigg[ F \exp\Big( \delta(h) - \tfrac{\|h\|_{\RKHS_T}^2}{2} \Big) \bigg] =& \bE\bigg[ \hat{f}\big( \delta(h_{i_1}), ..., \delta(h_{i_m}) \big) \exp\Big( \delta(h) - \tfrac{\|h\|_{\RKHS_T}^2}{2} \Big) \bigg]
            \\
            =& \bE\bigg[ \hat{f}\Big( \delta(h_{i_1}) + \langle h_{i_1}, h \rangle_{\RKHS_T}, ..., \delta(h_{i_m}) + \langle h_{i_m}, h \rangle_{\RKHS_T} \Big) \bigg]
            \\
            =& \bE\bigg[ \hat{f}\Big( \delta(h_{i_1}) + \langle h_{i_1}, \Pi_t[h] \rangle_{\RKHS_T}, ..., \delta(h_{i_m}) + \langle h_{i_m}, \Pi_t[h] \rangle_{\RKHS_T} \Big) \bigg]
            \\
            =& \bE\bigg[ F \exp\Big( \delta\big( \Pi_t[h] \big) - \tfrac{\|h\|_{\RKHS_t}^2}{2} \Big) \bigg]. 
        \end{align*}
        Therefore
        \begin{equation*}
            \bE\bigg[ \bE\Big[ \exp\Big( \delta(h) - \tfrac{\|h\|_{\RKHS_T}^2}{2} \Big) \Big| \cF_t^Z \Big] F \bigg] = \bE\Big[ \exp\Big( \delta\big( \Pi_t[h] \big) - \tfrac{\|h\|_{\RKHS_t}^2}{2} \Big) F \Big]
        \end{equation*}
        and by density of smooth random variables that are $\cF_t^Z$-measurable in the space of square integrable random variables that $\cF_t^Z$-measurable, we obtain that
        \begin{equation*}
            \bE\bigg[ \exp\Big( \delta(h) - \tfrac{\|h\|_{\RKHS_T}^2}{2} \Big) \bigg| \cF_t^Z \bigg] = \exp\Big( \delta\big( \Pi_t[h] \big) - \tfrac{\|h\|_{\RKHS_t}^2}{2} \Big). 
        \end{equation*}
    \end{proof}
    
    \begin{lemma}
        \label{lemma:adaptedness=Mall}
        Let $(\Omega, \cF, \bF, \bP)$ be a filtered probability space carrying a Brownian motion. Let $K:[0,T] \to L^2\big( [0,T]; \lin(\bR^d, \bR^d) \big)$ be a Volterra kernel that satisfies Assumption \ref{assumption:VolterraK} and let $(Z_t)_{t\in [0, T]}$ be a Gaussian process of the form \eqref{eq:GaussVolterra}. 
        
        Let $F:\Omega\to \bR$ such that $F$ is Malliavin differentiable with Malliavin derivative
        \begin{equation*}
            DF \in L^2\big( \Omega, \bP; \RKHS_T \big). 
        \end{equation*}
        Then $F$ is $\cF_t^Z$-measurable if and only if $DF = \Pi_t\big[ DF \big]$. 
    \end{lemma}

    \begin{proof}
        First, we fix $t\in [0,T]$ and let $F$ be a smooth random variable such that $F$ is $\cF_t^Z$-adapted. Let $(h_n)_{n\in \bN}$ be an orthogonal basis of $\fWIC_t$ and let 
        \begin{equation*}
            \sigma^n = \sigma\Big( Z(h_i) : i=1, ..., n \Big) 
            \quad \mbox{where}\quad
            Z(h_i) = \int_0^T h_i(s) dZ_s. 
        \end{equation*}
        Then the sequence of $\sigma$-algebras satisfies that
        \begin{equation*}
            \sigma^n \subseteq \sigma^{n+1}
            \quad\mbox{and}\quad
            \cF_0^Z \otimes \bigcup_{n\in \bN} \sigma^n = \cF_t^Z. 
        \end{equation*}
        For each $i \in \{1, ..., n\}$, we have that $Z(h_i)$ is $\cF_t^Z$ measurable and we define
        \begin{equation*}
            F_n:= \bE\big[ F \big| \sigma^n \big] = f_n\Big( Z(h_1), ..., Z(h_n) \Big) 
        \end{equation*}
        where $f_n:\bR^{n} \to \bR$ is a smooth function. Then $F_n \to F$ is $L^2( \Omega, \bP; \bR)$ and the Malliavin derivative is
        \begin{align*}
            D_sF_n = \sum_{i=1}^n \frac{\partial f_n}{\partial x_i}\Big( Z(h_1), ..., Z(h_n) \Big) \scI_t[h_i](s)
            \quad s\in [0,t]
        \end{align*}
        where $\scI_t: \fWIC_t \to \RKHS_t$ is an isometric isomorphism. As $F \in \bD^{1, 2}$, we additionally have that $D_s F_n \to D_s F$ in $L^2(\Omega, \bP; \RKHS_T)$. 
        
        Thanks to Assumption \ref{assumption:VolterraK}, $\Pi_t\big[ \scI[h] \big] = \scI[h]$ for any choice of $h \in \fWIC_t$ and we obtain that
        \begin{align*}
            DF_n =& \sum_{i=1}^n \frac{\partial f}{\partial x_i}\Big( Z(h_1), ..., Z(h_n) \Big) \Pi_t\big[ \scI_t[h_i]\big]
            = \Pi_t\big[ DF_n \big]. 
        \end{align*}
        Taking the limit as $n\to \infty$, we obtain $DF = \Pi_t[ DF]$. 

        On the other hand, suppose that $F$ is a smooth random variable that satisfies $DF = \Pi_t[DF]$. Then, for any choice of $h\in \RKHS_T$ we have that $\bP$-almost surely
        \begin{equation}
            \label{eq:lemma:adaptedness=Mall-p1}
            F(\cdot + h) - F(\cdot) = \int_0^1 \Big\langle DF\big( \cdot + \lambda h \big), h \Big\rangle_{\RKHS_T} d\lambda. 
        \end{equation}
        We define
        \begin{equation*}
            \RKHS_t^* = \Big\{ h\in \RKHS_T: \Pi_t[h] = 0 \Big\}
        \end{equation*}
        so that $\RKHS_T = \RKHS_t \oplus \RKHS_t^*$ and we can express $\bP = \bP_t \times \bP_t^*$. By Assumption \ref{assumption:VolterraK} and Proposition \ref{prop:Existence_Q}, we can write 
        \begin{equation*}
            \RKHS_t^* = \bigg\{ \int_0^\cdot K(\cdot, s) h_s ds: h \in L^2\big( [t, T]; \bR^d \big) \bigg\}. 
        \end{equation*}
        In particular, for any $h_1\in \RKHS_t$ and $h_2 \in \RKHS_t^*$, we have that $\langle h_1, h_2 \rangle_{\RKHS_T} = 0$. Let $(h_n)_{n\in \bN}$ be an orthogonal basis of $\fWIC_t$ and let $(h_n')_{n\in \bN}$ be an orthogonal basis of $\fWIC_t^*$. We denote the $\sigma$-algebra
        \begin{equation*}
            \cF_t^* = \sigma\Big( Z(h'_n): n\in \bN \Big) 
        \end{equation*}
        so that any random variable $G$ that is $\cF_t^Z$-measurable must satisfy that 
        \begin{equation*}
            \bE\Big[ G \Big| \cF_t^* \Big] = \bE\Big[ G \Big]. 
        \end{equation*}
        Thus, for any $h\in \RKHS_t^*$, we have that $\bP$-almost surely
        \begin{equation*}
            \Big\langle DF, h \Big\rangle_{\RKHS_T} = \Big\langle \Pi_t \big[ DF \big], \Pi\big[ h \big] \Big\rangle_{\RKHS_T} = 0. 
        \end{equation*}
        Hence, by the Cameron-Martin formula, 
        \begin{equation*}
            \bE\bigg[ \Big\langle DF(\cdot + \lambda h), h \Big\rangle_{\RKHS_T} \bigg] = \bE\bigg[ \Big\langle DF, h \Big\rangle_{\RKHS_T} \exp\Big( \delta\big( \lambda h \big) - \tfrac{\|\lambda h\|_{\RKHS_T}^2}{2} \Big) \bigg] = 0
        \end{equation*}
        and we conclude that
        \begin{equation*}
            \bE\Big[ F(\cdot + h ) - F(\cdot) \Big] = \int_0^1 \bE\bigg[ \Big\langle DF(\cdot+\lambda h), h \Big\rangle_{\RKHS_T} \bigg] d\lambda = 0. 
        \end{equation*}
        Therefore, applying the Cameron-Martin Theorem again we obtain for any $h \in \RKHS_t^*$ that
        \begin{equation*}
            \bE\bigg[ F \exp\Big( \delta\big( h \big) - \tfrac{\| h\|_{\RKHS_T}^2}{2} \Big) \bigg] 
            = 
            \bE\Big[ F \Big]. 
        \end{equation*}
        In particular, 
        \begin{equation*}
            \bE\bigg[ \bE\Big[ F\Big| \cF_t^*\Big] \exp\Big( \delta\big( h \big) - \tfrac{\| h\|_{\RKHS_T}^2}{2} \Big) \bigg] 
            = 
            \bE\bigg[ \bE\Big[ F \Big] \exp\Big( \delta\big( h \big) - \tfrac{\| h\|_{\RKHS_T}^2}{2} \Big) \bigg]. 
        \end{equation*}
        As the collection of random variables
        \begin{equation*}
            \Big\{ \exp\big( \delta(h) - \tfrac{\|h\|_{\RKHS_T}^2}{2} \big): h\in \RKHS_t^* \Big\}
        \end{equation*}
        form a dense subset of the square integrable random variables generated by polynomials of the collection of random variables $\big\{ Z(h): h\in \fWIC_t^* \big\}$, we obtain that
        \begin{equation*}
            \bE\Big[ F\Big| \cF_t^*\Big] = \bE\Big[ F \Big] \quad \bP\mbox{-almost surely. }
        \end{equation*}
    \end{proof}
    
    \begin{proposition}
        \label{proposition:Martingale-Ust}
        Let $(\Omega, \cF, \bF, \bP)$ be a filtered probability space carrying a Brownian motion. Let $K:[0,T] \to L^2\big( [0,T]; \lin(\bR^d, \bR^d) \big)$ be a Volterra kernel that satisfies Assumption \ref{assumption:VolterraK} and let $(Z_t)_{t\in [0, T]}$ be a Gaussian process of the form \eqref{eq:GaussVolterra}. 
        
        Let $h \in L^2( \Omega, \bP; \RKHS_T)$ such that $h \in \mbox{Dom}(\delta)$ so that
        \begin{equation}
            \label{eq:proposition:Martingale-Ust2}
            \bE\Big[ \big\| h\big\|_{\RKHS_T}^2 \Big] < \infty. 
        \end{equation}
        Then the stochastic process $t\mapsto \delta\big( \Pi_t(h) \big)$ is an $\bF^Z=(\cF_t^Z)_{t\in [0,T]}$-martingale and
        \begin{equation}
            \label{eq:proposition:Martingale-Ust1}
            \bE\Big[ \big| \delta\big( \Pi_t(h) \big) \big|^2 \Big] = \bE\Big[ \big\| h \big\|_{\cH_t}^2 \Big]. 
        \end{equation}
    \end{proposition}

    \begin{proof}
        First of all, we want to show that $\delta\big( \Pi_t[h] \big)$ is $\cF_t^Z$-measurable. To see this, let $v\in\RKHS_T$ such that $\Pi_t[v] = 0$ and observe that
        \begin{align*}
            \Big\langle D \Big( \delta\big( \Pi_t(h) \big) \Big), v \Big\rangle =& \Big\langle \Pi_t\big[ h \big], v \Big\rangle_{\RKHS_T} + \delta\Big( D^v \big( \Pi_t[h] \big) \Big)
            \\
            =& 0 + \delta(0) = 0 \quad \bP\mbox{-almost surely. }
        \end{align*}
        Hence
        \begin{equation*}
            D\Big( \delta\big( \Pi_t(h) \big) \Big) = \Pi_t\bigg[ D\Big(\delta\big( \Pi_t(h) \big) \Big) \bigg]
        \end{equation*}
        so that by Lemma \ref{lemma:adaptedness=Mall} the random variable $\delta\big( \Pi(h) \big)$ is $\cF_t^Z$-measurable. 

        Next, observe that for any smooth random variable $F \in \cS$ such that $F$ is $\cF_t^Z$-measurable, we can apply \cite{nualart2006malliavin}*{Proposition 1.3.1} to get for any $v\in \RKHS_T$ such that $\Pi_t[v] = 0$ that
        \begin{align*}
            \bE\Big[ \bE\big[ \delta(h) \big| \cF_t \big] F \Big] 
            =& 
            \bE\Big[ \delta(h) F \Big]
            = 
            \bE\Big[ \big\langle h, DF \big\rangle_{\RKHS_T} \Big]
            \\
            =& \bE\Big[ \big\langle h, D\Pi_t[F] \big\rangle_{\RKHS_T} \Big]
            =
            \bE\Big[ \big\langle h, \Pi_t[DF] \big\rangle_{\RKHS_T} \Big] 
            \\
            =& \bE\Big[ \big\langle \Pi_t[h], \Pi_t[DF] \big\rangle_{\RKHS_T} \Big]
            = 
            \bE\Big[ \delta\big( \Pi_t[h] \big) \cdot F \Big]. 
        \end{align*}
        By density of smooth random variables, we conclude that for any $t\in[0,T]$, 
        \begin{equation*}
            \delta\big( \Pi_t(h) \big) = \bE\Big[ \delta(h) \Big| \cF_t^Z \Big]
        \end{equation*}
        and thus $t\mapsto \delta\big( \Pi_t(h) \big)$ is an $\cF_t^Z$-martingale. To compute the second moments, let $F\in \cS$ be a smooth random variable that is $\cF_s^Z$-measurable for some $s\in [0,t)$ and notice that
        \begin{align*}
            \bE\bigg[ \Big( \big| \delta\big( \Pi_t(h) \big)\big|^2 - \big\| h \big\|_{\cH_t}^2 \Big) F \bigg] =& \bE\bigg[ D^{\Pi(h)} \Big( \delta\big( \Pi_t(h) \big) \cdot F \Big) \bigg] - \bE\Big[ \big\| h \big\|_{\RKHS_t}^2 F \Big]
            \\
            =& \bE\Big[ \big( D^{\Pi_t(h)} F \big) \cdot \delta\big( \Pi_t(h) \big) \Big] + \bE\Big[ \delta\big( D^{\Pi_t(h)} Pi_t(h) \big) \cdot F \big] - \bE\Big[ \big\| h \big\|_{\RKHS_t}^2 F \Big]
            \\
            =& \bE\Big[ D^{\Pi_t(h)} F \cdot \delta \big( \Pi_t(h) \big) \Big] + \bE\Big[ \delta\big( D^{\Pi_t(h)} \Pi_t(h) \big) \cdot F \Big]
            \\
            =& \bE\Big[ D^{\Pi_t(h)} \circ D^{\Pi_t(h)} F \Big] + \bE\bigg[ \Big\langle D^{\Pi_t(h)} \Pi_t(h) , DF \Big\rangle_{\RKHS_T} \bigg]
        \end{align*}
        Thanks to Lemma \ref{lemma:adaptedness=Mall}, since $F$ is $\cF_s^Z$-measurable we have that
        \begin{equation*}
            \Big\langle DF, v \Big\rangle_{\RKHS_T} = \Big\langle DF, \Pi_s[v] \Big\rangle_{\RKHS_T}
        \end{equation*}
        so that
        \begin{align*}
            \bE\bigg[ \Big( \big| \delta\big( \Pi_t(h) \big)\big|^2 - \big\| h \big\|_{\cH_t}^2 \Big) F \bigg] =& \bE\Big[ D^{\Pi_s(h)} \circ D^{\Pi_s(h)} F \Big] + \bE\bigg[ \Big\langle D^{\Pi_s(h)} \Pi_s(h) , DF \Big\rangle_{\RKHS_T} \bigg]. 
        \end{align*}
        Now, reversing the argument implies that for any $F\in \cS$ such that $F$ is $\cF_s^Z$-measurable, we obtain
        \begin{equation*}
            \bE\bigg[ \Big( \big| \delta\big( \Pi_t(h) \big) \big|^2 - \big\| h \big\|_{\RKHS_t}^2 \Big) F \bigg] = \bE\bigg[ \Big( \big| \delta\big( \Pi_s(h) \big)\big|^2 - \big\| h \big\|_{\cH_s}^2 \Big) F \bigg]
        \end{equation*}
        and we conclude that
        \begin{equation*}
            \bE\Big[ \big| \delta\big( \Pi_t(h) \big)\big|^2 - \big\| h \big\|_{\cH_t}^2 \Big| \cF_s^Z \Big] = \bE\Big[ \big| \delta\big( \Pi_s(h) \big) \big|^2 - \big\| h \big\|_{\cH_s}^2 \Big]. 
        \end{equation*}
        From this, we conclude Equation \eqref{eq:proposition:Martingale-Ust1}. 
    \end{proof}
    We conclude that the Dol\'eans-Dade exponential of the martingale constructed in Proposition \ref{proposition:Martingale-Ust}
    \begin{equation*}
        t\mapsto \exp\bigg( \delta\Big( \Pi_t(h) \Big) - \tfrac{\| h \|_{\cH_t}}{2} \bigg)
    \end{equation*}
    is a local martingale under the assumption that $h \in \mbox{Dom}(\delta)$ and satisfies Equation \eqref{eq:proposition:Martingale-Ust2}. 

    \subsection{Secure local non-determinism and the fundamental martingale}
    \label{subsection:SLND+fm}

    We wish to explore a framework for the \emph{reproducing kernel Hilbert space} that is equivalent to Assumption \ref{assumption:VolterraK} (upon which the results of this work rely). Following on from Section \ref{subsection:HeuristicFundMart} where we saw that the invertibility of each covariance matrix implied the existence of a fundamental martingale and that this was equivalent to \emph{local non-determinism} as defined in Equation \eqref{eq:lemma:LND1}. In this section, we will explore how these concepts adapt to the continuous time setting:
    \begin{definition}
        \label{definition:Filtration-Hilbert}
        Let $(I, \leq)$ be a totally ordered set and suppose that for every $i \in I$ we have that $\big(\cH_i, +, \langle \cdot, \cdot\rangle_{i} \big)$ is a Hilbert space over a common field $\bF$. We say that $(\cH_i)_{i\in I}$ is a \emph{filtration of Hilbert spaces} if
        \begin{equation*}
            \mbox{for every $i\leq j$}\quad \cH_i \subseteq \cH_j \quad \mbox{and}\quad \forall f,g \in \cH_i \quad \langle f, g \rangle_j = \langle f, g \rangle_i. 
        \end{equation*}

        Two Hilbert space filtrations $(\cH_i)_{i\in I}$ and $(\cG_i)_{i \in I}$ are said to be isomorphic if for every $j \in I$ there exists a Hilbert space isomorphism
        \begin{equation*}
            \Psi^j: \big( \cH_j, \langle \cdot, \cdot\rangle_j \big)
            \to
            \big(\cG_j, \langle, \cdot, \cdot \rangle_j \big)
            \quad\mbox{such that}
            \quad \forall i \leq j \quad
            \Psi^j\big|_{\cH_i} = \Psi^i. 
        \end{equation*}
    \end{definition}
    We contrast Definition \ref{definition:Filtration-Hilbert} with the concept of a filtration of $\sigma$-algebras. The purpose of the $\sigma$-algebra is to describe the event space, so that a filtration of $\sigma$-algebras allows mathematicians to capture the ordered nature of events. When working with Gaussian processes, all information necessary to compute the probability of events can be derived from the reproducing kernel Hilbert space so that we only need to keep track of the time-ordered nature of covariances. 
    \begin{example}
        \label{example:HilbertFiltration1}
        Let $\big( [0,T], \leq \big)$ be the totally ordered compact time interval and for every $t\in [0,T]$ consider the vector space
        \begin{equation*}
            L^2\big( [0,T]; \bR \big) \supseteq \Big\{ f \cdot \1_{[0,t]}: f \in L^2\big( [0,T]; \bR\big) \Big\} := L_{t, T}^2(\bR)
        \end{equation*}
        Pair each of these vector spaces with the inner product
        \begin{align*}
            \Big\langle \cdot, \cdot\Big\rangle_t: L^2\big([0,T]; \bR\big) \times L^2\big([0,T]; \bR\big) \to \bR,
            \\
            \Big\langle f,g \Big\rangle_t = \int_0^T \Big( f_s \cdot \1_{[0,t]}(s) \Big) \cdot \Big( g_s \cdot \1_{[0,t]}(s) \Big)ds. 
        \end{align*}
        Then
        \begin{equation*}
            \Big( L_{t, T}^2, \big\langle \cdot, \cdot \big\rangle_t \Big)_{t\in [0,T]}
        \end{equation*}
        is an example of a filtration of Hilbert spaces. 

        Next, for every $t\in [0,T]$ the reproducing kernel Hilbert space for fractional Brownian motion $\RKHS_t$ as expressed in Example \ref{example:fbm} with the inner product
        \begin{equation*}
            \bigg\langle \int_0^\cdot K(\cdot, s) f_s ds, \int_0^{\cdot} K(\cdot, s) g_s ds \bigg\rangle_{t} = \int_0^t \langle f_s, g_s\rangle_{\bR^d} ds
        \end{equation*}
        form a Hilbert space filtration. Finally, the Hilbert space filtrations
        \begin{equation*}
            \Big( L_{t, T}^2(\bR), \big\langle \cdot, \cdot \big\rangle_t \Big)_{t\in [0,T]}
            \quad \mbox{and}\quad
            \Big( \RKHS_t, \big\langle \cdot, \cdot \big\rangle_{t} \Big)_{t\in [0,T]}
            \quad \mbox{are isomorphic.}
        \end{equation*}
    \end{example}

    \begin{example}
        \label{example:HilbertFiltration2}
        Let $(Z_t)_{t\in [0,T]}$ be a Gaussian process with covariance function $R:[0,T]^{\times 2} \to \bR$ defined by 
        \begin{equation*}
            R(t, s) = \bE\Big[ \big\langle Z_t, Z_s\big\rangle_{\bR^d} \Big]
        \end{equation*}
        Recall that for any $t\in [0,T]$ the collection of \emph{reproducing kernel Hilbert spaces}, the closure of the linear span
        \begin{align*}
            &\RKHS_t = \spn\bigg\{ \bE\Big[ f(Z) Z_{\cdot} \Big]: f \in (\cC_t^d)^* \bigg\}
            \\
            &\mbox{with inner product}\quad
            \Big\langle \bE\big[ f(Z) Z_{\cdot} \big], \bE\big[ g(Z) Z_{\cdot} \big] \Big\rangle_{\RKHS_t} = \bE\Big[ f(Z) \cdot g(Z) \Big]
        \end{align*}
        form a filtration. Similarly, the collection of \emph{first Wiener-Ito chaos}, the closure of the linear span
        \begin{align*}
            &\fWIC_t:=\spn\bigg\{ \1_{[0,s]}(\cdot) e_{i,i}: s\in [0,t], i\in \{1, ..., d\} \bigg\}
            \\
            &\mbox{with inner product}\quad
            \Big\langle \1_{[0,s]} e_{i,i}, \1_{[0,r]} e_{j,j} \Big\rangle_{\fWIC_t} = \bE\Big[ \langle Z_s, e_i\rangle \cdot \langle Z_r, e_j\rangle \Big]. 
        \end{align*}
        form a filtration. Further, the two Hilbert space filtrations
        \begin{equation*}
            \Big( \RKHS_t, \langle \cdot, \cdot \rangle_{\RKHS_t} \Big)_{t\in [0,T]} 
            \quad \mbox{and}\quad
            \Big( \fWIC_t, \langle \cdot, \cdot \rangle_{\fWIC_t} \Big)_{t\in [0,T]} 
            \quad \mbox{are isomorphic}
        \end{equation*}
        due to the Hilbert space isomorphism $\scI: \fWIC_T \to \RKHS_T$ satisfying that $\scI\big|_{\fWIC_t} = \RKHS_t$. 
    \end{example}

    \begin{proposition}
        \label{prop:LND-filtration}
        Let $K:[0,T] \to L^2\big( [0,T]; \lin(\bR^d, \bR^d) \big)$ be a Volterra kernel and let $\RKHS_T$, $\fWIC_T$ and $\cV_T$ be the isometrically isomorphic Hilbert spaces defined in Equation \eqref{eq:CovarianceCommute}. 

        Suppose that for every $t\in [0, T]$
        \begin{equation}
            \label{eq:prop:LND-filtration-H}
            \RKHS_t = \bigg\{ \int_0^\cdot K(\cdot, s) f_s ds: f\in L^2\big([0,t]; \bR^d \big) \bigg\} \subseteq \cC_{0,T}^{d}. 
        \end{equation}
        Then for every $t\in [0,T]$, there exists $L(t, \cdot) \in \fWIC_t$ such that
        \begin{equation}
            \label{eq:prop:LND-filtration}
            \scJ^*\big[ L(t, \cdot) \big](s) = \1_{[0,t]}(s) I_d. 
        \end{equation}
        In particular, a Volterra kernel $K:[0,T] \to L^2\big( [0,T]; \lin(\bR^d, \bR^d) \big)$ satisfies Assumption \ref{assumption:VolterraK} if and only if the Hilbert space filtration
        \begin{equation*}
            \Big( \RKHS_t, \big\langle \cdot, \cdot \big\rangle_{\RKHS_t} \Big)_{t\in [0,T]}
            \quad \mbox{is isomorphic to} \quad
            \Big( L_{t, T}^2(\bR^d), \big\langle \cdot, \cdot \big\rangle_t \Big)_{t\in [0,T]}
        \end{equation*}
    \end{proposition}
    
    \begin{proof}
        First, we suppose that the Hilbert space filtrations
        \begin{equation*}
            \Big( \RKHS_t, \langle \cdot, \cdot \rangle_{\RKHS_t} \Big)_{t\in [0,T]}
            \quad \mbox{is isomorphic to} \quad
            \Big( L_{t, T}^2(\bR^d), \big\langle \cdot, \cdot \big\rangle_t \Big)_{t\in [0,T]}. 
        \end{equation*}
        Next, we remark that the Hilbert space filtration
        \begin{equation*}
            \Big( \fWIC_t, \langle \cdot, \cdot \rangle_{\fWIC_t} \Big)_{t\in [0,T]} = \bigg( \spn\Big\{ \1_{[0,s]} I_d: s \in [0,t] \Big\}, \big\langle \cdot, \cdot \big\rangle_{\fWIC_t} \bigg)_{t\in [0,T]}
        \end{equation*}
        is isomorphic to $\Big( \RKHS_t, \langle \cdot, \cdot \rangle_{\RKHS_t} \Big)_{t\in [0,T]}$ thanks to the Hilbert space isomorphism $\scI_t:\fWIC_t \to \RKHS_t$ defined by
        \begin{equation*}
            \scI_T\Big[ \1_{[0,s]} I_d \Big](r) = R(s, r)
        \end{equation*}
        satisfying that $\scI_T\big|_{\fWIC_t} = \scI_t$. Hence, by the transitivity of isomorphisms, 
        \begin{equation*}
            \Big( \fWIC_t, \langle \cdot, \cdot, \rangle_{\fWIC_t} \Big)_{t\in [0,T]} 
            \quad\mbox{is isomorphic to}\quad
            \Big( L_{t, T}^2(\bR^d), \big\langle \cdot, \cdot \big\rangle_t \Big)_{t\in [0,T]}.
        \end{equation*}
        Hence, for every $t \in [0,T]$, there exists an element $L(t, \cdot) \in \fWIC_T$ such that Equation \eqref{eq:prop:LND-filtration} is satisfied. 

        Lemma \ref{lemma:RKHS=I^*} proves the reverse implication and we conclude. 
    \end{proof}

    \begin{example}
        The filtration $(\RKHS_t)_{t\in [0,T]}$ associated to the reproducing kernel Hilbert space of a Brownian bridge with termination time $T$ is not isomorphic to
        \begin{equation*}
            \Big( L_{t, T}^2(\bR^d), \big\langle \cdot, \cdot \big\rangle_t \Big)_{t\in [0,T]}
        \end{equation*}
        To see this, first note that
        \begin{align*}
            &\RKHS_{T}^T = \bigg\{ \int_0^\cdot f_s ds: f \in L^2\big( [0,T]; \bR^d \big) 
            \quad \mbox{and}\quad
            \int_0^T f_s ds = 0 \bigg\}
            \\
            &\mbox{with inner product}\quad
            \Big\langle \int_0^{\cdot} f_s ds, \int_0^{\cdot} g_s ds \Big\rangle_{\RKHS_T^T} = \int_0^T \big\langle f_t , g_t \big\rangle_{\bR^d} dt
        \end{align*}
        On the other hand, the reproducing kernel Hilbert space generated by the Brownian bridge with termination time T running over the sub-interval $[0,t]$ where $t<T$ is
        \begin{equation*}
            \RKHS_t^T = \bigg\{ \int_0^{\cdot\wedge t} f_s ds: f \in L^2\big( [0,T]; \bR^d \big) \bigg\}
            \quad \mbox{with}\quad
            \Big\langle \int_0^{\cdot\wedge t} f_s ds, \int_0^{\cdot\wedge t} g_s ds \Big\rangle_{\RKHS_t^T} = \int_0^t \big\langle f_s , g_s \big\rangle_{\bR^d} ds. 
        \end{equation*}
        Therefore, the Hilbert space filtration
        \begin{equation*}
            \Big( \cH_t^T, \big\langle \cdot, \cdot \big\rangle_{\RKHS_t^T} \Big)_{t\in [0, T)}
            \quad\mbox{is isomorphic to}\quad
            \Big( L^2\big( [0,T]; \bR^d \big), \langle \cdot, \cdot\rangle \Big)_{t\in [0,T)}
            \quad \mbox{with isomorphism } \Psi_t,
        \end{equation*}
        but it is not the case that under the associated mapping
        \begin{equation*}
            \Psi_T\Big[ \RKHS_T^T, \langle \cdot, \cdot\rangle_{\RKHS_T^T} \Big] 
            \quad \mbox{is not equal to}\quad
            \Big( L^2([0,T]; \bR^d), \langle \cdot, \cdot \rangle \Big).
        \end{equation*}
        In particular, there is no way to bijectively transform a Brownian bridge into a Brownian motion without enhancing the underlying $\sigma$-algebra. 
    \end{example}

    \begin{theorem}
        \label{thm:LND-filtration}
        Let $(Z_t)_{t\in [0,T]}$ be a $d$-dimensional Gaussian process with covariance $R:[0,T]^{\times 2} \to \bR$. Let
        \begin{align*}
            &\RKHS_t:= \spn \Big\{ \bE\big[ \langle Z_s, u\rangle_{\bR^d} \cdot Z_{\cdot} \big]: s\in [0,t], u\in \bR^d \Big\}
            \\
            &\mbox{with inner product}\quad
            \Big\langle \bE\big[ \langle Z_s, u\rangle_{\bR^d} \cdot Z_{\cdot} \big], \bE\big[ \langle Z_r, v\rangle_{\bR^d} \cdot Z_{\cdot} \big] \Big\rangle_{\RKHS_t} = \bE\Big[ \langle Z_s, u\rangle_{\bR^d} \cdot \langle Z_r, v\rangle_{\bR^d} \Big]. 
        \end{align*}
        Then the following are equivalent:
        \begin{enumerate}
            \item 
            \label{enum:thm:LND-filtration-1}
            The Hilbert space filtration
            \begin{equation*}
                \Big( \RKHS_t, \langle \cdot, \cdot \rangle_{\RKHS_t} \Big)_{t\in [0,T]} 
                \quad\mbox{is isomorphic to} \quad
                \Big( L_{t, T}^2(\bR^d), \big\langle \cdot, \cdot \big\rangle_t \Big)_{t\in [0,T]}; 
            \end{equation*}
            \item 
            \label{enum:thm:LND-filtration-2}
            There exists a Volterra kernel $K:[0,T] \to L^2\big( [0,T]; \lin(\bR^d, \bR^d) \big)$ such that
            \begin{equation*}
                Z_t = \int_0^t K(t, s) dW_s
            \end{equation*}
            and $K$ satisfies Assumption \ref{assumption:VolterraK};
        \end{enumerate}
    \end{theorem}

    \begin{proof}
        Proposition \ref{prop:LND-filtration} proves that \ref{enum:thm:LND-filtration-2} implies \ref{enum:thm:LND-filtration-1} so we focus on proving \ref{enum:thm:LND-filtration-1} implies \ref{enum:thm:LND-filtration-2}. 

        Therefore, we start by assuming that the filtration of reproducing kernel Hilbert spaces satisfies \ref{enum:thm:LND-filtration-1}. Following on from Example \ref{example:HilbertFiltration2}, we conclude that this is equivalent to 
        \begin{equation*}
            \Big( \fWIC_t, \langle \cdot, \cdot \rangle_{\fWIC_t} \Big)_{t\in [0,T]} 
                \quad\mbox{being isomorphic to} \quad
                \Big( L_{t, T}^2(\bR^d), \big\langle \cdot, \cdot \big\rangle_t \Big)_{t\in [0,T]}. 
        \end{equation*}
        In particular, this means that for every $t\in [0,T]$ there exists $L(t,\cdot) \in \fWIC$ such that
        \begin{equation*}
            W_t^*:=\int_0^t L(t, s) dZ_s
            \quad\mbox{and}\quad
            \bE\Big[ \big\langle W_t^*,  W_s^* \big\rangle_{\bR^d} \Big] = d \cdot \Big\langle \1_{[0,t]} , \1_{[0,s]} \Big\rangle_{L^2\big([0,T]; \bR\big)}
        \end{equation*}
        and we conclude that the Gaussian process $(W^*_t)_{t\in [0,T]}$ has the same covariance as a Brownian motion. 

        Similarly, we also obtain that for every $t\in [0,T]$ there exists a $K(t, \cdot) \in L^2\big( [0,T]; \lin(\bR^d, \bR^d) \big)$ such that
        \begin{equation*}
            \bE\Big[ \big\langle Z_t, Z_s \big\rangle_{\bR^d} \Big] = R(t, s) = \int_0^T \Big\langle K(t, r), K(s, r) \Big\rangle_{\lin(\bR^d, \bR^d)} dr. 
        \end{equation*}
        Further, since $\1_{[0,t]} \in \fWIC_t$ and the restriction of $\1_{[0,t]}$ to $\fWIC_s$ is just $\1_{[0,s]}$, we conclude that
        \begin{equation*}
            K(t, \cdot) \in L^2\big( [0,t]; \lin(\bR^d, \bR^d) \big)
            \quad \mbox{and}\quad
            \int_t^T K(t, s) ds = 0.
        \end{equation*}
        Hence, we conclude that $K$ is a Volterra kernel. 

        As $W$ is a Brownian motion, we define the new Gaussian process $(Z_t^*)_{t\in [0,T]}$ where
        \begin{equation*}
            Z_t^*: = \int_0^t K(t, s) dW_s^*. 
        \end{equation*}
        By direct calculation, we conclude that 
        \begin{equation*}
            \bE\Big[ \big\langle Z_t^*, Z_t^* \big\rangle_{\bR^d} \Big] = R(t, s)
        \end{equation*}
        so that $(Z_t^*)_{t\in [0,T]}$ has the same covariance as the Gaussian process $(Z_t)_{t\in [0,T]}$. Hence, we conclude that $(Z_t)_{t\in [0,T]}$ admits a Volterra representation. 
    \end{proof}
    
    \begin{definition}
    	\label{definition:SLND}
    	We say that a Gaussian measure $\gamma$ associated to the abstract Wiener space 
        \begin{equation*}
            (\cC_{0, T}^d, \RKHS_T, \BSi_T)
        \end{equation*}
        is \emph{securely locally non-deterministic} if the filtration of Hilbert spaces
    	\begin{equation*}
    		\Big( \RKHS_t, \big\langle \cdot, \cdot \big\rangle_{t} \Big)_{t\in [0,T]}
    		\quad \mbox{is isomorphic to}\quad
    		\Big( L_{t, T}^2(\bR^d), \big\langle \cdot, \cdot \big\rangle_t \Big)_{t\in [0,T]}. 
    	\end{equation*}
    \end{definition}
    
    \begin{remark}
    	The authors endeavoured to find a direct connection between Definition \ref{definition:SLND} and more common definitions of \emph{Local Non-determinism} such as for every $s\in [0,T]$ and $t\in (s, T]$, 
    	\begin{equation}
    		\label{eq:rem:SLND}
    		\bE\bigg[ \Big( Z_t - \bE\big[ Z_t\big| \cF_s\big] \Big)^2 \bigg]>0. 
    	\end{equation}
    	For a more detailed exploration of some non-equivalent definitions of \emph{local non-determinism}, we refer the reader to \cite{Xiao2006Properties}. We conjecture that \eqref{eq:rem:SLND} is not equivalent to \ref{definition:SLND}, but we leave this as an open problem. The \emph{local non-determinism} of fractional Brownian motion has found a lot of interest recently in the context of regularisation by noise, see for example \cites{Galeati2022Prevalence, Galeati2022Solution}. These techniques rely on a scaling property for estimates of the form \eqref{eq:rem:SLND} which is quite different from our setting. None the less, we do feel that it is worth exploring this connection in more detail. 
    \end{remark}

    \subsection{Girsanov's Theorem}
    \label{subsection:Girsanov}

    Having established a class of Gaussian processes that can be transformed via a Volterra kernel into a Brownian motion bijectively, our next goal is to understand how the existence of the fundamental martingale allows us to prove additional properties for such Gaussian measures using martingale techniques that would otherwise only apply to the law of Brownian motion. 
    
    The following classical result can be found in \cite{bogachev1998gaussian}:
    \begin{theorem}[Cameron Martin Theorem]
        \label{theorem:Cameron-Martin}
        Let $(\cX, \cB, \gamma)$ be a probability space with a Gaussian measure and let $\RKHS$ be the reproducing kernel Hilbert space. Then for any $h\in \RKHS$, 
        \begin{equation*}
            \int_{\cX} F(x + h) d\gamma(x) = \int_{\cX} F(x) \exp\Big( \delta(h) - \tfrac{\|h\|_{\RKHS}^2}{2} \Big) d\gamma(x). 
        \end{equation*}
    \end{theorem}
    Theorem \ref{theorem:Cameron-Martin} allows us to study the linear translation of Gaussian measures, in this case by a Hilbert space element $h$. However, it is much more practical to want to consider non-linear transformations of a Gaussian measure. We refer the interested reader to \cite{nualart2006malliavin}*{Chapter 4} for an impressive collection of results on the translation of the Wiener measure. Our focus is instead on establishing a Girsanov-type result for Gaussian processes that are securely locally non-deterministic. 

    For every $t\in [0, T]$, we denote $\fP_t: \cC_T^d \to \cC_t^d$ to be the canonical projection and denote the $\sigma$-algebra filtration on $\cC_T^d$ by
    \begin{equation*}
        \cF_t:= \sigma\Big( (\fP_t)^{-1}[A]: A \in \cC_t^d \Big). 
    \end{equation*}
    
    We now come to the first main result of this work:
    \begin{theorem} 
        \label{theorem:ap:girsanov}
        Let $K:[0,T] \to L^2\big( [0,T]; \lin(\bR^d, \bR^d) \big)$ be a Volterra kernel that satisfies Assumption \ref{assumption:VolterraK} and let $(\RKHS_t)_{t\in [0,T]}$ be a filtration of Hilbert spaces defined as in Equation \eqref{eq:prop:LND-filtration-H}. Let $\BSi_t:\RKHS_t \to \cC_{0,T}^d$ be a compact embedding and let $\gamma$ be the Gaussian measure associated with the abstract Wiener space $(\cC_{0, T}^d, \RKHS_T, \BSi_T)$. Let $P_0 \in \cP_2( \bR^d )$ and let denote $P^* = P_0 \times \gamma$. 
        
        Let $b : [0,T] \times \cC_T^d \to \bR^d$ be progressively measurable and suppose that 
        \begin{equation*}
            \Big\| \int_0^\cdot b\big( s, X[s] \big) ds \Big\|_{\RKHS_T}< \infty \quad \mbox{$P^*$-almost surely. }
        \end{equation*}
        Further, suppose that on the filtered probability space $\big( \cC_T^d, \cB(\cC_T^d), (\cF_t)_{t\in [0,T]}, P^* \big)$ the process
        \begin{equation}
            \label{eq:stochasticExpon}
            (t, X) \mapsto \cZ_t\Big[ \int_0^\cdot b\big( s, X[s] \big) ds \Big] := \exp\Bigg( \delta\bigg( \Pi_t\Big[ \int_0^{\cdot} b\big(s, X[s] \big) ds \Big] \bigg) - \tfrac{1}{2} \Big\| \int_0^\cdot b\big(s, X[s] \big) ds \Big\|_{\RKHS_t}^2 \Bigg)
        \end{equation}
        (where $\delta$ is the Malliavin divergence) is a martingale that satisfies that
        \begin{equation}
            \label{eq:theorem:ap:girsanov-Martingale}
            \bE^{P^*}\bigg[ \cZ_T\Big[ \int_0^\cdot b\big( s, X[s] \big) ds \Big] \bigg] = 1. 
        \end{equation}
        Let $P$ be the probability measure defined by
        \begin{equation*}
            \frac{dP}{dP^*} \bigg|_{\cF_t} = \cZ_t\Big[ \int_0^\cdot b\big( s, X[s] \big) ds \Big]. 
        \end{equation*}
        Then the law of the process
        \begin{equation*}
            X_t - \int_0^t b\big( s, X[s] \big) ds \quad\mbox{under $P$}
        \end{equation*}
        is the same as the law of the canonical process $X$ under $P^*$. 
    \end{theorem}

    Notice that Theorem \ref{theorem:ap:girsanov*} follows as an example of Theorem \ref{theorem:ap:girsanov}. Further, note that the statement of Theorem  \ref{theorem:ap:girsanov*} relies on the projection mapping $\Pi$ defined in Definition \ref{definition:RKHSprojection}. Without the Volterra kernel $K$ satisfying Assumption \ref{assumption:VolterraK}, this projection mapping would not be 'rich enough' for this result to hold. 
    \begin{proof}[Proof of Theorem \ref{theorem:ap:girsanov}]
        Let us start by observing that since for every $t \in [0,T]$ we have that
        \begin{equation*}
            \int_0^\cdot b(s, X[s]) ds \in \RKHS_t \quad\mbox{for $P^*$-almost everywhere,}
        \end{equation*}
        we can write
        \begin{align*}
            &\int_0^{\cdot \wedge t} K(\cdot, s) Q_s^b\big(X[s] \big) ds = \Pi_t\Big[ \int_0^\cdot b\big(s, X[s] \big) ds \Big] 
            \quad \mbox{and}
            \\
            &\Big\| \int_0^\cdot K(\cdot, s) Q_s^b\big(X[s] \big) ds \Big\|_{\RKHS_t} = \bigg\| \Pi_t\Big[ \int_0^\cdot b\big(s, X[s] \big) ds \Big] \bigg\|_{\RKHS_T}
        \end{align*}
        Therefore, to align notation with the previous results we henceforward write $h(X) = \int_0^\cdot b(s, X[s]) ds$. Thus Equation \eqref{eq:stochasticExpon} can be rewritten as
        \begin{equation*}
            t\mapsto \cZ_t[h] = \exp\Big( \delta\big( h\big) - \tfrac{\|h\|_{\RKHS_t}}{2} \Big)
        \end{equation*}
        and thanks to Equation \eqref{eq:theorem:ap:girsanov-Martingale} we conclude that $\cZ_t[h]$ is a $P^*$-martingale. 

        Thanks to Assumption \ref{assumption:VolterraK}, the process
        \begin{equation*}
            t\mapsto W_t^*(X):= X_0 + \int_0^t L(t, s) dX_s
        \end{equation*}
        is a Brownian motion so that under $P$ the process
        \begin{equation*}
            t\mapsto W_t^*(X) - \int_0^t Q_s^b\big( X[s] \big) ds
        \end{equation*}
        is a Brownian motion. Finally, thanks to Proposition \ref{prop:Existence_Q}, we conclude that the process
        \begin{equation*}
            t\mapsto X_t - X_0 - \int_0^t b\big( s, X[s] \big) ds = \int_0^t K(t, s) dW_s^* - \int_0^t K(t, s) Q_s^b\big( X[s] \big) ds
        \end{equation*}
        is a fractional Brownian motion under $P$. 
    \end{proof}

    In practice, proving that the local martingale defined in Equation \eqref{eq:stochasticExpon} is in fact a martingale is as challenging as it is in the classical setting. In the next result, we provide an adaption of Novikov's famous condition from \cite{karatzasShreve}*{Corollary 3.5.14}: 
    \begin{proposition}[Novikov's Condition]
        \label{proposition:Novikov}
        Let $K:[0,T] \to L^2\big( [0,T]; \lin(\bR^d, \bR^d) \big)$ be a Volterra kernel that satisfies Assumption \ref{assumption:VolterraK} and let $(\RKHS_t)_{t\in [0,T]}$ be a filtration of Hilbert spaces defined as in Equation \eqref{eq:prop:LND-filtration-H}. Let $b:[0,T] \times \cC_T^d \to \bR^d$ be progressively measurable suppose there exists a monotone increasing sequence $(t_n)_{n\in \bN}$ taking values in $[0,T]$ such that $t_n \uparrow \infty$ and for every $n\in \bN$
        \begin{equation}
            \label{eq:proposition:Novikov}
            \bE^{P^*}\Bigg[ \exp\bigg( \Big\| \Pi_{t_n, t_{n+1}}\Big[ \int_0^\cdot b\big(s, X[s] \big) ds\Big] \Big\|_{\RKHS_T}^2 \bigg) \Bigg]< \infty
        \end{equation}
        where $\Pi_{t_n, t_{n+1}} = \Pi_{t_{n+1}} - \Pi_{t_n}$ and $\Pi_t$ is the operator defined in Definition \ref{definition:RKHSprojection}. 

        Then
        \begin{equation*}
            t\mapsto \cZ_t\Big[ \int_0^\cdot b\big( s, X[s]\big) ds\Big] \quad \mbox{is an $\cF_t$-martingale. }
        \end{equation*}
    \end{proposition}

    \begin{proof}
        Firstly, on the interval $[0,t_1]$ Equation \eqref{eq:proposition:Novikov} implies that
        \begin{equation*}
            \bE^{P_{t_1}^*} \Bigg[ \exp\bigg( \Big\| \int_0^\cdot b\big( s, X[s] \big) ds \Big\|_{\RKHS_{t_1}}^2 \bigg) \Bigg] < \infty
        \end{equation*}
        so that we conclude that for any $t\in [0,t_1]$ the process
        \begin{equation*}
            t\mapsto \cZ_t\Big[ \int_0^\cdot b\big( s, X[s] \big) ds \Big] 
            \quad \mbox{is a martingale and}\quad
            \bE^{P^*}\bigg[ \cZ_t\Big[ \int_0^\cdot b\big( s, X[s] \big) ds \Big] \bigg] = 1. 
        \end{equation*}
        Next, for any choice of $n\in \bN$ we denote
        \begin{equation*}
            h_n(X) =  \Pi_{t_n, t_{n+1}} \bigg[ \int_0^\cdot b\big( s, X[s] \big) ds \bigg]
        \end{equation*}
        and we conclude that on the interval $[0,t_{n+1}]$ Equation \eqref{eq:proposition:Novikov} implies that 
        \begin{align*}
            &t\mapsto \cZ_{t}\big[ h_n(X) \big] \quad \mbox{is a martingale and}
            \\
            &\bE^{P^*}\Big[ \cZ_{t}\big[ h_n(X) \big] \Big| \cF_{s}^Z \Big] = 
            \left\{ 
            \begin{aligned}
                &1
                \quad &\mbox{for any $s\in [0,t_n]$,}&
                \\
                &\cZ_{s}\big[ h_n(X) \big] \quad &\mbox{for any $s\in [t_n, t_{n+1}]. $}&
            \end{aligned}
            \right.
        \end{align*}
        Next, using the orthogonality of the Hilbert spaces projections $\Pi_{t_n}$ and $\Pi_{t_n, t_{n+1}}$ we obtain that for any $h\in \RKHS_{T}$
        \begin{equation*}
            \frac{\big\| h \big\|_{\RKHS_{t_{n+1}}}^2}{2} = \frac{\big\| \Pi_{t_{n}}[h] \big\|_{\RKHS_{t_{n+1}}}^2}{2} + \frac{\big\| \Pi_{t_n,t_{n+1}}[h] \big\|_{\RKHS_{t_{n+1}}}^2}{2} 
        \end{equation*}
        so that we obtain that for $t\in [t_1, t_2]$
        \begin{align*}
            t\mapsto& \cZ_t\big[ h(X) \big] = \exp\Big( \delta\big( \Pi_t[h]\big) - \tfrac{ \big\| \Pi_t[h] \big\|_{\RKHS_T}^2}{2} \Big)
            \\
            &= \exp\Big( \delta\big( \Pi_{t_1\wedge t}[h]\big) - \tfrac{ \big\| \Pi_{t_1 \wedge t}[h] \big\|_{\RKHS_T}^2}{2} \Big) \cdot \exp\Big( \delta\big( (\Pi_{t_2\wedge t}-\Pi_{t_1\wedge t})[h]\big) - \tfrac{ \big\| (\Pi_{t_2\wedge t}-\Pi_{t_1\wedge t})[h] \big\|_{\RKHS_T}^2}{2} \Big)
            \\
            &= \cZ_t\big[ h_1(X) \big] \cdot \cZ_t\big[ h_2(X) \big]
        \end{align*}
        is a martingale. By induction on $n$ and using that $t_n \uparrow \infty$, we conclude that $\cZ_t[h]$ is a martingale over the whole interval $t\in [0,T]$. 
    \end{proof}

    \subsection{Examples of Gaussian processes}
    
    The following was first observed in \cite{Mishura2020Gaussian} and includes fractional Brownian motion:
    \begin{example}
        \label{example:Mishura}
        Motivated by the case $H>\tfrac{1}{2}$, let $K:[0,T] \to L^2\big( [0,T]; \bR\big)$ be a Volterra kernel of the form
		\begin{equation}
			\label{eq:example:Mishura-1}
			K(t, s) = a(s) \cdot \int_s^t b(u) \cdot c(u-s) du
		\end{equation}
		where $a, b, c:[0,T] \to \bR$ satisfy that
		\begin{enumerate}
			\item The functions $a\in L^p([0,T])$, $b\in L^q([0,T])$ and $c\in L^r([0,T])$ for $p \in [2, \infty]$, $q\in [1, \infty]$ and $r \in [1, \infty]$ such that
			\begin{equation*}
				\tfrac{1}{p} + \tfrac{1}{q} + \tfrac{1}{r} \leq \tfrac{3}{2}. 
			\end{equation*}
			\item The function $c\in L^1([0,T])$ forms a Sonine pair with $h\in L^1([0,T])$, that is 
			\begin{equation*}
				\int_0^t c(s) \cdot h(t-s) ds = 1 
				\quad
				\forall t\in (0, T]. 
			\end{equation*}
			\item The functions $a$ and $b$ are positive almost everywhere on $[0,T]$. 
			\item The functions $a^{-1} \in C^1([0,T])$, $d:=b^{-1} \in C^2([0,T])$ and either
			\begin{enumerate}
				\item We have that $d(0) = d'(0) = 0$. 
				\item The function $a^{-2} \cdot h \in C^1([0,T])$. 
			\end{enumerate}
		\end{enumerate}

        Then the Volterra kernel $K$ satisfies Assumption \ref{assumption:VolterraK} and the Volterra kernel $L:[0,T] \to \fWIC$ is equal to
        \begin{equation*}
            L(t, s) = \frac{h(t-s)}{a(t) b(s)} + \frac{1}{b(s)} \int_s^t \frac{a'(v) h(v-s)}{a(v)^2} dv. 
        \end{equation*}
        In particular
        \begin{align*}
            \scJ^*\big[ \1_{[0,t]} \big](s) =& K(t, s) = K(t, s) - K(s, s)
            = \int_s^T \frac{\partial K}{\partial u} (u, s) \1_{[0,t]}(u) du
        \end{align*}
        so that
        \begin{align}
            \label{eq:sonine1}
			\scJ^*\big[& L(t, \cdot) \big](s) = \int_s^t \frac{\partial K}{\partial u}(u, s) \cdot L(t, u) du
			\\
            \nonumber
			=& \scJ^*\bigg[ \1_{[0,t]} \cdot \frac{p(t) \cdot h(t-\cdot)}{b(\cdot)} \bigg](s) - \scJ^*\bigg[ \frac{1}{b(\cdot )} \int_{\cdot }^t p'(v) \cdot h(v - \cdot) ds \bigg](s)
			\\
            \nonumber
			=& \int_s^t a(s) \cdot b(u) \cdot c(u-s) \cdot\frac{p(t) \cdot h(t-u)}{b(u)} du
			- 
			\int_s^t a(s) \cdot c(u-s) \cdot \int_u^t p'(v) \cdot h(v-u) dv du
			\\
			\nonumber
            =& a(s) p(t) \int_s^t c(u-s) h(t-u) du
			-
			\int_s^t a(s) p'(v) \int_s^v c(u-s) h(v-u) du dv
			\\
            \nonumber
			=& a(s) \cdot p(t) \1_{[0,t]}(s) - a(s) \cdot \big[ p(t) - p(s) \big] \1_{[0,t]}(s) = \1_{[0,t]}(s). 
		\end{align}
    \end{example}

    \begin{example}
        \label{example:Mishura-2}
        Extending the ideas of \cite{Mishura2020Gaussian} and motivated by the case $H<\tfrac{1}{2}$, consider a Volterra kernel of the form
        \begin{equation*}
            K(t, s) = a(s) \bigg( c(t-s) b(t) - \int_s^t c(u-s) \frac{d}{du}\Big[ b(u) \Big] du \bigg)
        \end{equation*}
        Observe that
        \begin{equation*}
            \scJ^*\big[ \1_{[0,t]} \big](s) = K(t, s) = \int_0^T K(u, s) \delta_t(u) du = - \int_s^T K(u, s) \frac{\partial \1_{[0,t]}}{\partial u}(u) du
        \end{equation*}
        so that
        \begin{align}
            \label{eq:sonine2}
            \scJ^*\big[ L(t, \cdot) \big](s) = -\int_s^T K(u, s) \frac{\partial L}{\partial u} (t, u) du = \1_{[0,t]}(s)
        \end{align}
    \end{example}

    Equation \eqref{eq:sonine1} and \eqref{eq:sonine2} can equally be interpreted as the existence of a \emph{Sonine kernel} for the Volterra kernel, see \cite{Sonine1884Sur}.

%    \newpage
    \section{Locally interacting processes and the 2-MRF property}
    \label{section:2MRF}

    Our next goal it to study stochastic differential equations of the form
    \begin{equation}
        \label{eq:locally-interaction}
        dX_t^u = b_u\Big(t, X^u[t], X^{N_u}[t] \Big) dt + dZ_t^u, 
        \qquad 
        u \in V,
        \qquad 
        (X_0^u)_{u\in V} \sim \mu_0
    \end{equation}
    where $(V, E) \in \cG$ and recall $N_u=\{v \in V: \{u, v\} \in E\}$. We want to study the existence and uniqueness of such countably infinite collections of processes, along with their Markov Random Field properties. 
    
    Firstly, in Section \ref{subsection:WeakExist+Uniq} we prove a \emph{weak existence and uniqueness} result for such collections of stochastic differential equations driven by additive Gaussian processes. The key challenge here is that each equation is strongly correlated with its neighbours and the number of equations is taken to be countably infinite. This relies on Theorem \ref{theorem:ap:girsanov} which we established previously. 
    
    In Section \ref{subsection:MRF}, we prove the Markov Random Field property for collections of stochastic differential equations of the form \eqref{eq:locally-interaction}. Firstly, we consider the finite graph case and prove Theorem \ref{theorem:MRF-1} via a clique factorisation. When the graph is countably infinite, this approach does not work and instead we need to consider appropriate truncation and convergence arguments where we prove Theorem \ref{theorem:MRF-2*}. 

    \subsection{Weak existence and uniqueness}
    \label{subsection:WeakExist+Uniq}

    We commence by defining in what sense we establish our solution:
    \begin{definition}
        Let $(V, E) \in \cG$. We say that 
        \begin{equation*}
            \Big( (\Omega, \cF, \bP), \big( X_0^u, Z^u, b_u, X^u \big)_{u \in V} \Big)
        \end{equation*}
        is a weak solution to the stochastic differential equation \eqref{eq:locally-interaction} when $(\Omega, \cF, \bP)$ is a probability space and
        \begin{equation*}
            \big( X_0^u, Z^u, X^u[t] \big)_{u \in V} : \Omega \to \big( \bR^d \times \cC_{0, T}^d \times \cC_T^d\big)^{V}
        \end{equation*}
        is a random variable such that:
        \begin{enumerate}
            \item The random variable $(X_0^V)$ has distribution $\mu_0 \in \cP\big( (\bR^d)^V \big)$; 
            \item The expectation
            \begin{equation*}
                \sup_{u \in V} \bE\Big[ \big\| X^u \big\|_{\infty}^2 \Big] < \infty;
            \end{equation*}
            \item For every $u \in V$ the function $b_u:[0,T] \times \cC_T^d \times (\cC_T^d)^{N_u} \to \bR^d$ is progressively measurable and the random variable
            \begin{equation*}
                \omega \mapsto \bigg( \int_0^T \Big| b_u\big( t, X^{u}[t], X^{N_u}[t] \big) \Big| dt \bigg)_{u \in V} \in (\bR^d)^{V}
                \quad
                \bP\mbox{-almost surely; }
            \end{equation*}	
            \item The random variables
            \begin{align*}
                &(X^V) 
                \quad \mbox{and}\quad
                \Big( X_0^u + \int_0^\cdot b_u\big( t, X^u[t], X^{N_u}[t] \big) dt + Z_{\cdot}^u \Big)_{u \in V}
            \end{align*}
            are $\bP$-almost surely equal.
        \end{enumerate}
    \end{definition}

    We work under the following set of assumptions:
    \begin{assumption}
        \label{assumption:ExUn}
        Let $(V, E) \in \cG$, let $M \in L^1\big( [0,T]; \bR\big)$ and let $\mu_0 \in \cP( (\bR^d)^{V} \big)$. Let $(\gamma^u, b_u)_{u\in V}$ and suppose that:
        \begin{enumerate}[label=(X.\arabic*)]
            \item 
            \label{enum:assumption:ExUn-1}
            For every $u\in V$, let $\gamma^u \in \cP(\cC_{0,T}^d)$ be a Gaussian measure with abstract Wiener space $(\cE^u, \RKHS_T^u, \BSi^u )$ where $\cE^u \subseteq \cC_{0, d}^d$. Further, suppose that the filtration of reproducing kernel Hilbert spaces
            \begin{equation*}
                \Big( \RKHS_t^u, \big\langle \cdot, \cdot \big\rangle_{\RKHS_t} \Big)_{t\in [0,T]}
                \quad \mbox{is isomorphic to}\quad
                \Big( L_{t, T}^2(\bR^d), \big\langle \cdot, \cdot \big\rangle_t \Big)_{t\in [0,T]}. 
            \end{equation*}
            \item 
            \label{enum:assumption:ExUn-2}
            For every $u \in V$, the function $b_u: [0, T] \times \cC_T^d \times (\cC_T^d)^{N_u} \to \bR^d$ is progressively measurable and suppose for any $(X^u, X^{N_u}) \in \cC_T^d \times (\cC_T^d)^{N_u}$ that
            \begin{align}
                \label{eq:assumption:MRF-1-finite}
                &\bigg\| \int_0^\cdot b_u\big( s, X^u[s], X^{N_u}[s] \big) ds \bigg\|_{\RKHS_T}< \infty. 
            \end{align}
            Further, suppose for every $s, t\in[0,T]$ and $(X, X^{N_u}) \in (\bR^d \times \cE^u) \times (\bR^d \times \cE^v)^{N_u}$ that
            \begin{align}
                \nonumber
                \bigg\| \Pi_{s, t}\Big[& \int_0^\cdot b_u\big( s, X^u[s], X^{N_u}[s] \big) ds \Big] \bigg\|_{\RKHS_T}^2 
                \\
                \label{eq:assumption:MRF-1-linear}
                &\leq 
                \int_s^t M_r dr \cdot \Big( 1 + \big\| X^u \big\|_{\cE^u}^2 + \frac{1}{|N_u|} \sum_{v\in N_u} \big\| X^v \big\|_{\cE^v}^2 \Big) < \infty
            \end{align}
            where the operator $\Pi_t:\RKHS_T \to \RKHS_t$ is defined in Definition \ref{definition:RKHSprojection};
        \end{enumerate}
    \end{assumption}

    \begin{remark}
        Let us take a moment to compare Equation \eqref{eq:assumption:MRF-1-linear} and \eqref{eq:assumption:MRF-1-linear*}: thanks to Proposition \ref{prop:Existence_Q} we conclude that for every $(X^u, X^{N_u}) \in \cC_T^d \times (\cC_T^d)^{N_u}$, 
        \begin{equation*}
            \bigg\| \Pi_{s, t}\Big[ \int_0^\cdot b_u\big( r, X^u[r], X^{N_u}[r] \big) dr \Big] \bigg\|_{\RKHS_T}^2 = \int_s^t \Big| Q^{b_u}\big(r, X^u[r], X^{N_u}[u] \big) \Big|^2 dr
        \end{equation*}
        where $Q^{b_u}$ is defined in Equation \eqref{eq:def-Q}. 
        
        Therefore, if Equation \eqref{eq:assumption:MRF-1-linear} holds for some function $M \in L^1\big([0,T]; \bR\big)$ then the function $F$ from Equation \eqref{eq:assumption:MRF-1-linear*} is simply
        \begin{equation*}
            F_u\big( X^u, X^{N_u} \big) = \Big( 1 + \big\| X^u \big\|_{\cE^u}^2 + \frac{1}{|N_u|} \sum_{v\in N_u} \big\| X^v \big\|_{\cE^v}^2 \Big)
        \end{equation*}
        As we saw in Examples \ref{example:Q} and \ref{example:Q2}, the norm of the abstract Wiener space can by $\| \cdot \|_{\infty, T}$ when $H<\tfrac{1}{2}$, but when $H>\tfrac{1}{2}$ we need the larger norm $\| \cdot \|_{H-\varepsilon}$-H\"older norm. 
    \end{remark}
    
    \begin{theorem}
        \label{theorem:ExUn}
        Let $(V, E) \in \cG$, $M \in L^1\big([0,T]; \bR \big)$, $\mu_0 \in \cP( (\bR^d)^V \big)$ and $(\gamma^u, b_u)_{u\in V}$ satisfy Assumption \ref{assumption:ExUn}. Then there exists a unique in law solution to Equation \eqref{eq:locally-interaction}. 
    \end{theorem}

    \begin{proof}[Proof of Theorem \ref{theorem:ExUn} when $|V|< \infty$:]
        We use the techniques of \cite{karatzasShreve}*{Proposition 5.3.6}: \\ We rewrite $(\cC_T^d)^V = \cC_T^{d \times |V|}$ and denote $d' = d \times |V|$. 

        Since the collection of measures $(\gamma^u)_{u \in V}$ satisfies \ref{enum:assumption:ExUn-1}, we conclude from Theorem \ref{thm:LND-filtration} that the product measure $\prod_{u\in V} \gamma^u \in \cP\big( \cC_{0, T}^{d'} \big)$ is the law of a $d'$-dimensional Gaussian Volterra process that satisfies Assumption \ref{assumption:VolterraK}. 

        Using that $\cC_T^{d'} \equiv \bR^{d'} \times \cC_{0, T}^{d'}$, we can define the measure
        \begin{equation}
            \label{eq:measureP^*}
            P^* := \mu_0 \times \prod_{u\in V} \gamma^u \in \cP\big( \bR^{d'} \times \cC_{0, T}^{d'} \big) \equiv \cP\big( \cC_T^{d'} \big)
        \end{equation}
        Thanks to \ref{enum:assumption:ExUn-2}, we conclude that for every $u \in V$
        \begin{equation*}
            \bigg\| \int_0^{\cdot} b_u \big( s, X^u[s], X^{N_u}[s] \big) ds \bigg\|_{\RKHS_T} < \infty \quad P^*\mbox{-almost surely. }
        \end{equation*}
        Further, Equation \eqref{eq:assumption:MRF-1-linear} in particular means that we can choose a monotone increasing sequence $(t_n)_{n\in \bN}$ taking values in $[0,T]$ such that $t_n \uparrow \infty$ and for every $n\in \bN$
        \begin{align*}
            \bE^{P^*}\Bigg[ \exp&\bigg( \sum_{u\in V} \Big\| \Pi_{t_n, t_{n+1}}\Big[ \int_0^\cdot b_u\big(s, X^u[s], X^{N_u}[s] \big) ds\Big] \Big\|_{\RKHS_T}^2 \bigg) \Bigg]
            \\
            =&\bE^{P^*}\Bigg[ \exp\bigg( \int_{t_n}^{t_{n+1}} M_r dr \cdot \sum_{u\in V} \Big( 1+  \big\| X^u\big\|_{\cE^u}^2 + \tfrac{1}{|N_u|} \sum_{v\in N_u} \big\| X^{v}\big\|_{\cE^v}^2 \Big) \bigg) \Bigg] < \infty
        \end{align*}
        thanks to Ferniques Theorem. Therefore, we can apply Proposition \ref{proposition:Novikov} to conclude that
        \begin{equation*}
            t\mapsto \cZ_t\bigg[ \bigoplus_{u\in V} \int_0^\cdot b_u\Big( s, X^u[s], X^{N_u}[s] \Big) ds \bigg] \quad \mbox{is an $\cF_t$-martingale. }
        \end{equation*}
        We apply Theorem \ref{theorem:ap:girsanov} to conclude that there is a measure $P \in \cP\big( \cC_T^{d'} \big)$ defined by
        \begin{equation*}
            \frac{d P}{dP^*} \bigg|_{\cF_t} = \cZ_t\bigg[ \bigoplus_{u\in V} \int_0^\cdot b_u\Big( s, X^u[s], X^{N_u}[s] \Big) ds \bigg]
        \end{equation*}
        and under $P$ the law of the process
        \begin{equation*}
            (Z^u)_{u\in V} := \bigoplus_{u\in V} \bigg( X_\cdot^u - \int_0^\cdot b_u\Big( s, X^u[s], X^{N_u}[s] \Big) ds \bigg)
        \end{equation*}
        is the same as the law of the canonical process under $P^*$. Therefore, our weak solution to Equation \eqref{eq:locally-interaction} is the probability space $\big( \cC_T^{d'}, \cB(\cC_T^{d'}), P \big)$ paired with the collection $(X_0^u, Z^u, b_u, X^u)_{u \in V}$. 
    \end{proof}

    \subsubsection*{Measure theory recap}

    Suppose now that $|V|\nless \infty$: then the vector space $(\cC_T^d)^V$ is no longer a Banach space but a \emph{locally convex topological vector space} and we need to take additional care when we define our measure change. 

    The following well known Theorem which can be found in \cite{Tao2011Introduction} shall be used to extend Theorem \ref{theorem:ap:girsanov} to countably infinite collections of Gaussian processes:
    \begin{theorem}[Kolmogorov's measure extension Theorem]
        \label{theorem:KMET}
        Let $\big( (\cC_T^d)^\bN, \cB'\big( (\cC_T^d)^\bN \big) \big)$ be a measurable space for every $n\in \bN$ let $P_n \in \cP\big( (\cC_T^d)^{\times n} \big)$. We say that the sequence of measures $(P_n)_{n\in \bN}$ is \emph{consistent} if for every $n\in \bN$ and every $A \in \cB\big( (\cC_T^d)^{\times n} \big)$, 
        \begin{equation}
            \label{eq:consistent}
            P_{n+1}\big[ A \times \cC_T^d \big] = P_n\big[ A \big]. 
        \end{equation}
        Then there exists a unique probability measure $P \in \cP\big( (\cC_T^d)^{V} \big)$ such that for every $n \in \bN$ and $A \in \cB\big( (\cC_T^d)^{\times n} \big)$, the cylinder sets
        \begin{equation*}
            C_n(A):=\Big\{ (x_1,x_2, ...) \in (\cC_T^d)^{\bN}: (x_1, ..., x_n) \in (\cC_T^d)^{\times n} \Big\}
        \end{equation*}
        satisfy that
        \begin{equation}
            \label{eq:theorem:KMET}
            P\big[ C_n(A) \big] = P^n[A]. 
        \end{equation}
    \end{theorem}
    
    \begin{proof}[Proof of Theorem \ref{theorem:ExUn} when $|V|\nless \infty$:]
        As the set of vertices is countably infinite, let $\phi:V \to \bN$ be an enumeration and denote $V^n:=\big\{ v \in V: \phi[v]\leq n \big\}$. 

        For every $n\in \bN$, we denote the canonical projection $\pi_n: (\bR^d)^V \to (\bR^d)^{V^n}$ and $\mu^n:= \mu_0 \circ (\pi_n)^{-1} \in \cP\big( (\bR^d)^{\times n} \big)$. For every $u \in V$, we have the Gaussian measure $\gamma_u \in \cP(\cC_T^d)$ and we define the product measure 
        \begin{equation*}
            P^{*,n}:=\mu^n \times \prod_{i=1}^n \gamma^{\phi^{-1}[i]} \in \cP\Big( (\bR^d)^{V^n} \times (\cC_{0,T}^d)^{V^n} \Big) \equiv \cP\Big( (\cC_T^d)^{V^n} \Big)
        \end{equation*}
        Then the sequence of measures $P^{n, *}$ satisfies Equation \eqref{eq:consistent} and we conclude from Theorem \ref{theorem:KMET} that there exists a measure $P^* \in \cP\big( (\cC_T^d)^{V} \big)$ that satisfies Equation \eqref{eq:theorem:KMET}. 

        Thanks to \ref{enum:assumption:ExUn-2}, we conclude that for every $n \in 
        \bN$
        \begin{equation*}
            \sum_{u \in V^n} \bigg\| \int_0^{\cdot} b_u \big( s, X^u[s], X^{N_u}[s] \big) ds \bigg\|_{\RKHS_T} < \infty \quad P^{*, n}\mbox{-almost surely. }
        \end{equation*}
        Further, Equation \eqref{eq:assumption:MRF-1-linear} in particular means that we can choose a monotone increasing sequence $(t_m)_{m\in \bN}$ taking values in $[0,T]$ such that $t_m \uparrow \infty$ and for every $n\in \bN$
        \begin{equation*}
            \bE^{P^{*,n}}\Bigg[ \exp\bigg( \sum_{u \in V^n} \Big\| \Pi_{t_m, t_{m+1}}\Big[ \int_0^\cdot b_u\big(s, X^u[s], X^{N_u}[s] \big) ds\Big] \Big\|_{\RKHS_T}^2 \bigg) \Bigg]< \infty
        \end{equation*}
        as above. Therefore, we can apply Proposition \ref{proposition:Novikov} to conclude that
        \begin{equation*}
            t\mapsto \cZ_t\bigg[ \bigoplus_{u\in V^n} \int_0^\cdot b_u\Big( s, X^u[s], X^{N_u}[s] \Big) ds \bigg] \quad \mbox{is an $\cF_t$-martingale. }
        \end{equation*}
        We apply Theorem \ref{theorem:ap:girsanov} to conclude that there is a measure $P^n \in \cP\big( (\cC_T^{d})^{V^n} \big)$ defined by
        \begin{equation*}
            \frac{d P^n}{dP^{*,n}} \bigg|_{\cF_t} = \cZ_t\bigg[ \bigoplus_{u\in V^n} \int_0^\cdot b_u\Big( s, X^u[s], X^{N_u}[s] \Big) ds \bigg]
        \end{equation*}
        We can also verify that the sequence of measures $P^n$ satisfies Equation \eqref{eq:consistent} and we conclude from Theorem \ref{theorem:KMET} that there exists a measure $P \in \cP\big( (\cC_T^d)^{V} \big)$ that satisfies Equation \eqref{eq:theorem:KMET}. Under $P$, for any choice of $n\in \bN$ the law of the collection of processes
        \begin{equation*}
            \big( Z^u \big)_{u \in V^n}: = \bigoplus_{u \in V^n} \bigg( X_\cdot^u - \int_0^\cdot b_u\big( s, X^u[s], X^{N_u}[s] \big) ds \bigg)
        \end{equation*}
        is the same as the law of the $\big( X^u\big)_{u \in V^n}$ under $P^{*}$. Therefore, our weak solution to Equation \eqref{eq:locally-interaction} is
        \begin{equation*}
            \Big( (\cC_T^{d})^{V}, \cB'\big( (\cC_T^{d})^{V} \big), P \big)
            \quad \mbox{paired with the collection}\quad
            (X_0^u, Z^u, b_u, X^u)_{u \in V}.
        \end{equation*}
    \end{proof}

    \subsection{Markov Random Field property}
    \label{subsection:MRF}

    We start by recalling the Second-order Hammersley-Clifford Theorem (a proof of which can be found in \cite{lacker2020Locally}):
	\begin{proposition}[$2^{nd}$-order Hammersley-Clifford]
		\label{proposition:2HammersleyClifford}
		Let $G = (V, E)$ be a finite graph and let $(\cX, d)$ be a metric space. Suppose that $\nu \in \cP(\cX^{V})$ is absolutely continuous with respect to a product measure $\nu^* = \bP^{V}$ for some $\bP \in \cP(\cX)$. 
		\begin{enumerate}[label=(\roman*)]
			\item 
			\label{proposition:2HammersleyClifford-1}
			The measure $\nu$ is a 2-Markov Random Field. 
			\item 
			\label{proposition:2HammersleyClifford-2}
			The Radon–Nikodym derivative of the measure $\nu$ with respect to the measure $\nu^*$ factorises of the form
			\begin{equation}
		        \label{eq:proposition:2HammersleyClifford} 
				\frac{d \nu}{d\nu^*}(x) = \prod_{K \in \clq{G}{2}} f_K(x),  \quad x\in \cX^{V}
			\end{equation}
			for some measurable function $f_K: \cX^{V} \to \bR^+$ for any $K \in \clq{G}{2}$. 
		\end{enumerate}
		Then \ref{proposition:2HammersleyClifford-2} implies \ref{proposition:2HammersleyClifford-1}. Further, if $\tfrac{d \nu}{d\nu^*}(x)$ is strictly positive then \ref{proposition:2HammersleyClifford-1} implies \ref{proposition:2HammersleyClifford-2}. 
	\end{proposition}
    Curiously, Hammersley and Clifford never published their original work on the link between Gibbs random fields and Markov random fields as they considered it \emph{incomplete} given the - now known to be essential - requirement of positive definite probabilities, but subsequent publications were accepted along with new adaptions of the proof \cites{Grimmett1973Theorem, Besag1974Spatial}. 

    This actually points us to a vital detail of our technique: when the graph $(V, E)$ has a vertex set that is countably infinite, we are not able to write down a Radon-Nikodym derivative that is strictly positive so that we must use alternative truncation methods:
    \subsubsection*{Finite graphs}

    We start by considering a system of interacting equations with interactions described by a finite graph. 
    \begin{assumption}
		\label{assumption:MRF-1}
		Let $(V, E)$ be a finite graph, let $M \in L^1\big([0,T]; \bR \big)$ and let $\mu_0 \in \cP( (\bR^d)^{V} \big)$. Let $(\gamma^u, b_u)_{u\in V}$ satisfy Assumption \ref{assumption:ExUn} and additionally that:
        \begin{enumerate}[label=(X.\arabic*)]
            \setcounter{enumi}{2}
            \item For every $u \in V$, there exist probability measures $\lambda_u \in \cP(\bR^d)$ such that $\mu_0 \in \cP_2\big( (\bR^d)^{V} \big)$ is absolutely continuous with respect to the product measure
            \begin{equation}
                \label{eq:assumption:MRF-1-InitCond}
                \mu_0^* = \prod_{v\in V} \lambda_v 
                \quad \mbox{and the density}\quad
                \frac{d\mu_0}{d\mu_0^*}(x) = \prod_{K \in \clq{G}{2}} f_K(x^K), \quad x\in (\bR^d)^{V}. 
            \end{equation}
        \end{enumerate}
	\end{assumption}

    \begin{theorem}
		\label{theorem:MRF-1}
		Let $(V, E) \in \cG$, $M \in L^1\big([0,T]; \bR \big)$, $\mu_0 \in \cP( (\bR^d)^{V} \big)$ and $(\gamma^u, b_u)_{u\in V}$ satisfy Assumption \ref{assumption:MRF-1}. Then for any $t\in [0,T]$, the unique in law solution to Equation \eqref{eq:locally-interaction}
        \begin{equation*}
            \bP\circ \Big( X^V[t] \Big)^{-1} \in \cP\Big( (\cC_t^d)^V \Big)
        \end{equation*}
        is a 2-MRF. 
	\end{theorem}

	\begin{proof}
        Let $P^* \in \cP\big( (\cC_T^d)^V \big)$ be defined as in Equation \eqref{eq:measureP^*} and let $P$ be the unique law to the weak solution to Equation \eqref{eq:locally-interaction}. 
        
        For each $u \in V$, since $b_u$ is dependent only on $(X^u, X^{N_u}) \in \cC_T^d \times (\cC_T^d)^{N_u}$ and the set $\{u\} \cup N_u \in \clq{G}{2}$, we can conclude that for any $K\in \clq{G}{2}$ there exists measurable functions $\tilde{f}_K:(\cC_T^d)^K \to \bR$ such that
        \begin{equation}
            \frac{dP}{dP^*}\big( X^V[t] \big) = \prod_{K \in \clq{G}{2}} \tilde{f}\big( X^K \big)
        \end{equation}
        where
        \begin{equation*}
            \tilde{f}\big( X^K[t] \big) = 
            \begin{cases}
                \cZ_t\Big[ \int_0^\cdot b_u\big(s, X^K \big) ds \Big] 
                &\mbox{if }K = \{u\} \cup N_u
                \\
                0
                & \mbox{otherwise.}
            \end{cases}
        \end{equation*}

        Secondly, by Assumption \ref{assumption:MRF-1}
        \begin{equation*}
            \frac{d(\mu_0 \times \gamma^V)}{d(\mu_0^* \times \gamma^V)}(x_0^V) = \frac{d\mu_0}{d\mu_0^*}(x_0^V) = \prod_{K \in \clq{G}{2}} f_K\big( x_0^K \big)
        \end{equation*}
        so that
		\begin{equation*}
		    \frac{dP}{d(\mu_0^* \times \gamma^V)}\big( X^V[t] \big) = \frac{dP^*}{d(\mu_0^* \times \gamma^V)} \cdot \frac{dP}{dP^*}\big( X^V[t] \big) = \prod_{K \in \clq{G}{2}} \tilde{f}_K\big( X^K[t] \big) \cdot f_K\big( X_0^K \big). 
		\end{equation*}
		Thus for each $t\in [0,T]$ the Radon–Nikodym derivative $\tfrac{dP}{d(\mu_0^* \times \gamma^V)}\big( x[t] \big)$ has a 2-clique factorisation of the form Equation \eqref{eq:proposition:2HammersleyClifford} and courtesy of Proposition \ref{proposition:2HammersleyClifford} we conclude that $P$ is a 2-Markov Random Field over the graph $(V, E)$. 
	\end{proof}

    \subsubsection*{Infinite graphs}

    The argument used in the proof of Theorem \ref{theorem:MRF-1} relies on the underlying graph $(V, E)$ being finite, but we also want to consider the case where the graph is countably infinite:
    \begin{assumption}
		\label{assumption:MRF-2}
		Let $(V, E)$ be a countably infinite locally finite graph, let $M \in L^1\big( [0,T]; \bR \big)$ and let $\mu_0 \in \cP\big( (\bR^d)^V \big)$. Let $(\gamma^u, b_u)_{u\in V}$ satisfy Assumption \ref{assumption:ExUn} and additionally that: 
        \begin{enumerate}[label=(X.\arabic*')]
            \setcounter{enumi}{2}
            \item Suppose that the measure $\mu_0 \in \cP\big( (\bR^d)^{V} \big)$ is a 2-Markov Random Field and further that there exists a collection of measures $(\lambda_u)_{u\in V}$ such that for any finite set $A\subset V$, the marginal measure $\mu_0^A$ is equivalent to the product measure
            \begin{equation*}
                \mu_0^{*,A} = \prod_{v\in A} \lambda_v
            \end{equation*}
            and the initial law $\mu_0$ satisfies 
            \begin{align}
                \label{eq:assumption:MRF-2.int}
                \sup_{v \in V}\int_{(\bR^d)^V} |x^v|^2 d\mu_0(x^V) + \sup_{v\in V} \int_{\cE^u} \big\| X \big\|_{\cE^u}^2 d\gamma^u(X) < \infty;
            \end{align}
        \end{enumerate}
	\end{assumption}

	Building on Lemma \ref{le:specification-abstract} and \cite{lacker2020Locally}*{Proposition 4.4}, we obtain the following:
	\begin{proposition} 
        \label{pr:specification-finitegraph} 
		Let $\tilde{G}=(\tilde{V}, \tilde{E})$ and $\bar{G}=(\bar{V}, \bar{E} )$ be finite graphs and suppose that $V^* \subset \tilde{V} \cap \bar{V}$ satisfies \eqref{asmp:edgesets}. Let $A \subset V^*$ satisfying $\partial^2_{\tilde{G}} A \subset V^*$ and $\partial^2_{\bar{G}} A \subset V^*$.

		Suppose $\big( \tilde{\mu}_0, (\tilde{b}_u, \tilde{\gamma}^u)_{u\in \tilde{V}} \big)$ and $\big( \bar{\mu}_0, (\bar{b}_u, \bar{\gamma}^u)_{u\in \bar{V}} \big)$ both satisfy Assumption \ref{assumption:MRF-1} and let $\tilde{P} \in \cP\big( (\cC_T^d)^{\tilde{V}} \big)$ and $\bar{P} \in \cP\big( (\cC_T^d)^{\bar{V}} \big)$ be the corresponding unique laws of SDE \eqref{eq:locally-interaction}. Additionally, suppose that:
		\begin{align}
            \label{def:b-consistency} 
            \forall v \in A \cup \partial^2 A, \quad \tilde{b}_v \equiv \bar{b}_v
            \quad \mbox{and}\quad
            \forall v \in \tilde{V} \cap \bar{V}, \quad \tilde{\gamma}^v = \bar{\gamma}^v. 
        \end{align}
        Further, for each $v \in \tilde{V} \cup \bar{V}$ there exists $\lambda_v \in \cP(\bR^d)$ such that the product measure
        \begin{align}
            &\mu_0^* = \prod_{v \in \tilde{V} \cup \bar{V}} \lambda_v \in \cP\big( (\bR^d)^{\tilde{V} \cup \bar{V}} \big)
            \quad \mbox{such that}
            \\
            \label{mu0-factorization}
            &\frac{d\tilde{\mu}_0}{d\mu_0^{*,\tilde{V}}} \big( x^{\tilde{V}} \big) = \prod_{K \in \clq{\tilde{G}}{2}} \tilde{f}_K\big( x^K \big), 
            \quad \mbox{and}\quad 
            \frac{d\bar{\mu}_0}{d\mu_0^{*,\bar{V}}} \big( x^{\bar{V}} \big) = \prod_{K \in \clq{\bar{V}}{2} } \bar{f}_K \big( x^K \big), 
        \end{align}
        for some measurable functions $(\tilde{f}_K: (\bR^d)^K \mapsto \bR_+ )_{K \in \clq{\tilde{G}}{2}}$ and $(\bar{f}_K: (\bR^d)^K \mapsto \bR_+)_{K \in \clq{\bar{G}}{2}}$ and
        \begin{equation*}
            \forall K \in \mathcal{K}_A 
            \quad
            \tilde{f}_K \equiv \bar{f}_K
        \end{equation*}
        where $\cK_A$ are defined as in Equation \eqref{eq:cliques}. Then $\tilde{P}_t^A\big[ \cdot \big| \partial^2A \big] = \bar{P}_t^A\big[ \cdot \big|  \partial^2A \big]$ for each $t > 0$, both in the sense of $\tilde{P}_t^{\partial^2A}$-almost sure and $\bar{P}_t^{\partial^2A}$-almost sure. 
	\end{proposition}
 
	\begin{proof}
        Using that $(\cC_T^d)^{\tilde{V} \cup \bar{V}} \equiv (\bR^d)^{\tilde{V} \cup \bar{V}} \times (\cC_{0, T}^d)^{\tilde{V} \cup \bar{V}}$, we denote $P^* \in \cP\big( (\cC_T^d)^{\tilde{V} \cup \bar{V}} \big)$ to be the product measure 
        \begin{equation*}
            P^* = \mu_0 \times \bigg( \prod_{u\in \tilde{V}} \tilde{\gamma}^u \cdot \prod_{u \in \bar{V} \backslash \tilde{V}} \bar{\gamma}^u \bigg). 
        \end{equation*}
        Working on the canonical probability space 
        \begin{equation*}
            \Big( (\cC_T^d)^{\tilde{V} \cup \bar{V}}, \cB'\big( (\cC_T^d)^{\tilde{V} \cup \bar{V}}\big), P^* \Big)
        \end{equation*}
        and recalling Equation \eqref{eq:stochasticExpon}, for $t\in [0,T]$ we denote:
		\begin{align*}
            &\mbox{for each $u\in \tilde{V}$}
            \\
            &t \mapsto \tilde{\cZ}_t^u\big( X^V \big)
            =
            \exp\Bigg( \delta\bigg( \Pi_t\Big[ \int_0^{\cdot} \tilde{b}_u\big(s, X^u, X^{N_u} \big) ds \Big] \bigg) - \tfrac{1}{2} \bigg\| \int_0^{\cdot} \tilde{b}_u\big( s, X^u, X^{N_u} \big) ds \bigg\|_{\RKHS_t}^2 \Bigg);
            \\
            &\mbox{for each $u\in \bar{V}$}
            \\
            &t \mapsto \bar{\cZ}_t^u\big( X^V \big)
            =
            \exp\Bigg( \delta\bigg( \Pi_t \Big[ \int_0^{\cdot} \bar{b}_u\big( s, X^u, X^{N_u} \big) ds \Big] \bigg) - \tfrac{1}{2} \bigg\| \int_0^{\cdot} \bar{b}_u\big( s, X^u, X^{N_u} \big) ds \bigg\|_{\RKHS_t}^2 \Bigg). 
		\end{align*}
		Then by Theorem \ref{theorem:ap:girsanov} and Equation \eqref{mu0-factorization}, we have
		\begin{align*}
            \frac{d\tilde{P}_t}{dP^{*, \tilde{V}}_t} = \prod_{K \in \clq{\tilde{G}}{2}} \tilde{f}_K \big( X_0^K \big) \cdot \prod_{v \in \tilde{V}} \tilde{\cZ}_t^v
            \quad \mbox{and}\quad
			\frac{d\bar{P}_t}{dP_t^{*,\bar{V}}} = \prod_{K \in \clq{\bar{G}}{2}} \bar{f}_K\big( X_0^K \big) \cdot \prod_{v \in \bar{V}} \bar{\cZ}_t^v. 
		\end{align*}
		Finally, due to \eqref{asmp:edgesets} we have that $u \in A \cup \partial A$ implies $\tilde{N}_v = \bar{N}_u$ so that $\tilde{\cZ}_t^u = \bar{\cZ}_t^u$ for $u \in A \cup \partial A$. 
		
		Applying Lemma \ref{le:specification-abstract}, it follows that $\tilde{P}_t^A\big[ \cdot \big| \partial^2A \big] = \bar{P}_t^A\big[ \cdot \big| \partial^2A \big]$ holds in the sense of $P^*_t[\partial^2A]$-almost sure equality. Since both $P^H_t[\partial^2A]$ and $P^G_t[\partial^2A]$ are absolutely continuous with respect to $P^*_t[\partial^2A]$, the claim follows. 
	\end{proof}

    \begin{theorem}
		\label{theorem:MRF-2*}
		Let $(V, E)$, $M \in L^1\big([0,T]; \bR \big)$, $\mu_0 \in \cP( (\bR^d)^{V} \big)$ and $(\gamma^u, b_u)_{u\in V}$ satisfy Assumption \ref{assumption:MRF-2}. Then for any $t\in [0,T]$, the unique in law solution to Equation \eqref{eq:locally-interaction}
        \begin{equation*}
            \bP\circ \Big( X^V[t] \Big)^{-1} \in \cP\Big( (\cC_T^d)^{V} \Big)
        \end{equation*}
        is a 2-MRF. 
	\end{theorem}
    Notice that Theorem \ref{theorem:MRF-2} follows immediately from Theorem \ref{theorem:MRF-2*}. 
	\begin{proof}
		Let $(V, E)$ be a countably infinite locally finite connected graph, fix $\big( \mu_0, (b_u, \gamma^u)_{u\in V} \big)$ and let $X^V = (X^v)_{v \in V} \in (\cC_T^d)^{V}$ be the weak solution to Equation \eqref{eq:locally-interaction}. For $n \geq 4$, let $G_n = (V_n, E_n)$ be the sequence of finite graphs defined in Definition \ref{definition:graph_truncation} below.

		First note that by Assumption \ref{assumption:MRF-2} and Lemma \ref{le:MRFprojections}, the marginal $\mu_0^{V_n}$ is a 2MRF with respect to the graph $G_n$. Moreover, the Radon-Nikodym derivative 
        \begin{equation*}
            \frac{d\mu_0^{V_n}}{d\mu_0^{*,V_n}}
        \end{equation*}
        is strictly positive by Assumption (\ref{assumption:MRF-2}) and  Proposition \ref{proposition:2HammersleyClifford} demonstrates that $\mu_0^{V_n}$ admits a $2$-clique factorization with respect to the product measure $\mu_0^{*,V_n}$ for each $n$. Further, courtesy of Equation \eqref{driftn} we conclude that $\big( \mu_0^{V_n}, (b_u^n, \cH^u)_{u \in V_n} \big)$ satisfy Assumption \ref{assumption:MRF-1}. As $P^{n, V_n}$ is the law of the SDE \eqref{eq:truncatedSDE} on the finite graph $G_n$, by Theorem \ref{theorem:MRF-1} it is a 2-MRF.
		
        Now, fix two finite sets $A, B \subset V$ with $B$ disjoint of $A \cup \partial^2A$. Let $n_0$ denote the smallest integer greater than or equal to $4$ for which $A \cup \partial_G^2 A \cup B \subset V_{n_0-3}$, and let $n \geq n_0$. Then, Lemma \ref{le:MRFprojections} implies that $\mu_0^{V_n}$ and $\mu_0^{V_{n_0}}$ admit $2$-clique factorisation which are consistent in the sense that the corresponding measurable functions $f_K^{G_n}$ and $f_K^{G_{n_0}}$ agree for every $K \in \clq{G_{n_0}}{2}$ that intersects $A$ (equivalently, for every $K \in \clq{G}{2}$ that intersects $A$).  

        Following Equation \eqref{driftn}, for all $v \in A \cup \partial_G^2 A$ we have $b_v^{n} = b_v^{n_0} = b_v$ so we apply Proposition \ref{pr:specification-finitegraph}, with $\tilde{G} = G_{n}$, $\bar{G} = G_{n_0}$ and $V^* = V_{n_0-3}$, 
        \begin{equation*}
            \mbox{for } k \in \{n_0, ..., n\}
            \quad
            \mu_0^{G_k} = \mu_0^{V_k}, 
            \quad
            \mbox{for }  
            \big( b^{k}_v)_{v \in G_k} = \big( b^k_v \big)_{v \in G_k}, 
        \end{equation*}
        to deduce that for all $n \geq n_0$, $P^{n,A}_t\big[ \cdot \big| \partial^2 A \big] = P_t^{n_0, A}\big[ \cdot \big| \partial^2A \big]$. 
        
		Thus, given a bounded continuous function $f: (\cC_T^d)^A \to \bR$, there exists a measurable function $\varphi$ (that does not depend on $n$) such that for $n \geq n_0$
		\begin{equation}
			\label{phi-meas}
			\varphi\big( X^{\partial^2 A}[t] \big) = \bE^{P^n}\Big[ f(X^A[t]) \Big|  X^{\partial^2 A}[t] \Big] \quad P^n\mbox{-almost surely}.
		\end{equation}
		Now, fix additional bounded continuous functions $g: (\cC_T^d)^{\partial^2A} \to \bR$ and $h: (\cC_T^d)^B \to \bR$. For $t > 0$, taking the conditional expectation with respect to $X^{V_n \setminus A}[t]$ inside the expectation on the left-hand side below and using the 2-MRF property of $P^n$ and Equation \eqref{phi-meas}, we have
		\begin{align*}
			\bE^{P^n}\Big[ f\big(X^A[t] \big) g\big( X^{\partial^2A}[t] \big) h\big( X^B[t] \big) \Big] 
            &= 
            \bE^{P^n}\bigg[ \bE^{P^n}\Big[ f\big( X^A[t] \big) \Big| X^{\partial^2A} \Big] g\big( X^{\partial^2A}[t] \big) h\big( X^B[t] \big) \bigg]
            \\
            &= \bE^{P^n}\Big[ \varphi\big( X^{\partial^2A}[t] \big) g\big( X^{\partial^2A}[t] \big) h\big( X^B[t] \big) \Big]
		\end{align*}
		Applying Lemma \ref{le:infinitegraphlimit}, for the finite set $A' = A \cup \partial^2 A \cup B$ and for
        \begin{align*}
            \psi \big( y^{A'} \big) :=& f\big( y^A \big) g\big( y^{\partial^2A} \big) h\big( y^B \big)
            = 
            \varphi\big( y^{\partial^2A} \big) g\big( y^{\partial^2 A} \big) h\big( y^B \big),
        \end{align*}
		we pass to the limit $n \to \infty$ to get
		\begin{align*}
			\bE^{P}\Big[ f\big( X^A[t] \big) g\big( X^{\partial^2A}[t] \big) h\big( X^B[t]\big) \Big] 
            = 
            \bE^{P}\Big[ \varphi\big( X^{\partial^2A}[t] \big) g \big( X^{\partial^2A}[t] \big) h\big( X^B[t] \big) \Big].
		\end{align*}
		This at once shows both that
		\begin{equation*}
            \bE^P\Big[ f\big( X^A[t] \big) \Big| X^{\partial^2A}[t] \Big] 
            = 
            \varphi\big( X^{\partial^2A}[t] \big) 
            = 
            \bE^{P^n} \Big[ f\big( X^A[t] \big) \Big| X^{\partial^2A}[t] \Big],
		\end{equation*}
		for all bounded continuous $f$ and $n \ge n_0$, which proves that
        \begin{equation*}
            P_t^{A}\big[ \cdot \big| \partial^2A \big] 
            = 
            P_t^{n, A}\big[ \cdot \big| \partial^2A \big]. 
        \end{equation*}
        and also that $X^A[t]$ and $X^B[t]$ are conditionally independent given $X_{\partial^2A}[t]$ under $P$. 

	\end{proof} 

    %%% Section 4 %%%
    
    %%%%%%%%%%%%%%%%%%%%%%%%%%%%%%%%
	%\bibliography{-Bibliography.bib} 
	\bibliographystyle{plain}

    %%%%%%%%%%%%%%%%%%%%%%%%%%%%%%%%

    \appendix
    \section{Approximations of locally interacting equations}
    \label{section:appendix}
    
    \subsection{Graph truncation}
    \label{se:MRFs-projections}

    Firstly, we include a result from \cite{lacker2020Locally} which we shall use to establish that the 2-MRF property is retained under limits. 

    \begin{definition}
        \label{definition:graph_truncation}
        Let $G=(V, E)$ be a graph, fix $\oSlash \in V$ to be the root and let $n\geq 4$. We define
        \begin{align*}
            V_n:=& \{v\in V: d_G(v, \oSlash)\leq n\},
            \quad 
            U_n:=V_n \backslash V_{n-2}
            \quad \mbox{and}
            \\
            E_n:=& \big\{ (u, v) \in V_n \times V_n: (u, v) \in E \big\} \cup \big\{ (u, v) \in U_n \times U_n, u \neq v \big\}. 
        \end{align*}
    \end{definition}

    \begin{example}
        In order to illustrate the graph truncation described in Definition \ref{definition:graph_truncation}, consider the graph $V = \bZ$ and $E = \big\{ (i, i+1): i\in \bZ \big\}$. Choosing $n = 5$ and $\oSlash=0$, we obtain that
        \begin{align*}
            V_n =& \big\{ -5, -4, -3, -2, -1, 0, 1, 2, 3, 4, 5\} \quad U_n = \big\{ -5, -4, 4, 5\big\}, 
            \\
            E_n =& \big\{ (-5, -4), (-4, -3), (-3, -2), (-2, -1), (-1, 0), (0, 1), (1, 2), (2, 3), (3, 4), (4, 5) \big\} 
            \\
            &\cup \big\{ (-5, 4), (-5, 5), (-4, 4), (-4, 5) \big\}. 
        \end{align*}
        We can visualise this as the following
        
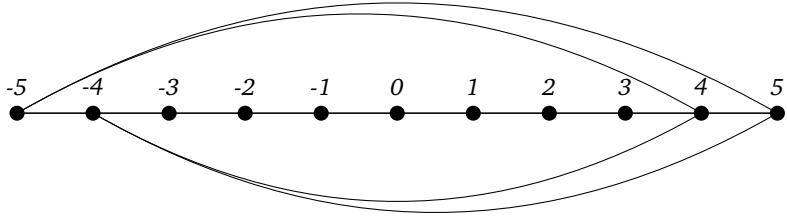
\begin{equation*}
			\begin{tikzpicture}
                \node[vertex, label=above:{\footnotesize 0}] at (0,0) {}; 
                \node[vertex, label=above:{\footnotesize 1}] at (1,0) {}; 
                \node[vertex, label=above:{\footnotesize 2}] at (2,0) {}; 
                \node[vertex, label=above:{\footnotesize 3}] at (3,0) {}; 
                \node[vertex, label=above:{\footnotesize 4}] at (4,0) {}; 
                \node[vertex, label=above:{\footnotesize 5}] at (5,0) {}; 
                \node[vertex, label=above:{\footnotesize -1}] at (-1,0) {}; 
                \node[vertex, label=above:{\footnotesize -2}] at (-2,0) {}; 
                \node[vertex, label=above:{\footnotesize -3}] at (-3,0) {}; 
                \node[vertex, label=above:{\footnotesize -4}] at (-4,0) {}; 
                \node[vertex, label=above:{\footnotesize -5}] at (-5,0) {}; 
                \draw[edge] (-5,0) -- (-4, 0);
                \draw[edge] (-4,0) -- (-3, 0);
                \draw[edge] (-3,0) -- (-2, 0);
                \draw[edge] (-2,0) -- (-1, 0);
                \draw[edge] (-1,0) -- (0, 0);
                \draw[edge] (0,0) -- (1, 0);
                \draw[edge] (1,0) -- (2, 0);
                \draw[edge] (2,0) -- (3, 0);
                \draw[edge] (3,0) -- (4, 0);
                \draw[edge] (4,0) -- (5, 0);
                \draw [black] (-5,0) to[bend left=30] (4,0);
                \draw [black] (-5,0) to[bend left=30] (5,0);
                \draw [black] (-4,0) to[bend right=30] (4,0);
                \draw [black] (-4,0) to[bend right=30] (5,0);
            \end{tikzpicture}
        \end{equation*}
    \end{example}
    
    For any $A\subset V_{n-3}$, it holds that $\partial_G^2 A = \partial_{G_n}^2 A$. Further, for any $A' \subset V_{n-2}$, $\partial_{G}^2 A' = \partial_{G_n}^2 A'$. Further, if $K\in \clq{G}{2}$ satisfies that $K \subset V_n$, then $K\in \clq{G_n}{2}$.
    
    \begin{lemma}[\cite{lacker2020Locally}*{Lemma 4.2}]
        \label{le:MRFprojections}
        Let $G=(V, E)$ be a locally finite graph and let $(\cX, d)$ be a metric space. Suppose that a $\cX^{V}$-valued random variable $Y^V=(Y^v)_{v\in V}$ is a 2-MRF with respect to $G$. Then $Y^{V_n}$ is a 2-MRF with respect to $G_n = (V_n, E_n)$ as defined in Definition \ref{definition:graph_truncation}.
        
        Secondly, suppose that $V$ is finite and the law $\nu \in \cP(\cX^V)$ admits the following 2-clique factorisation with respect to a product measure $\nu^* = \prod_{v\in V} \mu_v \in \cP(\cX^{V})$ for some sequence of measures $(\mu_v)_{v\in V}$, 
        \begin{equation*}
            \frac{d\nu}{d\nu^*}\big( x^V \big) = \prod_{K \in \clq{G}{2}} f_K\big( x^K \big)
        \end{equation*}
        for some measurable functions $f_K: \cX^{K} \to \bR_+$ and for $K \in \clq{G}{2}$. 
        
        Then for any $n>4$, the marginal measure $\nu^{V_n} \in \cP(\cX^{V_n})$ admits a 2-clique factorisation
        \begin{equation*}
            \frac{d\nu^{V_n}}{d\nu^{*,V_n}} \big( x^{V_n} \big) = \prod_{K \in \clq{G_n}{2} } f_K^0 \big( x^K \big),
        \end{equation*}
        for some measurable functions $f_K^0: \cX^{K} \to \bR_+$ and $K\in \clq{G}{2}$. 
        
        In particular, $f_K^0 \equiv f_K$ for any $K\in \clq{G}{2}$ such that $K\subset V_{n-3}$. 
    \end{lemma}
 
	Let us briefly recall a notation we introduced more carefully: For $\nu \in \cP(\cX^V)$ and $A,B \subset V$ we write $\nu^A\big[ \cdot \big| B \big]$ for the conditional law of the $A$-coordinates given the $B$-coordinates. 
	\begin{lemma}[\cite{lacker2020Locally}*{Lemma 4.3}]
        \label{le:specification-abstract}
		Let $\tilde{G}=(\tilde{V}, \tilde{E})$ and $\bar{G}=(\bar{V}, \bar{E})$ be finite graphs and assume that $V^* \subset \tilde{V} \cap \bar{V}$ satisfies
		\begin{equation}
            \label{asmp:edgesets}
			\tilde{E} \cap (V^* \times V^*) = \bar{E} \cap (V^* \times V^*). 
		\end{equation}
		Further, let $A \subset V^*$ and suppose that $\partial^2_{\tilde{G}} A \subset V^*$ and $\partial^2_{\bar{G}} A \subset V^*$. 
  
		Next, let $\tilde{\nu} \in \cP\big( \cX^{\tilde{V}} \big)$ and $\bar{\nu} \in \cP\big( \cX^{\bar{V}} \big)$ and let $\gmu^* = \prod_{v \in \tilde{V} \cup \bar{V}} \glambda_v \in \cP\big( \cX^{\tilde{V} \cup \bar{V}} \big)$ where for each $v \in V_G \cup V_H$ we have that $\glambda_v \in \cP(\cX)$. Suppose that the densities factorise as 
		\begin{align*}
	        \frac{d\tilde{\nu}}{d\gmu^{*,\tilde{V}}} \big( x^{\tilde{V}} \big) =  \prod_{K \in \clq{G}{2}} \tilde{f}_K \big( x^K \big)	
            \quad\mbox{and}\quad
            \frac{d\bar{\nu}}{d\gmu^{*,\bar{V}}} \big( x^{\bar{V}} \big) =  \prod_{K \in \clq{\bar{G}}{2}} \bar{f}_K\big( x^K \big), 
		\end{align*} 
		for measurable functions $(\bar{f}_K: \cX^K \mapsto \bR_+ )_{K \in \clq{\bar{G}}{2}}$ and $(\tilde{f}_K: \cX^K \mapsto \bR_+ )_{K \in \clq{G}{2}}$, and 
        \begin{align}
            \nonumber
            &\forall K \in \mathcal{K}_A 
            \quad
            f_K^H \equiv f_K^G
            \quad \mbox{where}
            \\
            \label{eq:cliques}
            &\cK_A:= \big\{ K \in \clq{\tilde{G}}{2} : K \cap A \neq \emptyset \big\} = \big\{ K \in \clq{\bar{G}}{2} : K \cap A \neq \emptyset \big\}
        \end{align}
		Then  $\tilde{\nu}^A \big[ \cdot \big| \partial^2 A] = \bar{\nu}^A\big[ \cdot \big| \partial^2 A \big]$, almost surely with respect to $\gmu^*[\partial^2 A]$. 
	\end{lemma}
    Given 2-MRFs on two  overlapping graphs, Lemma \ref{le:specification-abstract} provides conditions under which the conditional distributions of a subset in the intersection (given its complement) coincide for both 2-MRFs. 

    \subsection{Entropy estimates and convergence of truncation}
    \label{subs-tightness}

    For a collection of locally interacting equations of the form Equation \eqref{eq:locally-interaction} and $n\in \bN$, we define the sequence of collections of stochastic differential equations
    \begin{equation}
        \label{eq:truncatedSDE}
        dX_t^{n, u} = b_u^n\Big( t, X^{n, u}[t], \big( X^{n, v}[t] \big)_{v\in N_u} \Big) dt + dZ_t^u,
        \quad 
        u \in V,
        \quad
        (X_0^{n,u})_{u\in V} \sim \mu_0^{V}. 
    \end{equation}
    where for any $u\in V_n$, we define $b_u^n: [0,T] \times \cC_T^d \times (\cC_T^d)^{N_u} \to \bR^d$ to be the progressively measurable function such that 
    \begin{equation}
        \label{driftn}
            b_u^n\Big( t, x^u, (x^v)_{v\in N_u} \Big) = 
            \begin{cases}
                b_u\Big( t, x^u, (x^v)_{v\in N_u} \Big) \quad&\quad \mbox{if $u \in V_{n-2}$}
                \\
                0 \quad&\quad \mbox{if $u\in U_n$. }
            \end{cases}
    \end{equation}
    
	\begin{proposition}
        \label{pro:tightness}
		Let $(V, E)$ be a countably infinite locally finite graph, let $M\in L^1([0,T])$ and let $\big( (b_u, \gamma^u)_{u\in V}, \mu_0 \big)$ satisfy Assumption \ref{assumption:MRF-2}. Then for any $n\geq 4$, we denote $P^n$ to be the law of the solution to Equation \eqref{eq:truncatedSDE} and for every finite set $A \subset V$ and $t\in [0,T]$ there exists a constant $C_t>0$ such that 
		\begin{equation}
            \label{def:entropybound}
			\sup_{n} \bigg( \bH\Big[ P_t^{n,A} \Big| P_t^{*, A} \Big] \vee \bH\Big[ P_t^{*, A} \Big| P_t^{n, A} \Big] \bigg) \leq C_t \cdot |A|. 
		\end{equation} 
	\end{proposition}

	\begin{proof}
        Let $\big( (\cC_T^d)^V, \cB'\big( (\cC_T^d)^V\big), P^* \big)$ be the canonical measure space and denote 
        \begin{equation*}
            (X^u)_{u\in V}: (\cC_T^d)^V \to (\cC_T^d)^V
        \end{equation*}
        be the canonical process. Fix $n\in \bN$ such that $n\geq 4$.
        
        Consider the function $Z^V \in (\cC_{0, T}^d)^V$ of the canonical process
        \begin{equation}
            \label{eq:le:tightness-p1}
            Z_t^u = X_t^u - X_0^u - \int_0^t b_u^n\Big( s, X^u[s], X^{N_u}[s] \Big) ds,
            \quad 
            u \in V. 
        \end{equation}
        For any choice of $t\in [0,T]$ and $X^V \in (\cC_T^d)^V$ we denote $\cZ_t^n: [0,T] \times (\cC_T^d)^V \to \bR$ by
        \begin{align}
            t \mapsto \cZ_t^n\big( X^V):=& \prod_{u\in V} 
            \left\{
            \begin{aligned}
                &\cZ_t\Big[ \int_0^{\cdot} b_u^n\big( s, X^u[s], X^{N_u}[s]\big) ds \Big] \quad& u\in A \cap V_{n-2}
                \\
                &1 \quad& \mbox{otherwise.}
            \end{aligned}
            \right.
        \end{align}
        Thanks to Equation \eqref{eq:assumption:MRF-1-finite} we have for every $u\in V$ that
        \begin{equation*}
            \int_0^\cdot b_u\big( s, X^u[s], X^{N_u}[s] \big) ds \in \RKHS_T
        \end{equation*}
        so that by Proposition \ref{proposition:Martingale-Ust} the stochastic process $t\mapsto \cZ_t^n$ is an $\cF_t$-local martingale. Further, by Equation \eqref{eq:assumption:MRF-1-linear} we have for every for any $u \in V_{n-2}$  and $s, t\in [0,T]$ that
        \begin{align*}
            \bigg\| \Pi_{s, t}\Big[ \int_0^\cdot b_u\big( r, X^u[r], X^{N_u}[r] \big) dr \Big] \bigg\|_{\RKHS_T}^2 
            \leq& 
            \int_s^t M_r dr \cdot \bigg( 1 + \big\| X^u\big\|_{\cE^u}^2 + \tfrac{1}{|N_u|} \sum_{v\in N_u} \big\| X^v \big\|_{\cE^v}^2 \bigg)
            \\
            \leq& \int_s^t M_r dr \cdot \bigg( 1 + 2\sup_{v\in V_{n-1}} \big\| X^v\big\|_{\cE^v}^2  \bigg). 
        \end{align*}
        so that
        \begin{align*}
            \bE^{P^*}\Bigg[ \exp&\bigg( \Big\| \Pi_{s, t}\Big[ \int_0^\cdot b_u\big( r, X^u[r], X^{N_u}[r] \big) dr \Big] \Big\|_{\RKHS_T}^2\bigg) \Bigg]
            \\
            \leq& \bE^{P^*}\Bigg[ \exp\bigg( \int_s^t M_r dr \cdot \Big( 1 + 2\sup_{v\in V_{n}} \big\| X^v\big\|_{\cE^v}^2 \Big) \bigg) \Bigg]
            \\
            \leq& \exp\Big( \int_s^t M_r dr \Big) \cdot \bE^{P^*}\bigg[ \sup_{u\in V_{n}} \exp\Big( 2 \int_s^t M_r dr \cdot \big\| X^v\big\|_{\cE^v}^2 \Big) \bigg]
            \\
            \leq& \exp\Big( \int_s^t M_r dr \Big) \cdot \sum_{u\in V_{n}} \bE^{P^*}\bigg[ \exp\Big( 2 \int_s^t M_r dr \cdot \big\| X^v\big\|_{\cE^v}^2 \Big) \bigg]
        \end{align*}
        By choosing $|t-s|$ small enough so that
        \begin{equation}
            2\int_s^t M_r dr \leq \frac{1}{2} \inf_{u\in V_{n}} \frac{1}{\bE^{P^*} \big[ \| X^v \|_{\cE^v}^2 \big]}
        \end{equation}
        we obtain by Fernique's Theorem that
        \begin{equation*}
            \bE^{P^*}\Bigg[ \exp\bigg( \Big\| \Pi_{s, t}\Big[ \int_0^\cdot b_u\big( r, X^u[r], X^{N_u}[r] \big) dr \Big] \Big\|_{\RKHS_T}^2\bigg) \Bigg]< \infty. 
        \end{equation*}
        Therefore, we apply Proposition \ref{proposition:Novikov} to conclude that $\cZ_t^n$ is an $\cF_t^Z$-martingale and Theorem \ref{theorem:ap:girsanov} applies. Hence, we conclude that for every $t\in [0, T]$ the measure $P \in \cP\big( (\cC_t^d)^V \big)$ defined by
        \begin{equation*}
            \frac{dP}{dP^*} \Big|_{\cF_t^Z} \big( X^V \big) = \cZ_t^n\big( X^V[t] \big)
        \end{equation*}
        is a unique weak solution to Equation \eqref{eq:truncatedSDE}. What is more, the $V_n$-marginal $P^{n, V_n}$ is equal to the law of the system of interacting equations
        \begin{equation*}
            dX_t^u = b_u^n\Big( t, X^u[t], X^{N_u^n}[t] \Big) dt + dZ_t^u,
            \quad
            (X_0^u)_{u\in V_n} \sim \mu_0^{V_n},
            \quad
            u\in V_n
        \end{equation*}
        on the finite graph $(V_n, E_n)$. Given a normed subspace $\cE \subseteq \cC_T^d$, we also denote $\cE_t \subseteq \cC_t^d$ to be the canonical projection. Next, for $u \in V_{n-2}$ and $t\in [0,T]$, the expectation
        \begin{align*}
            \bE^{P^n}\Big[ \big\| X^u\big\|_{\cE_t^u}^2 \Big] 
            \leq& 
            3 \bE^{P^n}\bigg[ \big| X_0^u \big|^2 +  \big\| Z^u\|_{\cE^u}^2 + \Big\| \int_0^\cdot b_u^n\big( s, X^u[s], X^{N_u}[s] \big) ds \Big\|_{\RKHS_t}^2 \bigg] 
            \\
            \leq& 3\bigg( \bE^{\mu_0} \Big[ \big| x_0^u \big|^2 \Big] 
            + 
            \bE^{\gamma^u}\Big[ \big\| X^u \big\|_{\cE^u}^2 \Big] 
            + 
            \int_0^t M_s ds 
            \\
            &\quad + 2\int_0^t M_s \cdot \sup_{v \in B_1(u)} \bE^{P^n}\Big[ \big\| X^v \big\|_{\cE_s^v}^2 \Big] ds \bigg)
        \end{align*}
        and for $u \in U_{n}$
        \begin{equation*}
            \bE^{P^n}\Big[ \big\| X^u \big\|_{\cE^u}^2 \Big] \leq 2\bigg( \bE^{\mu_0} \Big[ \big| X_0^u \big|^2 \Big] + \bE^{P^n}\Big[ \big\| Z^u \big\|_{\cE^u}^2 \Big] \bigg). 
        \end{equation*}
        In particular, this means that for any choice of $t\in [0,T]$
        \begin{align*}
            \sup_{u\in V_n} \bE^{P^n}\Big[& \big\| X^u\big\|_{\cE_t^u}^2 \Big]
            \\
            \leq& 
            3\Bigg( \sup_{u\in V_n} \int_{(\bR^d)^{V_n}} \big| X_0^u \big|^2 d\mu_0\big( X_0^{V_n} \big) + \sup_{u\in V_n} \int_{(\cE^u)^{V_n}} \big\| X^u \big\|_{\cE^u}^2 d\gamma^{V_n}\big( X^{V_n} \big) + \int_0^t M_s ds  \Bigg)
            \\
            &\quad + 6 \int_0^t M_s \cdot \sup_{v \in V_n} \bE^{P^n}\Big[ \big\| X^v \big\|_{\cE_s^v}^2 \Big] ds 
        \end{align*}
		and an application of the Gr\"onwall inequality yields 
        \begin{align*}
            \sup_{u\in V_n}& \bE^{P^n}\Big[ \big\| X^u\big\|_{\infty, T}^2 \Big] 
            \\
            \leq& 
            3\Bigg( \sup_{u\in V_n} \int_{(\bR^d)^{V_n}} \big| X_0^u \big|^2 d\mu_0\big( X_0^{V_n} \big) + \sup_{u\in V_n} \int_{(\cE^u)^{V_n}} \big\| X^u \big\|_{\infty, T}^2 d\gamma^{V_n}\big( X^{V_n} \big) + \int_0^T M_s ds  \Bigg)
            \\
            &\cdot\exp\bigg( 6 \int_0^T M_s ds \bigg). 
        \end{align*}
        Finally, we take a supremum in $n\in \bN$ and apply Equation \eqref{eq:assumption:MRF-2.int} in order to conclude that
        \begin{equation*}
            \sup_{u \in V} \bE\Big[ \big\| X^u \big\|_{\cE^u}^2 \Big] < \infty. 
        \end{equation*}

        To conclude, we ecall Equation \eqref{eq:RelativeEntropy} and Proposition \ref{proposition:Martingale-Ust} we conclude that for any finite subset $A\subseteq V_n$ that
        \begin{align*}
            \bH\Big[ P_t^{n, A} \Big| P_t^{*, A} \Big] =& \int_{(\cC_T^d)^{V_n \cap A}} \log\bigg( \prod_{u\in V_n \cap A} \cZ_t\Big[ \int_0^\cdot b_u\big( s, X^u[s], X^{N_u}[s] \big) ds \Big] \bigg) dP\big( X^{V_n\cap A}[t] \big)
            \\
            =& \frac{1}{2} \sum_{u\in A \cap V_n} \bE^{P^n}\bigg[ \Big\| \int_0^\cdot b_u^n\big( s, X^u[s], X^{N_u}[s] \big) ds \Big\|_{\cH_t}^2 \bigg]
            \\
            \leq& \frac{1}{2} \sum_{u\in A \cap V_n} \int_0^t M_s ds \cdot \bigg( 1 + \bE\Big[ \big\| X^u\big\|_{\cE_t^u}^2 \Big] + \tfrac{1}{|N_u|} \sum_{v\in N_u} \bE\Big[ \big\| X^v\big\|_{\cE_t^v}^2 \Big] \bigg)
            \\
            \leq& |A| \cdot \int_0^t M_s ds \cdot \sup_{u \in V_n} \bE\Big[ \big\| X^u \big\|_{\cE_t^u}^2 \Big] < \infty
        \end{align*}
        which implies \eqref{def:entropybound}.         
	\end{proof}
 
	The next lemma  will be used to show both that the \emph{existence} of a weak solution to the infinite SDE system \eqref{eq:theorem:MRF-1} holds automatically and also that it arises as the limit of finite-graph systems. Recall that $P \in \cP(\cC^V)$ denotes the law of the solution of \eqref{eq:theorem:MRF-1}.
	
	\begin{lemma} 
        \label{le:infinitegraphlimit}
		Let $(V, E)$ be a locally finite graph, let $\big( \mu_0, (b_u, \cH^u)_{u\in V} \big)$ satisfy Assumption \ref{assumption:MRF-2}
		Then for any finite set $A \subset V$ and any bounded measurable function $\psi : (\cC_T^d)^{A} \to \bR$, we have 
		\begin{align*}
			\lim_{n\to\infty} \bE^{P^n}\Big[ \psi\big( X^{A} \big) \Big] = \bE\Big[ \psi\big( X^{A} \big) \Big].
		\end{align*}
        In particular, $P^n \to P$ weakly on $(\cC_T^d)^V$. 
	\end{lemma}
 
	\begin{proof}
		Firstly, we recall that for any $\nu \in \cP(\cX)$ and any $c>0$ the set
        \begin{equation*}
            \big\{ \mu \in \cP(\cX): \bH\big[ \mu\big| \nu \big]< c \big\}
        \end{equation*}
        is compact in the weak-$*$ topology (see \cite{DemboZeitouni2010}*{Lemma 6.2.16}) so that the entropy bound from Equation \eqref{def:entropybound} implies that for any $t\in [0,T]$ the sequence of measures $(P_t^{n,A})_{n \in \bN}$ is tight. Further, since this holds for every finite set $A \subseteq V$ and every $t \in [0,T]$ we deduce that the entire sequence $(P^n)_{n \in \bN}$ is tight in $(\cC_T^d)^V$.
		
        Note also that for sufficiently large $n$ it holds that under $P^n$ the processes
		\begin{equation}
            \label{pf:infinitegraphlimit1}
            X^v_s - X_0^v - \int_0^s b_v\big( r, X^v[r], X^{N_v}[r] \big) dr , \quad s \geq 0, \quad v \in V_{n-2}, 
		\end{equation}
		are independent Gaussian processes with distribution $\gamma^u$, due to the consistency condition for the $(b^n_u)_{u\in V}$ and the identity that $N_u^{G_n} = N_u^{G}$ for any $v \in V_{n-2}$.
		
		Now let $Q \in \cP(\cC^V)$ and suppose that there exists a subsequence $n_k$ such that $P^{n_k} \rightarrow Q$ weakly. In particular, this implies that for any finite set $A \subset V$ and any bounded measurable function $\psi: (\cC_T^d)^{A} \to \bR$,
		\begin{equation*}
            \lim_{k \to \infty} \bE^{P^{n_k}} \Big[ \psi \big( X^{A} \big) \Big] = \bE^{Q} \Big[ \psi \big( X^{A} \big) \Big],
		\end{equation*}
		
		We conclude that, under $Q$, the processes in \eqref{pf:infinitegraphlimit1} are independent Gaussian distributed processes with distribution $\gamma^u$ and hence $Q=P$.
	\end{proof} 

\end{document}